\documentclass{amsart}

\usepackage[margin=1.4in]{geometry}
\usepackage{amsmath,amssymb,amsthm}
\usepackage[retainorgcmds]{IEEEtrantools}
\usepackage{enumerate}
\usepackage{pifont,mathtools}
\usepackage{tikz-cd}
\usepackage[square,numbers]{natbib}

\newtheorem{theorem}{Theorem}[section]
\newtheorem{observation}[theorem]{Observation}
\newtheorem{lemma}[theorem]{Lemma}
\newtheorem{proposition}[theorem]{Proposition}
\newtheorem{corollary}[theorem]{Corollary}

\newtheorem{assumption}[theorem]{Assumption}

\numberwithin{equation}{section}

\theoremstyle{definition}
\newtheorem{definition}[theorem]{Definition}

\theoremstyle{remark}
\newtheorem{remark}[theorem]{Remark}

\newcommand\R{\mathbb{R}}
\newcommand\C{\mathbb{C}}
\newcommand\Q{\mathbb{Q}}
\newcommand\Z{\mathbb{Z}}
\newcommand\N{\mathbb{N}}

\newcommand\F{\mathbb{F}}

\DeclareMathOperator{\Lip}{Lip}

\DeclareMathOperator{\re}{Re}

\DeclareMathOperator{\Tr}{Tr}

\DeclareMathOperator{\id}{id}
\DeclareMathOperator{\TrP}{TrP}
\DeclareMathOperator{\NCP}{NCP}

\DeclareMathOperator{\Div}{Div}

\DeclareMathOperator{\Var}{Var}
\DeclareMathOperator{\II}{II}
\DeclareMathOperator{\app}{app}

\DeclarePairedDelimiter{\norm}{\lVert}{\rVert}
\DeclarePairedDelimiter{\ip}{\langle}{\rangle}

\begin{document}

\title[Convex Free Gibbs Laws]{Conditional Expectation, Entropy, and Transport for Convex Gibbs Laws in Free Probability}

\author{David Jekel}
\address{Department of Mathematics, UCLA, Los Angeles, CA 90095}
\email{davidjekel@math.ucla.edu}
\urladdr{www.math.ucla.edu/{$\sim$}davidjekel/}

\maketitle

\begin{abstract}
Let $(X_1,\dots,X_m)$ be self-adjoint non-commutative random variables distributed according to the free Gibbs law given by a sufficiently regular convex and semi-concave potential $V$, and let $(S_1,\dots,S_m)$ be a free semicircular family.  We show that conditional expectations and conditional non-microstates free entropy given $X_1$, \dots, $X_k$ arise as the large $N$ limit of the corresponding conditional expectations and entropy for the random matrix models associated to $V$.  Then by studying conditional transport of measure for the matrix models, we construct an isomorphism $\mathrm{W}^*(X_1,\dots,X_m) \to \mathrm{W}^*(S_1,\dots,S_m)$ which maps $\mathrm{W}^*(X_1,\dots,X_k)$ to $\mathrm{W}^*(S_1,\dots,S_k)$ for each $k = 1, \dots, m$, and which also witnesses the Talagrand inequality for the law of $(X_1,\dots,X_m)$ relative to the law of $(S_1,\dots,S_m)$.
\end{abstract}


\section{Introduction}

\subsection{Motivation}


Free probability initiated a fruitful exchange between random matrix theory and operator algebras.  In many situations, tuples of $N \times N$ random matrices $(X_1^{(N)}, \dots, X_m^{(N)})$ can be described in the large $N$ limit by non-commutative random variables $X_1$, \dots, $X_m$ which are operators in a tracial $\mathrm{W}^*$-algebra.  Conversely, many properties of non-commutative random variables (and the $\mathrm{W}^*$-algebras that they generate) are easier to understand when they can be simulated by finite-dimensional random matrix models.  For instance, Voiculescu used free entropy, defined in terms of matricial microstates, to prove the absence of Cartan subalgebras for free group $\mathrm{W}^*$-algebras $L(\mathbb{F}_n)$ \cite{VoiculescuFE3}; similar techniques were used to give sufficient conditions for a von Neumann algebra to be non-prime and non-Gamma (a convenient list of results and references can be found in \cite{CN2019}).  Further applications of random matrices to the properties of $\mathrm{C}^*$- and $\mathrm{W}^*$-algebras can be found for instance in \cite{HT2005} and \cite[\S 4]{GS2009}.

Free Gibbs laws are a prototypical example of the connection between random matrices and $\mathrm{W}^*$-algebras.  Free Gibbs laws describe the large $N$ behavior of self-adjoint tuples of random matrices $X^{(N)} = (X_1^{(N)}, \dots, X_m^{(N)})$ given by a probability measure $\mu^{(N)}$ of the form
\[
d\mu^{(N)}(x) = \frac{1}{Z^{(N)}} e^{-N^2 V^{(N)}(x)}\,dx,
\]
where $x \in M_N(\C)_{sa}^m$ is a self-adjoint tuple, $dx$ denotes Lebesgue measure, $V^{(N)}: M_N(\C)_{sa}^m \to \R$ is a function (known as a potential) chosen so that $e^{-N^2 V^{(N)}(x)}$ is integrable, and $Z^{(N)}$ is normalizing constant to make $\mu^{(N)}$ a probability measure.  Here $V^{(N)}(x)$ could be given by $V^{(N)}(x) = \tau_N(p(x_1,\dots,x_m))$, where $\tau_N = (1/N) \Tr$ and $p$ is a non-commutative polynomial; for instance, taking
\[
V^{(N)}(x) = \frac{1}{2} \sum_{j=1}^m \tau_N(x_j^2)
\]
produces the Gaussian unitary ensemble (GUE).  Under certain assumptions on $V$ (e.g.\ convexity and good asymptotic behavior as $N \to \infty$), there will be non-commutative random variables $X_1$, \dots, $X_m$ in a tracial $\mathrm{W}^*$-algebra $(\mathcal{M},\tau)$ such that
\[
\tau_N(p(X_1^{(N)}, \dots, X_m^{(N)})) \to \tau(p(X_1,\dots,X_m)) \text{ in probability for every non-commutative polynomial } p;
\]
see \cite[Theorems 3.3 and 3.4]{GMS2006}, \cite[Proposition 50 and Theorem 51]{DGS2016}, \cite[Theorem 4.1]{Jekel2018}.  The random matrix models satisfy the relation, derived from integration by parts, that
\[
E[\tau_N(D_{x_j} V^{(N)}(X^{(N)}) p(X^{(N)}))] = E[\tau_N \otimes \tau_N(\partial_{x_j}p(X^{(N)}))],
\]
where $D_{x_j} V$ is a normalized gradient with respect to the coordinates of $x_j$ and $\partial_{x_j}$ denotes the free difference quotient, and hence the non-commutative tuple $X = (X_1,\dots,X_m)$ satisfies
\[
\tau(D_{x_j} V(X) p(X)) = \tau \otimes \tau(\partial_{x_j}p(X));
\]
see \cite[\S 2.2 - 2.3]{GMS2006}.  The non-commutative law of a tuple $X$ satisfying such an equation is known as a \emph{free Gibbs law} for the potential $V$.

Given sufficient assumptions on $V^{(N)}$ (for instance, Assumption \ref{ass:convexRMM}), many of the classical quantities associated to $X^{(N)}$ will converge in the large $N$ limit to their free counterparts, besides obviously the convergence of the non-commutative moments $\tau_N(p(X^{(N)}))$.  For instance, the normalized classical entropy will converge to the microstates free entropy (see \cite[\S 2]{VoiculescuFE1}, \cite[Theorem 5.1]{GS2009}, \cite[\S 5.2]{Jekel2018}), and the normalized classical Fisher information will converge to the free Fisher information (see \cite[\S 5.3]{Jekel2018}).  The monotone transport maps of Guionnet and Shlyakhtenko are well-approximated by classical transport maps for the random matrix models \cite[Theorem 4.7]{GS2014}.  The solutions of classical SDE associated to the random matrix models approximate the solutions of free SDE; see for instance \cite{BCG2003}, \cite[\S 2]{GS2009}, \cite[\S 4]{Dabrowski2017}.

\subsection{Summary of Main Results}

This paper will further develop the connection between classical and free probability for convex free Gibbs laws, by studying conditional expectation (\S \ref{sec:conditionalexpectation}), conditional entropy and Fisher information (\S \ref{sec:entropy}), and conditional transport (\S \ref{sec:transport}).  This is an extension of our previous work \cite{Jekel2018}.

We consider a sequence of random matrix tuples $(X^{(N)},Y^{(N)}) = (X_1^{(N)}, \dots, X_m^{(N)}, Y_1^{(N)}, \dots, Y_n^{(N)})$ given by a uniformly convex and semi-concave sequence of potentials $V^{(N)}$ such that the normalized gradient $DV^{(N)}$ is asymptotically approximable by trace polynomials (a notion of good asymptotic behavior as $N \to \infty$ defined in \S \ref{subsec:AATP}).  Then the following results hold:
\begin{enumerate}[(1)]
	\item The non-commutative moments $\tau_N(p(X^{(N)},Y^{(N)}))$ converge in probability to $\tau(p(X,Y))$ for some tuple $(X,Y)$ of non-commutative random variables in a tracial $\mathrm{W}^*$-algebra.  See Theorem \ref{thm:freeGibbslaw}.
	\item The classical conditional expectation $E[f^{(N)}(X^{(N)},Y^{(N)})| Y^{(N)}]$ behaves asymptotically like the non-commutative conditional expectation $E_{\mathrm{W}^*(Y)}[f(X,Y)]$ where $f$ comes from an appropriate non-commutative function space and $f^{(N)}: M_N(\C)_{sa}^{m+n} \to M_N(\C)$ is a sequence of uniformly Lipschitz functions that ``behaves like $f$ in the large $N$ limit'' in the sense of \S \ref{sec:TP}.  See Theorem \ref{thm:conditionalexpectation}.
	\item The classical conditional entropy $N^{-2} h(X^{(N)} | Y^{(N)}) + (m/2) \log N$ converges to the conditional free entropy $\chi^*(X: \mathrm{W}^*(Y))$.  This is a similar to a conditional version of $\chi = \chi^*$.  See Theorem \ref{thm:convergenceofentropy}.
	\item There exists a function $f(X,Y)$ such that $(f(X,Y),Y) \sim (S,Y)$ in non-commutative law, where $S$ is a free semicircular $m$-tuple freely independent of $Y$, and this function also arises from functions $f^{(N)}$ such that $(f^{(N)}(X^{(N)},Y^{(N)}),Y^{(N)}) \sim (S^{(N)},Y^{(N)})$, where $S^{(N)}$ is an independent GUE $m$-tuple.  This is the conditional version of transport to the Gaussian/semicircular law.  See Theorems \ref{thm:transport3}.
	\item This transport map also witnesses the conditional entropy-cost inequality for the law of $X$ relative to semicircular conditioned on $Y$.  See Theorem \ref{thm:transport3}.
	\item This transport map furnishes an isomorphism $\mathrm{W}^*(X,Y) \cong \mathrm{W}^*(S,Y) \cong \mathrm{W}^*(S) * \mathrm{W}^*(Y)$, which shows that $\mathrm{W}^*(Y)$ is freely complemented in $\mathrm{W}^*(X,Y)$.
	\item Actually, a second application of transport shows that $\mathrm{W}^*(Y)$ is isomorphic to the $\mathrm{W}^*$-algebra generated by a semicircular $n$-tuple, or in other words $L(\mathbb{F}_n)$.  So altogether there is an isomorphism $\mathrm{W}^*(X,Y) \to L(\mathbb{F}_{m+n})$ that maps $\mathrm{W}^*(Y)$ to the canonical copy of $L(\mathbb{F}_n)$ inside $L(\mathbb{F}_{m+n})$.
\end{enumerate}
Furthermore, the results about transport can be iterated to produce a ``lower-triangular transport'' as shown in Theorem \ref{thm:transport4} and discussed further in \S \ref{subsec:transportintro}.  This is analogous to the classical results on triangular transport of measure such as \cite{BKM2005}.

In the rest of the introduction, we will review notation and then motivate and explain the main results in more detail.  In the course of the paper, it will become clear that not only are our main results proved all by the similar techniques, but in fact their statements and proofs are tightly interrelated.

\subsection{Notation and Background}

We will continue to use the same notation and background as in \cite{Jekel2018}.  The one major change is that we will write superscript $(N)$ rather than subscript $N$ for measures and functions defined on $N \times N$ matrices.  Moreover, we will use the original notation $\partial$ for Voiculescu's free difference quotient, even though \cite{Jekel2018} used $\mathcal{D}$.

We assume familiarity with the basic properties of tracial $\mathrm{W}^*$-algebras (or tracial von Neumann algebras); see for instance \cite{AP2017}.  In particular, a tracial $\mathrm{W}^*$-algebra is a finite $\mathrm{W}^*$-algebra $\mathcal{M}$ with a specified trace $\tau: \mathcal{M} \to \C$.  If $\mathcal{N} \subseteq \mathcal{M}$ is a $\mathrm{W}^*$-algebra, then there is a unique trace-preserving conditional expectation $E_{\mathcal{N}}: \mathcal{M} \to \mathcal{N}$.  If $x = (x_1,\dots,x_m)$ is a tuple of operators in $\mathcal{M}$, then we denote by $\mathrm{W}^*(x)$ the $\mathrm{W}^*$-subalgebra which they generate.

There is an inner product on $\mathcal{M}$ defined by $\ip{x,y}_2 = \tau(x^*y)$, and the completion of $\mathcal{M}$ in this inner product is a Hilbert space known as $L^2(\mathcal{M},\tau)$.  We denote the self-adjoint elements of $\mathcal{M}$ by $\mathcal{M}_{sa}$ and recall that if $x$ and $y$ are self-adjoint, then $\ip{x,y}_2$ is real.  If $x = (x_1,\dots,x_m)$ and $y = (y_1,\dots,y_m)$ are tuples, we denote $\ip{x,y}_2 = \sum_{j=1}^m \ip{x_j,y_j}_2$.  We define $\norm{x}_\infty = \max_j \norm{x_j}_\infty$, that is, the maximum of the operator norms of $x_j$.

We denote by $\NCP_m = \C\ip{X_1,\dots,X_m}$ the $*$-algebra of non-commutative polynomials in $m$ self-adjoint variables.  A \emph{non-commutative law} is a linear map $\lambda: \C\ip{X_1,\dots,X_m} \to \C$ satisfying
\begin{enumerate}[(A)]
	\item $\lambda(1) = 1$.
	\item $\lambda(p^*p) \geq 0$ for all $p \in \NCP_m$.
	\item $\lambda(pq) = \lambda(qp)$ for all $p, q \in \NCP_m$.
	\item $|\lambda(X_{i_1} \dots X_{i_k})| \leq R^k$ for some constant $R$.
\end{enumerate}
The set of non-commutative laws that satisfy (D) for a fixed value of $R$ is denoted $\Sigma_{m,R}$, and it is equipped with the topology of pointwise convergence on $\NCP_m$.  Likewise, the space of all laws, equipped with the topology of pointwise convergence, will be denoted by $\Sigma_m$.

If $x = (x_1,\dots,x_m)$ is a tuple of self-adjoint elements of $(\mathcal{M},\tau)$, then we may define a non-commutative law $\lambda_x$ by $\lambda_x(p) = \tau(p(x))$.  Conversely, every non-commutative law can be realized in this way through the GNS construction.  In particular, a free Gibbs law can be realized by a tuple $(x_1,\dots,x_m)$ of self-adjoint operators, and thus the free Gibbs law has a corresponding $\mathrm{W}^*$-algebra $\mathrm{W}^*(x)$, that is unique up to isomorphism.

We always consider $M_N(\C)$ as a tracial $\mathrm{W}^*$-algebra with the normalized trace $\tau_N = (1/N) \Tr$, and in particular, we use the notation $\norm{x}_2$, $\norm{x}_\infty$, and $\lambda_x$ as defined above when $x$ is an $m$-tuple of matrices.  The notation $\norm{\cdot}_2$ and $\norm{\cdot}_\infty$ will never be used for the $L^2$ or $L^\infty$ norms of \emph{functions} on matrices, but if we write an $L^p$ norm it will be expressed $\norm{\cdot}_{L^p}$.

For a smooth function $u: M_N(\C)_{sa}^m \to \R$, we denote by $Du$ and $Hu$ the gradient and Hessian with respect to the normalized inner product $\ip{\cdot,\cdot}_2$.  In other words, $Du(x_0)$ is the vector in $M_N(\C)_{sa}^m$ and $Hu(x_0)$ is the $\R$-linear transformation of $M_N(\C)_{sa}^m$ satisfying
\[
u(x) = u(x_0) + \ip{Du(x_0), x - x_0}_2 + \frac{1}{2} \ip{Hu(x_0)(x - x_0), x - x_0}_2 + o(\norm*{x - x_0}_2^2).
\]
For functions $f: M_N(\C)_{sa}^m \to \R$ or $M_N(\C)_{sa}^m \to M_N(\C)$, we denote $\norm{f}_{\Lip}$ the Lipschitz (semi)norm with respect to using $\norm{\cdot}_2$ on $M_N(\C)_{sa}^m$ and $M_N(\C)$.

Note that $M_N(\C)_{sa}^m$ can also be equipped with the real inner product $\ip{x,y}_{\Tr} = \sum_{j=1}^m \Tr(x_jy_j) = N \ip{x,y}_2$.  Being a real inner-product space, $M_N(\C)_{sa}^m$ may be identified with $\R^{mN^2}$ by choosing an orthonormal basis in $\ip{\cdot,\cdot}_{\Tr}$.  Lebesgue measure on $M_N(\C)_{sa}^m$ should be understood with respect to this identification.  Moreover, the gradient $\nabla$, Jacobian matrix $J$, divergence $\Div$, and Laplacian $\Delta$ for functions on $M_N(\C)_{sa}^m$ should also be understood with respect to this identification.  Beware that this is \emph{not} equivalent to using entrywise coordinates for $M_N(\C)_{sa}^m$ since the off-diagonal entries are complex and conjugate-symmetric, while the diagonal entries are real, and that the normalized gradient above satisfies $Df = N \nabla f$.  For further discussion see \cite[\S 2.1]{Jekel2018}.

\subsection{Main Results on Conditional Expectation}

Consider a tuple
\[
(X^{(N)},Y^{(N)}) = (X_1^{(N)},\dots,X_m^{(N)},Y_1^{(N)},\dots,Y_n^{(N)})
\]
of random self-adjoint matrices given by a probability density $(1/Z^{(N)}) e^{-N^2 V^{(N)}(x,y)}\,dx\,dy$.  We assume that $V^{(N)}$ is uniformly convex and semi-concave and that the normalized gradient $DV^{(N)}$ is asymptotically approximable by trace polynomials (a certain notion of good asymptotic behavior as $N \to \infty$, explained below).  The precise hypotheses are listed in Assumption \ref{ass:convexRMM}.  We showed in \cite[Theorem 4.1]{Jekel2018} that in this case, there exists an $(m+n)$-tuple $(X,Y)$ of non-commutative random variables such that $\tau_N(p(X,Y)) \to \tau(p(X,Y))$ in probability.

Our first main result (Theorem \ref{thm:conditionalexpectation}) says roughly that the classical conditional expectation given $Y^{(N)}$ well approximates the $\mathrm{W}^*$-algebraic conditional expectation $E_{\mathrm{W}^*(Y)}: \mathrm{W}^*(X,Y) \to \mathrm{W}^*(Y)$.  This is motivated in general by the importance of conditional expectation in free probability, e.g.\ its relationship to free independence with amalgamation and to free score functions.  See \cite[\S 4]{BCG2003} for a study of the large $N$ limits of conditional expectations related to matrix SDE.  The relationship between classical and free conditional expectation also has implications for the study of relative matricial microstate spaces, such as the ``external averaging property'' introduced in the upcoming work with Hayes, Nelson, and Sinclair \cite{HJNS2019}.

Applications of conditional expectation within this paper include our results on free Fisher information and entropy (see Theorem \ref{thm:convergenceofentropy} and Remark \ref{rem:simplifiedentropyproof}), as well as our proof that Assumption \ref{ass:convexRMM} is preserved under marginals (see Proposition \ref{prop:marginals}).

The statement and proof of Theorem \ref{thm:conditionalexpectation} rely on a notion of asymptotic approximation for functions on $M_N(\C)_{sa}^m$ explained in \S \ref{sec:TP}.  We define a class of non-commutative functions $\overline{\TrP}_m^1$ as a certain Fr{\'e}chet space completion of trace polynomials, such that if $f \in \overline{\TrP}_m^1$ and $x_1$, \dots, $x_m$ are self-adjoint elements in an $\mathcal{R}^\omega$-embeddable tracial $\mathrm{W}^*$-algebra $(\mathcal{M},\tau)$, then $f(x_1,\dots,x_m)$ is a well-defined element of $L^2(\mathcal{M})$.  In particular, every $f \in \overline{\TrP}_m^1$ can be evaluated on a tuple of self-adjoint matrices.  Now if $f^{(N)}: M_N(\C)_{sa}^m \to M_N(\C)$, we say that $f^{(N)} \rightsquigarrow f$ if for every $R > 0$,
\[
\lim_{N \to \infty} \sup_{\substack{x \in M_N(\C)_{sa}^m \\ \norm{x}_\infty \leq R}} \norm*{f^{(N)}(x) - f(x)}_2 = 0,
\]
Moreover, if such an $f$ exists, then we say that $f^{(N)}$ is \emph{asymptotically approximable by trace polynomials}.

Consider the random matrices $(X^{(N)},Y^{(N)})$ and non-commutative random variables $(X,Y)$ as above, and suppose that $f^{(N)}: M_N(\C)_{sa}^{m+n} \to M_N(\C)$ is uniformly Lipschitz in $\norm{\cdot}_2$ and that $f^{(N)} \rightsquigarrow f \in \overline{\TrP}_{m+n}^1$.  Then we show that $E[f^{(N)}(X^{(N)},Y^{(N)}) | Y^{(N)}]$ is given by a function $g^{(N)}(Y^{(N)})$ such that $g^{(N)} \rightsquigarrow g \in \overline{\TrP}_n^1$, and moreover $E_{\mathrm{W}^*(Y)}[f(X,Y)] = g(Y)$.

A curious feature of this result is that the function $g$ is defined for all self-adjoint $n$-tuples of non-commutative random variables, not only for the specific $n$-tuple $Y$ that we are concerned with.  Similarly, the claim that $g^{(N)} \rightsquigarrow g$ describes the asymptotic behavior of $g^{(N)}(y)$ for all $y \in M_N(\C)_{sa}^n$, even though the distribution of the random matrix $Y^{(N)}$ is highly concentrated as $N \to \infty$ on much smaller sets, namely the ``matricial microstate spaces'' consisting of tuples $y \in M_N(\C)_{sa}^n$ with non-commutative moments close to those of $Y$.  Thus, the statement we prove about the functions $g^{(N)}$ is stronger than an asymptotic result about $L^2$ approximation such as \cite[Theorem 4.7]{GS2014}.

\subsection{Main Results on Entropy}

Voiculescu defined two types of free entropy (see \cite{VoiculescuFE2}, \cite{VoiculescuFE5}, \cite{Voiculescu2002}).  The first, called $\chi(X)$, is based on measuring the size of matricial microstate spaces, which is closely related to the classical entropy of the random matrix models (see \cite[\S 5.2]{Jekel2018}).  The second, called $\chi^*(X)$, is defined in terms of free Fisher information, which is based on classical Fisher information.  Either one should heuristically be the large $N$ limit of the classical entropy of random matrix models, but there were many technical obstacles to proving this.  The inequality $\chi \leq \chi^*$ is known in general thanks to \cite{BCG2003}.  However, even for non-commutative laws as well-behaved and explicit as free Gibbs states given by convex potentials, the equality of $\chi$ and $\chi^*$ when $m > 1$ was not proved until Dabrowski's paper \cite{Dabrowski2017}, and the problem is still open for non-convex Gibbs states.

Our previous work \cite{Jekel2018} gave a proof of this equality in the convex case based on the asymptotic analysis of functions and PDE related to the random matrix models.  Here we will use similar techniques for the conditional setting.  We will show (Theorem \ref{thm:convergenceofentropy}) that for a random tuple of matrices $(X^{(N)},Y^{(N)})$ given by a convex potential as above, the classical conditional entropy $N^{-2} h(X^{(N)} | Y^{(N)}) + (m/2) \log N$ converges to the conditional free entropy $\chi^*(X: \mathrm{W}^*(Y))$.  Actually, the proof here is shorter than those of \cite{Dabrowski2017} and \cite{Jekel2018} (see Remark \ref{rem:simplifiedentropyproof}), even considering the results we used from \cite{Jekel2018}.

We focus here only on the non-microstates entropy (defined using Fisher information).  It is not yet resolved in the literature what the correct definition of conditional microstates free entropy should be.  In light of \cite[\S 5.2]{Jekel2018}, the conditional classical entropy for the random matrix models seems to be a reasonable substitute for microstates entropy, and in the convex setting we expect this to agree with any plausible definition of conditional microstates entropy due to the exponential concentration of measure.

\subsection{Main Results on Transport} \label{subsec:transportintro}

A \emph{transport map} from a probability measure $\mu$ and to another probability measure $\nu$ is a function $f$ such that $f_* \mu = \nu$.  In probabilistic language, if $X \sim \mu$ and $Y \sim \nu$ are random variables, then $f_* \mu = \nu$ means that $f(X) \sim Y$ in distribution.  The theory of transport (and in particular optimal transport) has numerous and significant applications in the classical setting.  
For instance, if we have a function $f$ such that $f(X) \sim Y$ and we can numerically simulate the random variable $X$, then we can also simulate $Y$.

In the non-commutative world, transport is even more significant.  As remarked in \cite[\S 1.1]{GS2014}, there is no known analogue of a probability density in free probability.  However, the existence of transport maps that would express our given random variables as functions of a free semicircular family (for instance) would serve a similar purpose to a density, namely to provide a fairly explicit and analytically tractable model for a large class of non-commutative laws.

Moreover, in contrast to the classical setting, the very existence of transport maps is a nontrivial condition.  Being able to express a non-commutative tuple $Y$ as a function of another non-commutative tuple $X$ implies that $\mathrm{W}^*(Y)$ embeds into $\mathrm{W}^*(X)$, and having a transport map in the other direction as well implies that $\mathrm{W}^*(Y) \cong \mathrm{W}^*(X)$.  In the classical setting, any two diffuse (non-atomic) standard Borel probability spaces are isomorphic.  On the other hand, there are many non-isomorphic diffuse tracial $\mathrm{W}^*$-algebras, even after restricting our attention to factors (those which cannot be decomposed as direct sums); see \cite{McDuff1969}.  Moreover, Ozawa \cite{Ozawa2004} showed that there is no separable tracial factor that contains an isomorphic copy of all the others.  Thus, there are many instances where it is not even possible to transport one given non-commutative law to another.

The papers \cite{GS2009} and \cite{DGS2016} showed the existence of monotone transport maps between certain free Gibbs laws given by convex potentials and the law of a free semicircular family, and thus concluded that each of the corresponding $\mathrm{W}^*$-algebras was isomorphic to a free group factor $L(\mathbb{F}_n)$.  In particular, this result applies to the $q$-Gaussian variables for sufficiently small $q$.  These transport techniques have been extended to type III von Neumann algebras \cite{Nelson2015a}, to planar algebras \cite{Nelson2015b}, and to interpolated free group factors \cite{HN2018}.  We will focus on ``conditional transport'' in the tracial setting.

Our first main result about transport is contained in Theorems \ref{thm:transport1} and \ref{thm:transport2}.  Let $(X^{(N)},Y^{(N)})$ be an $(m+n)$-tuple of random matrices arising from a sequence of convex potentials satisfying Assumption \ref{ass:convexRMM}.  Let $(X,Y)$ be an $(m+n)$-tuple of non-commutative self-adjoint variables realizing the limiting free Gibbs law.  Then we construct functions $F^{(N)}: M_N(\C)_{sa}^{m+n} \to M_N(\C)_{sa}^m$ such that $(F^{(N)}(X^{(N)}, Y^{(N)}), Y^{(N)}) \sim (S^{(N)}, Y^{(N)})$ in distribution, where $S^{(N)}$ is a GUE $m$-tuple independent of $Y^{(N)}$.  We think of this as a conditional transport, which transports the law of $X^{(N)}$ to the law of $S^{(N)}$ \emph{conditioned on $Y^{(N)}$}.

Moreover, we show that the transport maps satisfy $F^{(N)} \rightsquigarrow F \in (\overline{\TrP}_{m+n}^1)_{sa}^m$.  In the large $N$ limit, we obtain $(F(X,Y), Y) \sim (S,Y)$ in non-commutative law, where $S$ is a free semicircular $m$-tuple freely independent of $Y$.  In particular, this means that $\mathrm{W}^*(X,Y) \cong \mathrm{W}^*(S,Y) = \mathrm{W}^*(S) * \mathrm{W}^*(Y)$ (where $*$ denotes free product).  In other words, $\mathrm{W}^*(Y)$ is freely complemented in $\mathrm{W}^*(X,Y)$.

By iterating this result, we can show that if $X = (X_1,\dots,X_m)$ is a tuple of non-commutative random variables given by a convex free Gibbs state as above, then there is an isomorphism $\mathrm{W}^*(X) \to \mathrm{W}^*(S)$ such that $\mathrm{W}^*(X_1,\dots,X_k)$ is mapped onto $\mathrm{W}^*(S_1,\dots,S_k)$ for each $k = 1$, \dots, $m$.  In other words, there is a ``lower-triangular transport.''  See Theorem \ref{thm:transport4}.  This is a (partial) free analogue of \cite[Corollary 3.10]{BKM2005}.

This result implies in particular that $\mathrm{W}^*(X_1)$ is a maximal abelian subalgebra and in fact maximal amenable (since the subalgebra $\mathrm{W}^*(S_1)$ is known to be maximal amenable thanks to Popa \cite{Popa1983}), and the same holds for each $\mathrm{W}^*(X_j)$ by symmetry.  For context on maximal amenable subalgebras, see for instance \cite{Popa1983} \cite{BC2015} \cite{BH2018}.  More generally, \emph{any von Neumann algebraic properties} of the sequence of inclusions $\mathrm{W}^*(X_1) \subseteq \mathrm{W}^*(X_1,X_2) \subseteq \dots \subseteq \mathrm{W}^*(X_1,\dots,X_m)$ are the same as for the case of free semicirculars, that is, for the standard inclusions $L(\Z) \subseteq L(\F_2) \subseteq \dots \subseteq L(\F_m)$.

Denote by $F$ the transport map from the law of $X$ to the law of $S$ in our construction, so that $F(X) \sim S$.  We can also arrange that $F$ witnesses the Talagrand entropy-cost inequality relative to the semicircular law, that is,
\[
\norm{F(X) - X}_2^2 \leq \norm{X}_2^2 + m \log 2\pi - 2 \chi^*(X),
\]
where the left hand side is twice the entropy relative to semicircular (see \S \ref{subsec:entropycost}).  This is not surprising because it was already known in the classical case that the Talagrand inequality can be witnessed by some triangular transport \cite[Corollary 3.10]{BKM2005}.  Moreover, our construction of the transport maps is a direct application of the same method that Otto and Villani used to prove the Talagrand entropy-cost inequality under the assumption of the log-Sobolev inequality \cite[\S 4]{OV2000}.  Thus, our main contribution is to study the large $N$ limit of the transport maps using asymptotic approximation by trace polynomials.  We also show that $F$ is $\norm{\cdot}_2$-Lipschitz, and we estimate $\norm{F(X) - X}_\infty$ in terms of the constants $c$ and $C$ specifying the uniform convexity and semi-concavity of $V^{(N)}$.  These estimates will in fact go to zero as $c, C \to 1$.

Unfortunately, the maps constructed here are not optimal triangular transport maps with respect to the $L^2$-Wasserstein distance, since Otto and Villani's proof of \cite[Theorem 1]{OV2000} uses a diffusion-semigroup interpolation between the two measures, not the displacement interpolation from optimal transport theory.  In that sense, the results of this paper do not fully prove an analogue of \cite[Corollary 3.10]{BKM2005}.  Even in the work of Guionnet and Shlyakhtenko \cite{GS2009}, which constructed monotone transport maps in the free setting, the question of whether these maps furnish an optimal coupling between $X$ and $S$ inside a tracial von Neumann algebra was left unresolved.  Future research should study optimal transport in the free setting, and determine whether the classical optimal transport (or more generally optimal triangular transport) maps for the random matrix models converge in the large $N$ limit in the sense of this paper.

\subsection{Outline}

The paper is organized as follows.  We remark that \S \ref{sec:RMbackground} and \S \ref{sec:diffeqtools} are mostly technical background, and the reader may treat them like appendices if desired.  In other words, it is feasible to read through the other sections in order and only refer to \S \ref{sec:RMbackground} and \S \ref{sec:diffeqtools} as needed to verify technical details of the main results.

\S \ref{sec:RMbackground} gives standard background on convex and semi-concave functions and on log-concave random matrix models.

\S \ref{sec:TP} sets up the algebra of trace polynomials, and the spaces $\overline{\TrP}_m^0$ and $\overline{\TrP}_m^1$ of functions that can be approximated by trace polynomials.   These spaces provide a framework for functional calculus in multiple self-adjoint variables $X_1$, \dots, $X_m$ that can realize every element of $L^2(\mathrm{W}^*(X_1,\dots,X_m))$.  They are a convenient tool to describe the large $N$ behavior of functions of several matrices, and thus will be used in the statements of our main theorems.

\S \ref{sec:diffeqtools} describes solving ODE's and the heat equation over $\overline{\TrP}_m^1$.  These are the technical lemmas used in the rest of the paper to show that the solutions of certain PDE's have well-defined large $N$ limits.

\S \ref{sec:conditionalexpectation} explains the setup of our random matrix models given by convex potentials, and then proves our main result on conditional expectation (Theorem \ref{thm:conditionalexpectation}).

\S \ref{sec:entropy} shows that the conditional entropy for random matrix models converges to the conditional non-microstates entropy (Theorem \ref{thm:convergenceofentropy}). 

\S \ref{sec:transport} proves the existence of transport maps from a free Gibbs law to the law of a free semicircular $m$-tuple which arise as the large $N$ limit of transport maps for the random matrix models (Theorem \ref{thm:transport1} and \ref{thm:transport2}).

\S \ref{sec:applications} discusses applications of our results.  We show that our standard set of assumptions for log-concave random matrix models is preserved under marginals, independent joins, linear change of variables, and convolution (\S \ref{subsec:operations}).  We show that the transport maps constructed above witness (the conditional version of) Talagrand's entropy-cost inequality relative to Gaussian measure (Theorem \ref{thm:transport3}).  Then by iterating our conditional transport results, we show the existence of triangular transport (Theorem \ref{thm:transport4}).

\section{Multi-matrix Models from Convex Potentials} \label{sec:RMbackground}

This section is a review and reference for basic results we will use throughout the paper.

We will be concerned with probability measures on $M_N(\C)_{sa}^m$ of the form
\[
d\mu(x) = \frac{1}{Z} e^{-N^2 V(x)}\,dx,
\]
where $x = (x_1,\dots,x_m)$ is a tuple of self-adjoint matrices, $V: M_N(\C)_{sa}^m \to \R$ such that $e^{N^2 V}$ is integrable, and $Z = \int e^{-N^2 V(x)}\,dx$ is the normalizing constant.  Here $dx$ denotes Lebesgue measure where we identify $M_N(\C)_{sa}^m$ with $\R^{mN^2}$ using the inner product associated to the trace (the normalization of Lebesgue measure is irrelevant here because if we multiply it by a constant, the normalizing constant $Z$ for $\mu$ will change to compensate).  In this case, we will say that \emph{$\mu$ is the measure given by the potential $V$}.  We will often assume $V$ is convex.  Note that $\mu$ only determines $V$ up to an additive constant, but we will still say that ``$V$ is the potential corresponding to $\mu$'' with a slight abuse of terminology.

A primary motivating example is $V(x) = \tau_N(f(x))$, where $\tau_N = (1/N) \Tr$ is the normalized trace and $f$ is a non-commutative polynomial in $x_1$, \dots, $x_m$.  Unlike the notation in many random matrix papers, we prefer to write $N^2 \tau_N(f)$ rather than $N \Tr(f)$.  This seems natural because $\tau_N(f)$ is a function with dimension-independent normalization and it would make sense for self-adjoint elements of a tracial $\mathrm{W}^*$-algebra.  Meanwhile, $N^2$ is the dimension of $M_N(\C)_{sa}^N$ and also the scale (in the sense of large deviations) for the standard concentration estimates that hold when $V$ is uniformly convex (see for instance \cite{BCG2003} or \S \ref{subsec:concentration} below).

\subsection{Semi-convex and Semi-concave Functions}

\begin{definition} \label{def:convexityHnotation}
Let $A: M_N(\C)_{sa}^m \to M_N(\C)_{sa}^m$ be a self-adjoint linear transformation and let $u: M_N(\C)_{sa}^m \to \R$.  We say that $Hu \leq A$ if $u(x) - (1/2) \ip{Ax,x}_2$ is concave.  We say that $Hu \geq A$ if $u(x) - (1/2) \ip{Ax,x}_2$ is convex.
\end{definition}

We will also regularly use the following observation:

\begin{lemma} \label{lem:convexgradient}
Suppose that $u: M_N(\C)_{sa}^m \to \R$, and let $A$ and $B$ be self-adjoint linear transformations.  The following are equivalent:
\begin{enumerate}[(1)]
	\item $A \leq Hu \leq B$.
	\item For each $x \in M_N(\C)_{sa}^m$, there exists $y \in M_N(\C)_{sa}^m$ such that
	\[
	\frac{1}{2} \ip{A(x' - x), x' - x} \leq u(x') - u(x) - \ip{y, x' - x}_2 \leq \frac{1}{2} \ip{B(x' - x), x' - x}_2
	\]
	for all $x' \in M_N(\C)_{sa}^m$.
	\item $u$ is continuously differentiable and we have
	\[
	\ip{A(x'-x), x'-x}_2 \leq \ip{Du(x') - Du(x),x' - x}_2 \leq \ip{B(x'-x), x' - x}_2
	\]
	for all $x, x' \in M_N(\C)_{sa}^m$.
\end{enumerate}
Moreover, in this case, $Du$ is $\max(\norm{A}, \norm{B})$-Lipschitz with respect to $\norm{\cdot}_2$.
\end{lemma}

\begin{proof}[Sketch of proof]
(1) $\implies$ (3).  Suppose (1) holds.  If $C = \max(\norm{A},\norm{B})$, then for each $x$ there exists $y$ such that
\[
-\frac{C}{2} \norm{x' - x}_2^2 \leq u(x') - u(x) - \ip{y, x' - x}_2 \leq \frac{C}{2} \norm{x' - x}_2^2.
\]
Hence, it follows from \cite[Proposition 2.13]{Jekel2018} that $u$ must be continuously differentiable and $Du$ is $C$-Lipschitz (which proves the last claim of our lemma as well).  To prove the inequality asserted by (3), we can reduce to the case when $u$ is smooth using a similar argument as in \cite[Proposition 2.13]{Jekel2018}).  But in the smooth case, the claim follows by estimating from above and below the formula
\[
\ip{Du(x') - Du(x), x' - x}_2 = \int_0^1 \ip{Hu(x + t(x' - x))(x' - x), x' - x}_2\,dt,
\]
where $Hu$ is the Hessian defined in the standard pointwise sense.

(3) $\implies$ (2).  Recall the formula
\[
u(x') - u(x) = \int_0^1 \ip{Du(x + t(x' - x)), x' - x}\,dt.
\]
This implies that
\begin{align*}
u(x') - u(x) - \ip{Du(x), x' - x}_2 &= \int_0^1 \ip{Du(x + t(x' - x)) - Du(x), x' - x}_2\,dt \\
&= \int_0^1 \frac{1}{t} \ip{Du(x + t(x' - x)) - Du(x), [x + t (x' - x)] - x}_2\,dt \\
&\leq \int_0^1 \frac{1}{t} \ip{B[t(x' - x)], t(x' - x)}_2\,dt \\
&= \frac{1}{2} \ip{B(x' - x), x' - x}_2.
\end{align*}
This proves the upper bound, and the lower bound is symmetrical.

(2) $\implies$ (1).  This follows from the characterization of convex functions by supporting hyperplanes.  Indeed, $u(x) - (1/2) \ip{Ax,x}$ is convex if and only if for every $x$, there exists $y$ satisfying
\[
u(x') - u(x) + \frac{1}{2} \ip{Ax', x'}_2 - \frac{1}{2} \ip{Ax,x}_2 \geq \ip{y, x - x'}_2.
\]
which is equivalent to the right inequality of (2), and the concavity of $u(x) - (1/2) \ip{Bx,x}$ follows similarly.
\end{proof}

\begin{lemma} \label{lem:convexgradientestimate}
Suppose that $0 \leq Hu \leq A$ for some linear transformation $A$.  Then $u$ is differentiable and we have
\[
|\ip{Du(x) - Du(x'),y}_2| \leq \ip{A(x-x'),x-x'}_2^{1/2} \ip{Ay,Ay}_2^{1/2},
\]
so that in particular, $\norm{Du(x) - Du(x')}_2 \leq \norm{A} \norm{x - x'}_2$.
\end{lemma}

\begin{proof}
As in \cite[Proposition 2.13]{Jekel2018}, we obtain differentiability; and moreover to prove the asserted estimate, it suffices to prove the claim for smooth functions $u$.  In this case,
\begin{align*}
\ip{Du(x) - Du(x'),y} &= \int_0^1 \ip{Hu(tx + (1 - t) x')(x - x'),y}_2 \,dt \\
&\leq \int_0^1 \ip{Hu(tx + (1-t)x')(x-x'),x-x'}_2^{1/2} \ip{Hu(tx + (1-t)x')y,y}_2^{1/2}\,dt \\
&\leq \int_0^1 \ip{A(x-x'), x-x'}^{1/2} \ip{Ay,y}_2^{1/2}\,dt \\
&= \ip{A(x-x'),x-x'}_2^{1/2} \ip{Ay,y}_2^{1/2}. \qedhere
\end{align*}
\end{proof}

\subsection{Some Basic Lemmas}

Let $V: M_N(\C)_{sa}^m \to \R$ satisfy $HV \geq c$.   Then one can check that $e^{-N^2 V(x)}$ is integrable; indeed, $V$ must achieve a minimum at some $x_0$ and we have $V(x) \geq V(x_0) + (c/2) \norm{x - x_0}_2^2$ and clearly $e^{-N^2 c \norm{x - x_0}_2^2}$ is integrable.  Therefore, the probability measure $\mu$ given by $(1 / Z^{(N)}) e^{-N^2 V(x)}\,dx$ is well-defined.

\begin{lemma} \label{lem:conjugatevariablebasics}
Let $V: M_N(\C)_{sa}^m \to \R$ satisfy $c \leq HV \leq C$ for some scalars $0 < c \leq C$.   Let $\mu$ be the probability measure given by $d\mu(x) = (1/ Z^{(N)}) e^{-N^2 V(x)}\,dx$ and let $X$ be a random variables whose distribution is $\mu$.  Then
\[
E[DV(X)] = 0
\]
and
\[
\frac{m}{C} \leq E \norm*{X - E(X)}_2^2 \leq \frac{m}{c}.
\]
\end{lemma}

\begin{proof}
We remark that $V$ is continuously differentiable by Lemma \ref{lem:convexgradient} $V$ is differentiable and $DV$ is Lipschitz.  It follows by some straightforward estimation that $\norm{DV}_2$ is integrable with respect to $\mu$, so that $E[DV(X)]$ is well-defined.  Then $E[DV(X)] = 0$ follows from integration by parts (see \S \ref{subsec:matrixentropy} for further context on this integration by parts).

Next, let $D_j V$ denote the normalized gradient with respect to the matrix variable $x_j$.  Using integration by parts again, we get $E \ip{D_j V(X), X_j - E(X_j)}_2 = 1$, so that
\[
E \ip{DV(X) - DV(E(X)), X - E(X)}_2 = E \ip{DV(X), X - E(X)}_2 = m.
\]
On the other hand, by Lemma \ref{lem:convexgradientestimate},
\[
c E \norm*{X - E(X)}_2^2 \leq E \ip{DV(X) - DV(E(X)), X - E(X)}_2 \leq C E \norm*{X - E(X)}_2^2.
\]
Since the middle term evaluates to $m$, the proof is complete.
\end{proof}

\begin{lemma} \label{lem:RVboundedness}
Let $X$ be a random variable in $M_N(\C)_{sa}^m$ and let $G: M_N(\C)_{sa}^m \to M_N(\C)_{sa}^n$ be Lipschitz with respect to $\norm{\cdot}_2$ in both the domain and target space, and let $\norm{G}_{\Lip}$ denote the corresponding Lipschitz (semi)norm. Then
\[
\norm*{G(x)}_2 \leq \norm*{E(G(X))}_2 + \norm*{G}_{\Lip} \left( \norm*{x - E(X)}_2 + (E\norm*{X - E(X)}_2^2)^{1/2} \right).
\]
\end{lemma}

\begin{proof}
Note that
\begin{align*}
\norm*{G(x) - E(G(X))}_2 &\leq \norm*{G}_{\Lip} E \norm*{x - X}_2 \\
&\leq \norm*{G}_{\Lip} \left( \norm*{x - E(X)}_2 + E \norm*{X - E(X)}_2 \right) \\
&\leq \norm*{G}_{\Lip} \left( \norm*{x - E(X)}_2 + (E \norm*{X - E(X)}_2^2)^{1/2} \right).  \qedhere
\end{align*}
\end{proof}

\begin{corollary} \label{cor:DVestimate}
Let $V: M_N(\C)_{sa}^m \to \R$ satisfies $c \leq HV \leq C$, let $\mu$ be the corresponding measure, and let $X \sim \mu$.  Then
\[
\norm*{DV(x)}_2 \leq C \left( \norm*{x - E(X)}_2 + \frac{m^{1/2}}{c^{1/2}} \right).
\]
\end{corollary}

\begin{proof}
We apply Lemma \ref{lem:RVboundedness} to $DV(X)$.  Also, $DV$ is $C$-Lipschitz by Lemma \ref{lem:convexgradient}.  By Lemma \ref{lem:conjugatevariablebasics} $E(DV(X)) = 0$ and $E \norm*{X - E(X)}_2^2 \leq m / c$.  
\end{proof}

\begin{lemma} \label{lem:limitoflogconcave}
Let $A$ and $B$ be positive definite linear transformations $M_N(\C)_{sa}^m \to M_N(\C)_{sa}^m$.  Let $\{V_k\}_{k \in \N}$ be a sequence of functions such that $A \leq HV_k \leq B$.  Let $d\mu_k(x) = (1/Z_k) e^{-N^2 V_k(x)}\,dx$ be the associated probability measure.  Let $\mu$ be another measure with finite mean.  Suppose $\mu_k$ converges weakly to $\mu$ and the mean of $\mu_k$ is bounded in $\norm*{\cdot}_2$ as $k \to \infty$.  Then there exists $V$ such that $d\mu(x) = (1/Z) e^{-N^2 V(x)}\,dx$ and $A \leq HV \leq B$.
\end{lemma}

\begin{proof}
Since adding a constant to $V_k$ does not change $\mu_k$, we can assume without loss of generality that $V_k(0) = 0$.  Now $DV_k$ is $C$-Lipschitz where $C = \max(\norm{A}, \norm{B})$, hence the sequence is equicontinuous.  It is also pointwise bounded in light of the previous lemma, since we assumed the mean of $\mu_k$ is bounded as $k \to \infty$.  Thus, by the Arzel{\`a}-Ascoli theorem, by passing to a subsequence, we can assume that $DV_k$ converges locally uniformly to some $F$ as $k \to \infty$.  Since $V_k(0) = 0$, this also implies that $V_k$ converges locally uniformly to some $V$, which must satisfy $A \leq HV \leq B$ since the family of such functions is closed under pointwise limits (which follows from the family of convex functions being closed under pointwise limits; compare \cite[Proposition 2.13(1)]{Jekel2018}).  Moreover, $DV = F$.

Let $\nu$ be the probability measure given by $d\nu(x) = (1/Z) e^{-N^2 V(x)}\,dx$.  Since $A$ is positive definite, we have $A \geq c$ for some scalar $c > 0$.  Because $DV_k(0)$ is bounded in $\norm*{\cdot}_2$ as $k \to \infty$ and $V_k(x) \geq \ip{x, DV_k(0)}_2 + c \norm*{x}_2^2$, we can see using the dominated convergence theorem that $Z_k \to Z$ as $k \to \infty$.  It follows again from dominated convergence that $\int \phi\,d\mu_k \to \int \phi\,d\nu$ for every continuous compactly supported $\phi$.  Hence, $\nu = \mu$, so $\mu$ is given by the potential $V$.
\end{proof}

\subsection{Log-Sobolev Inequality and Concentration} \label{subsec:concentration}

Log-concave matrix models exhibit concentration of measure as $N \to \infty$ as a consequence of the following classical inequalities. 

\begin{definition}
We say that a measure $\mu$ on $\R^m$ satisfies the \emph{log-Sobolev inequality with constant $c$} if for all sufficiently smooth $f$,
\begin{equation} \label{eq:LSI}
\int f^2 \log \frac{f^2}{\int f^2 \,d\mu} \,d\mu \leq 2c \int |\nabla f|^2\,d\mu.
\end{equation}
\end{definition}

\begin{definition}
We say that a measure $\mu$ on $\R^m$ satisfies \emph{Herbst's concentration inequality with constant $c$} if for all Lipschitz functions $f: \R^m \to \R$ and $\delta > 0$, we have $E |f(X)| < +\infty$ and
\begin{equation} \label{eq:Herbst1}
P(f(X) - E[f(X)] \geq \delta) \leq e^{-c \delta^2 / 2 \norm{f}_{\Lip}^2}
\end{equation}
where $X$ is a random variable distributed according to $\mu$.  Note that by symmetry this implies
\begin{equation} \label{eq:Herbst2}
P(|f(X) - E[f(X)]| \geq \delta) \leq 2 e^{-c \delta^2 / 2 \norm{f}_{\Lip}^2}.
\end{equation}
\end{definition}

The following theorem is now standard.  See for instance \cite[\S 2.3.3 and 4.4.2]{AGZ2009} and \cite{BL2000}.  To summarize the history, the log-Sobolev inequality was introduced by Gross \cite{Gross1975}.   In the theorem below, (1) is due to Bakry and Emery and (2) is due to unpublished work of Herbst.  The application to random matrices was introduced by Guionnet and Zeitouni \cite{GZ2000}.


\begin{theorem} ~
\begin{enumerate}
	\item Suppose that $\mu$ is a probability measure on $\R^m$ satisfying $d\mu(x) = (1/Z)e^{-V(x)}\,dx$ and suppose that $V(x) - (c/2) |x|^2$ is convex.  Then $\mu$ satisfies the log-Sobolev inequality with constant $1/c$.
	\item If $\mu$ satisfies the log-Sobolev inequality with constant $1/c$, then it satisfies Herbst's concentration inequality with constant $c$.
\end{enumerate}
\end{theorem}

In particular, we have the following consequences for random matrices.  Here we use the gradient $Df$ and Hessian $Hf$ with respect to the normalized inner product $\ip{\cdot,\cdot}_2$.

\begin{corollary} \label{cor:matrixLSI}
Suppose that $V: M_N(\C)_{sa}^m \to \R$ satisfies $HV \geq c > 0$ and let $d\mu(x) = (1/Z) e^{-N^2 V(x)}\,dx$.  Then $\mu$ satisfies the normalized log-Sobolev inequality
\begin{equation} \label{eq:normalizedLSI}
\int f^2 \log \frac{f^2}{\int f^2 \,d\mu} \,d\mu \leq \frac{2}{N^2 c} \int \norm{Df}_2^2\,d\mu,
\end{equation}
and hence also satisfies the normalized Herbst concentration inequality
\begin{equation} \label{eq:normalizedHerbst}
P(f(X) - E[f(X)] \geq \delta) \leq e^{-cN^2 \delta^2 / 2 \norm{f}_{\Lip}^2},
\end{equation}
where $f: M_N(\C)_{sa}^m \to \R$ is Lipschitz and $\norm{f}_{\Lip}$ denotes the Lipschitz norm with respect to $\norm{\cdot}_2$.
\end{corollary}

\begin{lemma} \label{lem:epsilonnet}
Suppose that $\mu$ is a probability measure on $M_N(\C)_{sa}^m$ satisfying \eqref{eq:normalizedHerbst} for some constant $c$.  Let $f: M_N(\C)_{sa}^m \to M_N(\C)_{sa}$ be Lipschitz with respect to $\norm*{\cdot}_2$.  Then we have
\begin{equation} \label{eq:epsilonnet}
P\Bigl( \norm*{f(X) - E[f(X)]}_\infty \geq c^{-1/2} \norm*{f}_{\Lip} (\Theta + \delta) \Bigr) \leq e^{-N \delta^2 / 2}.
\end{equation}
where $X \sim \mu$ and where $\Theta$ is a universal constant (independent of $N$ and $c$).
\end{lemma}

\begin{proof}
First, observe that $\norm{x}_\infty \leq N^{1/2} \norm{x}_2$ for $x \in M_N(\C)_{sa}^m$.  In particular, $g(x) = \norm{f(x) - E[f(X)]}_\infty$ is $N^{1/2} \norm{f}_{\Lip}$-Lipschitz with respect to $\norm{\cdot}_2$, and thus
\[
P (g(X) \geq E[g(X)] + \delta) \leq e^{-cN \delta^2 / 2 \norm{f}_{\Lip}^2},
\]
which implies after a change of variables for $\delta$ that
\[
P(g(X) \geq E[g(X)] + c^{-1/2} \norm{f}_{\Lip} \delta) \leq e^{-N \delta^2 / 2}.
\]
Therefore, it suffices to show that for some constant $\Theta$, we have
\begin{equation} \label{eq:expectationofnorm}
E[g(X)] = E[ \norm{f(X) - E[f(X)]}_\infty] \leq \Theta c^{-1/2} \norm{f}_{\Lip}.
\end{equation}

We may assume without loss of generality that $f$ is self-adjoint since in the general case, $f = (1/2)(f + f^*) + i(1/2i)(f - f^*)$, and each of the terms on the right hand side is Lipschitz.  Thus, the self-adjoint case would imply the non-self-adjoint case at the cost of doubling the constant $\Theta$.  Now to prove self-adjoint case, we use an ``$\epsilon$-net argument'' that is well-known in random matrix theory (see \cite[\S 2.3.1]{Tao2012}).  Fix $N$.  Let $\{\eta_j\}_{j=1}^J$ be a maximal collection of unit vectors in $\C^N$ such that $|\eta_i - \eta_j| \geq 1/3$ for all $i \neq j$.  Since this collection is maximal, for every unit vector $\eta$, there exists some $\eta_j$ with $|\eta - \eta_i| < 1/3$.  Now if $a \in M_N(\C)_{sa}$, then there is a unit vector with $\norm*{a}_\infty = \ip{\eta, a \eta}$.  We may then choose $\eta_j$ with $|\eta - \eta_j| < 1/3$
\begin{align*}
\norm*{a}_\infty &= \ip{\eta, a \eta} \\
&= \ip{\eta_j, a \eta_j} + \ip{\eta_j, a(\eta - \eta_j)}_2 + \ip{\eta - \eta_j, a \eta} \\
&\leq \ip{\eta_j, a \eta_j} + \frac{1}{3} \norm*{a}_\infty + \frac{1}{3} \norm*{a}_\infty,
\end{align*}
so that
\[
\norm*{a}_\infty \leq 3 \max_j \ip{\eta_j, a \eta_j}.
\]
Note that the balls $\{B(\eta_j,1/6)\}_{j=1}^J$ in $\C^N$ are disjoint and contained in $B(0,7/6)$.  Hence, we can estimate the number of vectors by
\[
J \leq \frac{|B(0,7/6)|}{|B(0,1/6)|} = 7^{2N}.
\]

Let $K = \norm{f}_{\Lip}$.  For a matrix $a \in M_N(\C)_{sa}$, we have
\[
|\ip{\eta_i, a \eta_j}| \leq \norm*{a}_\infty \leq N^{1/2} \norm*{a}_2.
\]
This implies that $x \mapsto \ip{\eta_j, f(x) \eta_j}$ is $KN^{1/2}$-Lipschitz with respect to $\norm*{\cdot}_2$ and hence
\[
P\Bigl( \ip{\eta_j, (f(X) - E[f(X)]) \eta_j} \geq \delta \Bigr) \leq e^{-cN \delta^2 / 2K^2}
\]
Since $\norm*{a}_\infty \leq 3 \max_{j} \ip{\eta_j, a \eta_j}$, we have
\begin{align*}
P \Bigl(\norm*{f(X) - E[f(X)]}_\infty \geq 3 \delta \Bigr) &\leq J e^{-cN \delta^2 / 2} \\
&\leq 7^{2N} e^{-cN \delta^2 / 2K^2}.
\end{align*}
Thus, for any $t_0 > 0$, we have
\begin{align*}
E[\norm*{f(X) - E[f(X)]}_\infty] &= \int_0^\infty P(\norm{f(X) - E[f(X)]}_\infty \geq t)\,dt \\
&\leq \int_0^{t_0} 1 \,dt + \int_{t_0}^\infty 7^{2N} e^{-cNt^2 / 18 K^2}\,dt \\
&\leq t_0 + \int_{t_0}^\infty 7^{2N} \frac{t}{t_0} e^{-cNt^2 / 18 K^2}\,dt \\
&= t_0 + 7^{2N} \frac{9K^2}{cNt_0} e^{-cNt_0^2 / 18 K^2}.
\end{align*}
Now substitute $t_0 = 6c^{-1/2} K (\log 7)^{1/2}$ and obtain \eqref{eq:expectationofnorm} with
\[
\Theta = 6(\log 7)^{1/2} + \frac{9}{6 (\log 7)^{1/2}}.
\]
(In fact, for a fixed $N$, we may use $\Theta_N = 6(\log 7)^{1/2} + 9 / 6N(\log 7)^{1/2}$ in the self-adjoint case.)
\end{proof}

\section{Functional Calculus and Asymptotic Approximation} \label{sec:TP}

In this section, we review the algebra $\TrP_m^1$ of trace polynomials in self-adjoint variables $X_1$, \dots, $X_m$, as well as a certain completed quotient $\overline{\TrP}_m^1$ of this algebra.  The elements of $\overline{\TrP}_m^1$ represent functions that can be applied to \emph{any} tuple of self-adjoint non-commutative random variables $(X_1,\dots,X_m)$ in an $\mathcal{R}^\omega$-embeddable tracial $\mathrm{W}^*$-algebra, and application of these functions will produce every element of $L^2(\mathrm{W}^*(X_1,\dots,X_m))$ (see Proposition \ref{prop:realizationofoperators}).  These functions are closed under certain algebraic and composition operations.  Moreover, they are a natural tool to describe the large $N$ limit of functions on $M_N(\C)_{sa}^m$, which we will apply in the rest of the paper.

\subsection{The Algebra of Trace Polynomials}

Trace polynomials have been used by several previous authors in the study of deterministic and random matrices; a brief list is \cite{Razmyslov1974}, \cite{Razmyslov1987}, \cite{Rains1997}, \cite{Sengupta2008}, \cite{Cebron2013}, \cite{DHK2013} (which coined the term ``trace polynomial''), \cite{Kemp2016}, \cite{Kemp2017}, \cite{DGS2016} but they are also used implicitly in many other works.  We use the same notation as in our previous paper \cite{Jekel2018}.

We denote by $\NCP_m = \C\ip{X_1,\dots,X_m}$ the $*$-algebra of polynomials in $m$ self-adjoint non-commuting variables $X_1$, \dots, $X_m$.

We denote by $\TrP_m^0$ the $*$-algebra of \emph{scalar-valued trace polynomials}.  A formal definition is given in \cite{Jekel2018}; in short, it is the tensor algebra of the vector space of non-commutative polynomials modulo cyclic symmetry.  Informally, this is the commutative $*$-algebra generated by functions of the form $\tau(p(X_1,\dots,X_m))$, where $p$ is a non-commutative polynomial in $X = (X_1,\dots,X_m)$ and $\tau$ is a formal symbol (which stands in for a normalized trace on a von Neumann algebra), where $\tau(p(X))^* = \tau(p(X)^*)$, and where we identify $\tau(p(X) q(X))$ with $\tau(q(X) p(X))$ for all polynomials $p$ and $q$.  Thus, $\TrP_m^0$ is spanned as a vector space by elements of the form $\tau(p_1(X)) \dots \tau(p_n(X))$ where $p_1$, \dots, $p_n \in \NCP_m$.

We denote by $\TrP_m^1$ the $*$-algebra of \emph{operator-valued trace polynomials}.  This is the $*$-algebra given formally as $\TrP_m^0 \otimes \NCP_m$.  As a vector space, it is spanned by elements of the form $\tau(p_1(X)) \dots \tau(p_n(X)) q(X)$, where $p_1$, \dots, $p_n$ and $q$ are in $\NCP_m$.  More generally, we would denote $\TrP_m^k = \TrP_m^0 \otimes (\NCP_m)^{\otimes k}$, but these spaces will not be needed in this paper.

The \emph{degree} of a trace polynomial is defined as one would expect; see \cite[\S 3.1]{Jekel2018} for precise explanation.

Suppose that $x_1$, \dots, $x_m$ are self-adjoint elements of a tracial von Neumann algebra $(\mathcal{M},\tau_0)$.  Then elements of $\NCP_m$, $\TrP_m^0$, and $\TrP_m^1$ can be evaluated on $(x_1,\dots,x_m)$ and $\tau_0$ by substituting the operator $x_j$ and the trace $\tau_0$ in place of the formal symbols $X_j$ and $\tau$.  More precisely, the evaluation map $\varepsilon_{(x_1,\dots,x_m)}: \NCP_m \to \mathcal{M}$ is the unique $*$-algebra homomorphism that sends $X_j$ to $x_j$.  Similarly, the evaluation map $\varepsilon_{(x_1,\dots,x_m)}^0: \TrP_m^0 \to \C$ is the unique $*$-algebra homomorphism that sends $\tau(p(X))$ to $\tau_0(\varepsilon_{(x_1,\dots,x_m)}(p))$.  Finally, the evaluation map $\varepsilon_{(x_1,\dots,x_m)}^1: \TrP_m^1 \to \mathcal{M}$ is $\varepsilon_{(x_1,\dots,x_m)}^0 \otimes \varepsilon_{(x_1,\dots,x_m)}$, that is,
\[
\tau(p_1(X)) \dots \tau(p_n(X)) q(X) \mapsto \tau_0(\varepsilon_{(x_1,\dots,x_m)}^0(p_1)) \dots \tau_0(\varepsilon_{(x_1,\dots,x_m)}^0(p_n)) \varepsilon_{(x_1,\dots,x_m)}(q).
\]

For the most part, we will abuse notation and denote $f(x) = \varepsilon_{(x_1,\dots,x_m)}(f)$ when $f \in \NCP_m$, and similarly for $f \in \TrP_m^0$ or $f \in \TrP_m^1$.  Note in particular that we can consider $(\mathcal{M},\tau_0) = (M_N(\C)_{sa}, \tau_N)$ and thus $f(x)$ is defined for $x \in M_N(\C)_{sa}^m$ and $f \in \TrP_m^0$ or $\TrP_m^1$.

These evaluation maps thus allow us to view $f \in \TrP_m^0$ as a function (or rather a family of functions) $\mathcal{M}_{sa}^m \to \C$ for every tracial $\mathrm{W}^*$-algebra $(\mathcal{M},\tau)$ and in particular $M_N(\C)_{sa}^m \to \C$ for every $m$.  Similarly, every $f \in \TrP_m^1$ defines a function $\mathcal{M}_{sa}^m \to \mathcal{M}$ for every tracial $\mathrm{W}^*$-algebra and in particular a function $M_N(\C)_{sa}^m \to M_N(\C)$ for every $N$.

\subsection{Functions Approximable by Trace Polynomials}

From an analytic viewpoint, we prefer to work with certain separation-completions of $\TrP_m^0$ and $\TrP_m^1$.  In \cite[\S 8.1]{Jekel2018}, we sketched several equivalent ways of defining these separation-completions.  Here we emphasize their description as functions that can be evaluated on any self-adjoint tuple in $\mathcal{R}^\omega$ (or, as we will see, any $\mathcal{R}^\omega$-embeddable $\mathrm{W}^*$-algebra).

Let $\mathcal{R}$ denote the hyperfinite $\II_1$ factor (tracial $\mathrm{W}^*$-algebra with trivial center) and let $\mathcal{R}^\omega$ be its (tracial $\mathrm{W}^*$-algebra) ultrapower with respect to some fixed free ultrafilter $\omega \in \beta \N \setminus \N$.

Consider the case of $\TrP_m^0$ first.  Let $\mathcal{F}_m^0$ denote the space of functions $(\mathcal{R}^\omega)_{sa}^m \to \C$ that are bounded on operator norm balls, equipped with the family of semi-norms
\[
\norm*{\phi}_{u,R} = \sup \{|\phi(x)|: \norm*{x}_\infty \leq R\}.
\]
(Here ``$u$'' stands for uniform.)  This is clearly a Fr{\'e}chet space since the topology is given by the countable family of semi-norms given by taking $R \in \N$ (for background on Fr{\'e}chet spaces, see e.g.\ \cite[\S 5.4]{Folland1999}).  Every $f \in \TrP_m^0$ defines a function $(\mathcal{R}^\omega)_{sa}^m \to \C$ that is a bounded an operator norm balls.  In other words, evaluation produces a map $\TrP_m^0 \to \mathcal{F}_m^0$.  We denote by $\overline{\TrP}_m^0$ the closure of the image of this map in $\mathcal{F}_m$.  In other words, $\overline{\TrP}_m^0$ is the space of functions $(\mathcal{R}^\omega)_{sa}^m \to \C$ that can be approximated uniformly on operator-norm balls by trace polynomials.

\begin{remark}
This space was denoted as $\mathcal{T}_m^0$ in our earlier paper \cite{Jekel2018}.  The notation $\overline{\TrP}_m^0$ is slightly abusive since we have not shown that the map $\TrP_m^0 \to \mathcal{F}_m^0$ is injective (and perhaps it is not).  However, we will still use the notation $\overline{\TrP}_m^0$ since it indicates the connection with trace polynomials.
\end{remark}

Earlier, we saw that it makes sense to evaluate a trace polynomial $f$ on any self-adjoint tuple $(x_1,\dots,x_m)$ in a tracial von Neumann algebra.  In fact, $f(x_1,\dots,x_m)$ makes sense for every $f \in \overline{\TrP}_m^0$ when $x_1$, \dots, $x_m$ come from a tracial von Neumann algebra that embeds into $\mathcal{R}^\omega$.  To see this, suppose $(\mathcal{M},\tau)$ admits a normal trace-preserving embedding $\iota: \mathcal{M} \to \mathcal{R}^\omega$.  Then we define $f(x_1,\dots,x_m) = f(\iota(x_1),\dots,\iota(x_m))$.  This is independent of the choice of trace-preserving embedding if $f$ is a trace polynomial, and hence it must also be independent of the choice of embedding when $f$ is in $\overline{\TrP}_m^0$.

A similar separation-completion can be defined for $\TrP_m^1$.  Indeed, let $\mathcal{F}_m^1$ be the set of functions $\phi: (\mathcal{R}^\omega)_{sa}^m \to L^2(\mathcal{R}^\omega)$ such that
\[
\norm*{\phi}_{u,R} := \sup \{\norm*{\phi(x)}_2: \norm*{x}_\infty \leq R\}
\]
is finite for each $R$.  Again, this is a Fr{\'e}chet space.  Through the evaluation map, every trace polynomial defines an element of $\mathcal{F}_m^1$ and hence there is a linear map $\TrP_m^1 \to \mathcal{F}_m^1$.  We define $\overline{\TrP}_m^1$ to be the closure of the image of this map in $\mathcal{F}_m^1$.

Similar to the scalar-valued case, we can define evaluation of $f \in \overline{\TrP}_m^1$ for tuples in an $\mathcal{R}^\omega$-embeddable tracial $\mathrm{W}^*$-algebra $(\mathcal{M},\tau)$ by using any trace preserving embedding $\iota: \mathcal{M} \to \mathcal{R}^\omega$.  Indeed, let $x_1,\dots,x_m \in \mathcal{M}_{sa}$.  Clearly, for $f \in \TrP_m^1$, we have $f(\iota(x_1),\dots,\iota(x_m)) \in \iota(\mathcal{M}) \subseteq \iota(L^2(\mathcal{M}))$ where the latter is defined by extending $\iota$ to a map $L^2(\mathcal{M}) \to L^2(\mathcal{R}^\omega)$.  Since this holds for $f \in \TrP_m^1$, then by taking limits, we have $f(\iota(x_1),\dots,\iota(x_m)) \in \iota(L^2(\mathcal{M}))$ for all $f \in \overline{\TrP}_m^1$.  Therefore, we may define $f(x_1,\dots,x_m)$ by $\iota(f(x_1,\dots,x_m)) = f(\iota(x_1),\dots,\iota(x_m))$.  Then one can check this is independent of the choice of embedding similarly as we did in the case of $\overline{\TrP}_m^0$.

\begin{remark}
Because the spaces $\overline{\TrP}_m^j$ used here are non-standard, let us briefly describe their relationship to other more familiar ideas.  Recall that $\Sigma_{m,R}$ denotes the space of non-commutative laws of $m$-tuples with operator norms bounded by $R$.  We denote by $\Sigma_{m,R}^{\app}$ the subspace of laws that can be realized by $m$-tuples in $\mathcal{R}^\omega$, and $\Sigma_m^{\app} = \bigcup_{R > 0} \Sigma_{m,R}^{\app}$.  Then we showed in \cite[Lemma 8.2]{Jekel2018} that $\overline{\TrP}_m^0$ consists of functions $\Sigma_m^{\app} \to \C$ such that the restriction to $\Sigma_{m,R}^{\app}$ is continuous for each $R$.  One could think of this alternatively as an inverse limit of $C(\Sigma_{m,R}^{\app})$ over the directed system of restriction maps $C(\Sigma_{m,R'}^{\app}) \to C(\Sigma_{m,R}^{\app})$ for $R' > R$.
\end{remark}

\begin{remark}
The spaces $\TrP_m^0$ and $\overline{\TrP}_m^0$ also arise naturally in the study of model theory of tracial von Neumann algebras introduced in \cite{FHS2013,FHS2014,FHS2014b}.  To avoid some of the technical complexities of sorts, we follow the definitions in \cite{FHS2014} where the language has multiple domains of quantification for each sort (and thus we can get away with fewer sorts), and in which formulas are obtained by applying continuous functions $\R^n \to \R$ to atomic formulas (rather than functions defined on some compact set).  For tracial von Neumann algebra $(M,\tau)$, the language includes (though this list is not exhaustive) a sort representing $M$ with domains of quantification for each operator norm ball of radius $n \in \N$, a special relation-like symbol $d(x,y)$ for the distance $\norm{x - y}_2$, a relation symbol for the trace $\tau(x)$, and function symbols for the adjoint, addition, and multiplication.

Now $\tau(p(x_1,\dots,x_m))$ is an example of a atomic formula (or strictly speaking, its real and imaginary parts are basic formulas).  Similarly, $\tau(p(\re(x_1),\dots, \re(x_m))$ is an atomic formula, where $\re(x_j) = (x_j + x_j^*) / 2$.  Since the elements of $\TrP_m^0$ is obtained by multiplying formulas such as $\tau(p)$, we see that $f(\re(x_1),\dots,\re(x_m))$ is a quantifier-free formula for every $f \in \TrP_m^0$.  Moreover, the supremum of $|f(\re(x_1),\dots,\re(x_m))$ over $\{x: \norm{x_j} \leq R_j\}$ is the same as the supremum of $f$ over $\{x: x_j = x_j^*, \norm{x_j} \leq R_j\}$.  The limiting objects $\overline{\TrP}_m^0$ (evaluated on the real parts of operators) are thus uniform limits of quantifier-free formulas on each domain of quantification for every $\mathcal{R}^\omega$-embeddable tracial von Neumann algebra, that is, they are ``quantifier-free definable predicates'' relative to the theory of $\mathcal{R}^\omega$-embeddable tracial von Neumann algebras.  Conversely, since $\overline{\TrP}_m^0$ is closed under the operation $(f_1,\dots,f_n) \mapsto \phi(f_1,\dots,f_n)$ for $\phi: \C^n \to \C$ continuous, every quantifier-free definable predicate $f$ satisfying $f(x_1,\dots,x_m) = f(\re(x_1),\dots,\re(x_m))$ is an element of $\overline{\TrP}_m^0$.

The elements of $\overline{\TrP}_m^1$, evaluated on the real parts of operators, may be viewed similarly as certain ``quantifier-free definable functions'' relative to the theory of $\mathcal{R}^\omega$-embeddable tracial von Neumann algebras, meaning that $\norm{f(x) - y}_2^2$ is a quantifier-free definable predicate | actually, for technical reasons a definable function is required to map an operator norm ball into an operator norm ball, so the last statement only applies if we assume our function $f \in \overline{\TrP}_m^1$ has this property (but it turns out that such functions exist in abundance in $\overline{\TrP}_m^1$; see Proposition \ref{prop:realizationofoperators} and Proposition \ref{prop:operatornormestimate}).  Alternatively, in order to deal with functions with codomain $L^2$, we must first modify the language by adding another sort for $L^2(M)$, with domains of quantification corresponding to $L^2$-balls, which will act as the target space of the functions in $\overline{\TrP}_m^1$.

The quantifier-free nature of these formulas is a model-theoretic heuristic for why they behave well under limits in non-commutative law (hence describing the large $N$ limits of random matrix models).  In fact,\cite[Proposition 6.28]{Jekel2018} re-expresses a formula given by quantifiers in a quantifier-free way in order to get behavior under limits.   There, we studied the inf-convolution $(Q_t V)(x) = \inf_y [V(y) - (1/2t) \norm{x - y}_2^2]$ for self-adjoint tuples $x$ and $y$.  If $V \in \TrP_m^0$, then for each $ > 0$,
\[
W_{t,R}(\re(x_1),\dots, \re(x_m)) := \inf_{\norm{y} \leq R} \left[V(\re(y_1),\dots,\re(y_m)) + \frac{1}{2t} \norm{\re(x) - \re(y)}_2^2 \right]
\]
is a formula in the language of tracial von Neumann algebras whose definition involves the quantifier $\inf$.  But if $V$ is convex and semi-concave and $DV \in (\overline{\TrP}_m^1)_{sa}^m$, then the self-adjoint tuple $y$ where the infimum
\[
W_t(\re(x_1),\dots, \re(x_m)) := \inf_{\norm{y} \leq R} \left[V(\re(y_1),\dots,\re(y_m)) + \frac{1}{2t} \norm{\re(x) - \re(y)}_2^2 \right]
\]
is achieved can be evaluated as the limit of a fixed-point iteration using functions from $(\overline{\TrP}_m^1)_{sa}^m$, and hence $y = \phi(\re(x))$ for some $\phi \in (\overline{\TrP}_m^1)_{sa}^m$ (see \cite[Proposition 6.28]{Jekel2018}).  Moreover, it follows from the results in \cite{Jekel2018} that $\phi$ is Lipschitz in $\norm{\cdot}_2$, and thus in light of Proposition \ref{prop:operatornormestimate} below, $\phi$ is bounded in operator norm on operator norm balls.  So $\phi(\re(x))$ is quantifier-free definable function.  We can also conclude that $W_{t,R} \to W_t$ as $R \to \infty$ uniformly on operator norm balls, so $W_t$ is a definable formula (allowing quantifiers).  But then because
\[
W_{t,R}(x) =V(\phi(x)) - \frac{1}{2t} \norm{\re(x) - \phi(x)}_2^2,
\]
we conclude that $W_t$ is in fact a \emph{quantifier-free} definable predicate.

On the other hand, without the ability to eliminate the quantifier like this, we could not hope for $Q_t V$ to behave so well for the large $N$ limit of random matrix models.  Indeed, for $Q_tV(x)$ to depend continuously on the non-commutative law $\lambda_x$ for $x$ in each operator norm ball, it must be in $\TrP_m^0$ by the last remark, and hence it is a quantifier-free definable predicate.

Many of the properties shown in the next section about operations on $\TrP_m^0$ and $\TrP_m^1$ are natural from the model theoretic viewpoint, but we sketch self-contained justifications nonetheless.
\end{remark}

\subsection{Asymptotic Approximation for Functions of Matrices} \label{subsec:AATP}

Our earlier work introduced asymptotic approximability by trace polynomials for a sequence of functions on $M_N(\C)_{sa}^m$, which is a precise description of good asymptotic behavior as $N \to \infty$ suitable for free probabilistic analysis in the limit.

\begin{definition} \label{def:AATP}
Let $\phi^{(N)}: M_N(\C)_{sa}^m \to \C$.  We say that $\{\phi^{(N)}\}$ is \emph{asymptotically approximable by trace polynomials} if for every $R > 0$ and $\epsilon > 0$, there exists $f \in \TrP_m^0$ such that
\[
\limsup_{N \to \infty} \sup_{\substack{x \in M_N(\C)_{sa}^m \\ \norm*{x}_\infty \leq R}} |\phi^{(N)}(x) - f(x)| \leq \epsilon.
\]
Similarly, for matrix-valued functions $\phi^{(N)}: M_N(\C)_{sa}^m \to M_N(\C)$, we say that $\{\phi^{(N)}\}$ is \emph{asymptotically approximable by trace polynomials} if for every $R > 0$ and $\epsilon > 0$, there exists $f \in \TrP_m^1$ such that
\[
\limsup_{N \to \infty} \sup_{\substack{x \in M_N(\C)_{sa}^m \\ \norm*{x}_\infty \leq R}} \norm*{\phi^{(N)}(x) - f(x)}_2 \leq \epsilon.
\]
\end{definition}

It will be convenient to denote
\[
\norm*{\phi^{(N)}}_{u,R}^{(N)} = \sup_{\substack{x \in M_N(\C)_{sa}^m \\ \norm*{x}_\infty \leq R}} |\phi(x)|
\]
in the scalar-valued case and similarly for the matrix-valued case with $\norm*{\phi(x)}_2$ rather than $|\phi(x)|$.  Thus, for instance, the preceding definition says that there exists a trace polynomial $f$ with
\[
\limsup_{N \to \infty} \norm*{\phi^{(N)} - f}_{u,R}^{(N)} < \epsilon.
\]

Moreover, it is implicit from our discussion in \cite[\S 8.1]{Jekel2018} that if $\phi^{(N)}$ is asymptotically approximable by trace polynomials, then it will  be \emph{asymptotic to} some $f \in \overline{\TrP}_m^0$ or $\overline{\TrP}_m^1$ in the following sense.

\begin{definition} \label{def:AA}
Let $\phi^{(N)}: M_N(\C)_{sa}^m \to \C$ or $M_N(\C)$ respectively, and let $f \in \overline{\TrP}_m^0$ or $\overline{\TrP}_m^1$ respectively.  Then we say that $\{\phi^{(N)}\}$ is \emph{asymptotic to $f$}, or $\phi^{(N)} \rightsquigarrow f$ if for every $R > 0$,
\[
\lim_{N \to \infty} \norm*{\phi^{(N)} - f}_{u,R}^{(N)} = 0.
\]
Similarly, if $\phi^{(N)}: M_N(\C)_{sa}^m \to M_N(\C)$ and $f \in \overline{\TrP}_m^1$, we make the same definitions with $|\phi^{(N)}(x) - f(x)|$ replaced by $\norm*{\phi^{(N)}(x) - f(x)}_2$.
\end{definition}

\begin{lemma} \label{lem:AATP}
Let $\phi^{(N)}: M_N(\C)_{sa}^m \to \C$ (respectively, $\to M_N(\C)$).  Then $\phi^{(N)}$ is asymptotically approximable by trace polynomials if and only if there exists $f \in \overline{\TrP}_m^0$ (respectively, $f \in \overline{\TrP}_m^1$) such that $\phi^{(N)} \rightsquigarrow f$.  Moreover, $\norm*{f}_{u,R} = \lim_{N \to \infty} \norm*{\phi^{(N)}}_{u,R}^{(N)}$ for each $R$.
\end{lemma}

\begin{proof}
We record the proof only for the case of scalar-valued functions, since the proof for operator-valued case is identical with minor changes of notation.  Suppose that $\{\phi^{(N)}\}$ is asymptotically approximable by trace polynomials.  Then there exists a sequence $\{f_k\}$ of trace polynomials such that for every $R > 0$,
\[
\lim_{k \to \infty} \limsup_{N \to \infty} \sup_{\substack{x \in M_N(\C)_{sa}^m \\ \norm*{x}_\infty \leq R}} |\phi^{(N)}(x) - f_k(x)| = 0.
\]
As in \cite[Lemma 8.1]{Jekel2018}, if $g \in \TrP_m^0$, then
\[
\sup_{\substack{x \in (\mathcal{R}^\omega)_{sa}^m \\ \norm*{x}_\infty \leq R}} |g(x)| = \sup_{N \in \N} \sup_{\substack{x \in M_N(\C)_{sa}^m \\ \norm*{x}_\infty \leq R}} |g(x)| = \limsup_{N \to \infty} \sup_{\substack{x \in M_N(\C)_{sa}^m \\ \norm*{x}_\infty \leq R}} |g(x)|
\]
which implies that
\[
\norm*{g}_{u,R} = \limsup_{N \to \infty} \norm*{g}_{u,R}^{(N)}.
\]
Applying this to $g = f_j - f_k$, we obtain from the triangle inequality
\[
\norm*{f_j - f_k}_{u,R} \leq \limsup_{N \to \infty} \norm*{f_j(x) - \phi^{(N)}}_{u,R}^{(N)} + \limsup_{N \to \infty} \norm*{\phi^{(N)} - f_k}_{u,R}^{(N)},
\]
and hence $f_k$ is Cauchy with respect to $\norm*{\cdot}_{u,R}$ for each $R > 0$.  Hence, $f_k$ converges to some $f \in \overline{\TrP}_m^0$.  By similar use of the triangle inequality,
\[
\limsup_{N \to \infty} \norm*{\phi^{(N)} - f}_{u,R}^{(N)} \leq \norm*{f - f_k}_{u,R} + \limsup_{N \to \infty} \norm*{\phi^{(N)} - f_k}_{u,R}^{(N)}.
\]
Hence, $\phi^{(N)} \rightsquigarrow f$.

Conversely, suppose that $\phi^{(N)} \rightsquigarrow f \in \overline{\TrP}_m^0$.  Choose $f_k \in \TrP_m^0$ such that $\norm*{f_k - f}_{u,R} \to 0$ for every $R$.  Then
\[
\limsup_{N \to \infty} \norm*{\phi^{(N)}(x) - f_k}_{u,R}^{(N)} \leq \norm*{f - f_k}_{u,R} + \limsup_{N \to \infty} \norm*{\phi^{(N)} - f}_{u,R}^{(N)} = \norm*{f - f_k}_{u,R}.
\]
Hence, it follows that $\{\phi^{(N)}\}$ is asymptotically approximable by trace polynomials, namely the polynomials $\{f_k\}$.

We leave the proof of the last claim that $\norm*{f}_{u,R} = \lim_{N \to \infty} \norm*{\phi^{(N)}}_{u,R}^{(N)}$ to the reader.
\end{proof}

\begin{remark}
If $\phi^{(N)}: M_N(\C)_{sa}^m \to M_N(\C)_{sa}$ and $\{\phi^{(N)}\}$ is asymptotically approximable by trace polynomials, then we can asymptotically approximate it using \emph{self-adjoint} trace polynomials.  Indeed, if
\[
\limsup_{N \to \infty} \norm*{\phi^{(N)} - f}_{u,R}^{(N)} \leq \epsilon,
\]
then the same holds with $f$ replaced by $(1/2)(f + f^*)$.  Similarly, if $\phi^{(N)}(x)$ is self-adjoint and $\phi^{(N)} \rightsquigarrow f \in \overline{\TrP}_m^1$, then $f$ must be self-adjoint.
\end{remark}

\begin{remark}
Definitions \ref{def:AATP} and \ref{def:AA} and Lemma \ref{lem:AATP} extend naturally to tuples $f = (f_1,\dots,f_n) \in (\overline{\TrP}_m^1)^n$ and $\phi^{(N)} = (\phi_1^{(N)}, \dots, \phi_n^{(N)}): M_N(\C)_{sa}^m \to M_N(\C)^n$.  We shall apply them to tuples without further comment in the rest of the paper.
\end{remark}

\subsection{Algebra, Composition, and Limits}

\begin{lemma} \label{lem:algebra}
$\overline{\TrP}_m^0$ is an algebra and $\overline{\TrP}_m^1$ is a module over $\overline{\TrP}_m^0$.  Also, if $f, g \in \overline{\TrP}_m^1$, then $\tau(fg) \in \overline{\TrP}_m^0$.  Moreover, suppose that $\phi^{(N)}, \phi^{(N)}: M_N(\C)_{sa}^m \to \C$ and $f^{(N)}, g^{(N)}: M_N(\C)_{sa}^m \to M_N(\C)$ are asymptotically approximable, and $\phi^{(N)} \rightsquigarrow \phi$, $\psi^{(N)} \rightsquigarrow \psi$, $f^{(N)} \rightsquigarrow f$, and $g^{(N)} \rightsquigarrow g$.  Then we have
\begin{align*}
\phi^{(N)} + \psi^{(N)} &\rightsquigarrow \phi + \psi \\
\phi^{(N)} \psi^{(N)} & \rightsquigarrow \phi \psi \\
f^{(N)} + g^{(N)} &\rightsquigarrow f + g \\
\phi^{(N)} f^{(N)} &\rightsquigarrow \phi f \\
\tau_N(f^{(N)} g^{(N)}) & \rightsquigarrow \tau(fg).
\end{align*}
\end{lemma}

\begin{proof}
Since the proofs of all the statements are straightforward and similar to each other, we will only explain how to show that if $\phi \in \overline{\TrP}_m^0$ and $f \in \overline{\TrP}_m^1$, then $\phi f \in \overline{\TrP}_m^1$ and that if $\phi^{(N)} \rightsquigarrow \phi$ and $f^{(N)} \rightsquigarrow f$, then $\phi^{(N)} f^{(N)} \rightsquigarrow \phi f$.

First, note that $\phi f$ is well-defined as a function on $(\mathcal{R}^\omega)_{sa}^m$ by multiplying the scalar $\phi(x)$ times the vector $f(x)$ for each $x \in (\mathcal{R}^\omega)_{sa}^m$, and also clearly $\norm*{\phi f}_{u,R} \leq \norm*{\phi}_{u,R} \norm*{f}_{u,R}$.  To show that $\phi f \in \overline{\TrP}_m^1$, it suffices to show that for every $\epsilon > 0$ and $R > 0$, the function $\phi f$ can be approximated by an element of $\TrP_m^1$ with respect to $\norm*{\cdot}_{u,R}$ with error less than $\epsilon$.  We first choose $h \in \TrP_m^1$ such that
\[
\norm*{h - f}_{u,R} \norm*{\phi}_{u,R} < \frac{\epsilon}{2}.
\]
Then we choose $\theta \in \TrP_m^0$ such that
\[
\norm*{\theta - \phi}_{u,R} \norm*{h}_{u,R} < \frac{\epsilon}{2},
\]
and we conclude with the routine observation that
\begin{align*}
\norm*{\theta h - \phi f}_{u,R} &\leq \norm*{(\theta - \phi) \cdot h}_{u,R} + \norm*{\phi \cdot (h - f)}_{u,R} \\
&\leq \norm*{\theta - \phi}_{u,R} \norm*{h}_{u,R} + \norm*{\phi}_{u,R} \norm*{h - f}_{u,R} < \epsilon.
\end{align*}

Next, to show $\phi^{(N)} f^{(N)} \rightsquigarrow \phi f$, first observe that
\[
M := \sup_N \norm*{\phi^{(N)}}_{u,R}^{(N)} < +\infty.
\]
Then
\begin{align*}
\norm*{\phi^{(N)} f^{(N)} - \phi f}_{u,R}^{(N)} &\leq \norm*{\phi^{(N)} \cdot (f^{(N)} - f)}_{u,R}^{(N)} + \norm*{(\phi^{(N)} - \phi) \cdot f}_{u,R}^{(N)} \\
&\leq M \norm*{f^{(N)} - f}_{u,R}^{(N)} + \norm*{\phi^{(N)} - \phi}_{u,R}^{(N)} \norm*{f}_{u,R},
\end{align*}
which implies that $\norm*{\phi^{(N)} f^{(N)} - \phi f}_{u,R}^{(N)} \to 0$.
\end{proof}

In addition to their algebraic structure, functions $(\mathcal{R}^\omega)_{sa}^m \to (\mathcal{R}^\omega)_{sa}^n$ given by trace polynomials are closed under composition.  It turns out that self-adjoint tuples from $\overline{\TrP}_m^1$ are closed under composition under the assumption of $\norm*{\cdot}_2$-uniform continuity of the ``outside'' function (Lemma \ref{lem:composition} below).

We say that $f \in \overline{\TrP}_m^1$ is \emph{$\norm*{\cdot}_2$-uniformly continuous} if for every $\epsilon > 0$, there exists $\delta > 0$ such that
\[
\forall x, y \in (\mathcal{R}^\omega)_{sa}^m, \quad \norm*{x - y}_2 < \delta \implies \norm*{f(x) - f(y)}_2 < \epsilon.
\]
Furthermore, we say $f \in \overline{\TrP}_m^1$ is \emph{$\norm*{\cdot}_2$-Lipschitz} if $\norm*{f(x) - f(y)}_2 \leq K \norm*{x - y}$ for some constant $K$, which is an important special case of uniform continuity.  We denote the minimum such constant by $\norm*{f}_{\Lip}$.  We make the analogous definitions for $f \in \overline{\TrP}_m^0$.

\begin{observation} \label{obs:continuousextension}
If $f$ is a function from $(\mathcal{R}^\omega)_{sa}^m$ to $\mathcal{R}^\omega$ or $\C$ that is $\norm*{\cdot}_2$-uniformly continuous, then it has a unique continuous extension to $L^2(\mathcal{R}^\omega)_{sa}^m$, which is also $\norm*{\cdot}_2$-uniformly continuous.  Similarly, if $f$ is Lipschitz on $(\mathcal{R}^\omega)_{sa}^m$, then the extension is also Lipschitz.
\end{observation}

\begin{lemma} \label{lem:limituniformlycontinuous}
Suppose that $f^{(N)}: M_N(\C)_{sa}^m \to \C$ or $M_N(\C)$ and that $f^{(N)} \rightsquigarrow f$.  If $f^{(N)}$ is $\norm{\cdot}_2$-uniformly continuous with respect to some modulus of continuity independent of $N$, then $f$ is $\norm{\cdot}_2$-uniformly continuous on $(\mathcal{R}^\omega)_{sa}^m$ with the same modulus of continuity.
\end{lemma}

\begin{proof}
Let us only explain the operator-valued case where $f^{(N)}$ is $M_N(\C)$-valued and $f \in \overline{\TrP}_m^1$, since the scalar-valued case is easier.  We define scalar-valued functions of $2m$ variables by $F^{(N)}(x,y) = \norm{f^{(N)}(x) - g^{(N)}(y)}_2^2$ and $F(x,y) = \norm{f(x) - f(y)}_2^2$.  By Lemma \ref{lem:algebra}, we have $F^{(N)} \rightsquigarrow F \in \overline{\TrP}_{2m}^0$.

Let $\epsilon(\delta)$ be a common modulus of continuity for $f^{(N)}$.  Let $x$ and $y \in (\mathcal{R}^\omega)_{sa}^m$.  Then we may embed $\mathrm{W}^*(x,y)$ into $(\mathcal{M},\tau) := \prod_{N \to \omega} (M_N(\C),\tau_N)$, that is the tracial $\mathrm{W}^*$-ultraproduct of matrices.  There exist tuples $x^{(N)}$ and $y^{(N)}$ of $N \times N$ matrices such that $x = \{x^{(N)}\}_{N \in \N}$ and $y = \{y^{(N)}\}_{N \in \N}$ in the ultraproduct and also $\norm{x^{(N)}}_\infty \leq \norm{x}_\infty$ and $\norm{y^{(N)}}_\infty \leq \norm{y}_\infty$.  Observe that
\[
F(x,y) = \lim_{N \to \omega} F(x^{(N)},y^{(N)}).
\]
(This equality holds for trace polynomials and hence holds for all functions in $\overline{\TrP}_{2m}^0$ by approximation.)  On the other hand, we also have for $R > \max(\norm{x}_\infty, \norm{y}_\infty)$ that
\[
|F^{(N)}(x^{(N)}, y^{(N)}) - F(x^{(N)},y^{(N)})| \leq \norm{F - F^{(N)}}_{u,R}^{(N)} \to 0.
\]
Therefore,
\[
\norm{f(x) - f(y)}_2 = \lim_{N \to \omega} \norm{f^{(N)}(x^{(N)}) - f^{(N)}(y^{(N)})}_2 \leq \lim_{N \to \omega} \epsilon( \norm{x^{(N)} - y^{(N)}}_2 ) \leq \epsilon (\norm{x - y}_2),
\]
since $\norm{x^{(N)} - y^{(N)}}_2 \to \norm{x - y}_2$.
\end{proof}

\begin{lemma} \label{lem:composition}
Let $j = 0$ or $1$.  Let $f \in \overline{\TrP}_m^j$ be $\norm*{\cdot}_2$-uniformly continuous and let $g = (g_1,\dots,g_m) \in (\overline{\TrP}_n^1)_{sa}^m$.
\begin{enumerate}
	\item Then $f \circ g$ is a well-defined function on $(\mathcal{R}^\omega)_{sa}^n$, and it is in $\overline{\TrP}_n^j$.
	\item If $g$ is also $\norm*{\cdot}_2$-uniformly continuous, then so is $f \circ g$.
	\item Suppose $f^{(N)}$ is a function on $M_N(\C)_{sa}^m$ and $g^{(N)}: M_N(\C)_{sa}^n \to M_N(\C)_{sa}^m$ such that $f^{(N)} \rightsquigarrow f$ and $g^{(N)} \rightsquigarrow g$.  Also, suppose that $f^{(N)}$ is $\norm*{\cdot}_2$-uniformly continuous with the modulus of continuity also uniform in $N$.  Then $f^{(N)} \circ g^{(N)} \rightsquigarrow f \circ g$.
\end{enumerate}
\end{lemma}

\begin{proof}
(1) Because $f$ extends to a function on $L^2(\mathcal{R}^\omega)_{sa}^m$, we can define $f \circ g$.  Now let us show $f \circ g \in \overline{\TrP}_m^j$.  Choose $\epsilon > 0$ and $R > 0$.  By uniform continuity of $f$, there exists a $\delta > 0$ such that $\norm*{x - y}_2 < \delta$ implies $|f(x) - f(y)|$ or $\norm*{f(x) - f(y)}_2 < \epsilon/2$ (for $j = 0$ or $1$ respectively).  Now choose $\tilde{g} \in (\TrP_n^1)_{sa}^m$ such that $\norm*{\tilde{g} - g}_{u,R} < \delta$, and hence
\[
\norm*{f \circ g - f \circ \tilde{g}}_{u,R} < \frac{\epsilon}{2}.
\]
Because $\tilde{g}$ is a trace polynomial, there is some $R'$ such that $\norm*{x}_\infty \leq R$ implies $\norm*{\tilde{g}}_\infty \leq R'$.  Choose $\tilde{f} \in \overline{\TrP}_m^j$ with $\norm*{\tilde{f} - f}_{u,R'} < \epsilon / 2$, and hence
\[
\norm*{f \circ \tilde{g} - \tilde{f} \circ \tilde{g}}_{u,R} < \frac{\epsilon}{2}.
\]
Then altogether we have $\norm*{f \circ g - \tilde{f} \circ \tilde{g}}_{u,R} < \epsilon$.

(2) This is immediate.

(3) This is similar to the proof of (1).  Fix $R > 0$ and $\epsilon > 0$.  Choose $\delta > 0$ such that $\norm*{x - y}_2 < \delta$ implies $|f(x) - f(y)|$ or $\norm*{f(x) - f(y)}_2 < \epsilon / 2$ and such that the same holds for $f^{(N)}$ as well.  Let $\tilde{g} \in (\TrP_n^1)_{sa}^m$ such that $\norm*{\tilde{g} - g}_{u,R} < \delta$.  Note that for sufficiently large $N$, we have $\norm*{g^{(N)} - \tilde{g}}_{u,R}^{(N)} < \delta$ and hence
\[
\norm*{f^{(N)} \circ g^{(N)} - f^{(N)} \circ \tilde{g}}_{u,R}^{(N)} < \epsilon /2.
\]
Then let $R'$ and $\tilde{f}$ be as in (1).  Then for sufficiently large $N$, we have
\[
\norm*{f^{(N)} \circ \tilde{g} - \tilde{f} \circ \tilde{g}}_{u,R}^{(N)} < \frac{\epsilon}{2},
\]
so overall
\[
\norm*{f^{(N)} \circ g^{(N)} - \tilde{f} \circ \tilde{g}}_{u,R}^{(N)} < \epsilon, \qquad \norm*{f \circ g - \tilde{f} \circ \tilde{g}}_{u,R} < \epsilon,
\]
so that$\norm*{f^{(N)} \circ g^{(N)} - f \circ g}_{u,R}^{(N)} < 2 \epsilon$ for large enough $N$.
\end{proof}

Moreover, asymptotically approximable sequences are closed under limits in an appropriate sense.

\begin{lemma} \label{lem:limits}
Let $f_k^{(N)}: M_N(\C)_{sa}^m \to \C$ or to $M_N(\C)$ for $k$ and $N \in \N$.  Suppose that $f_k^{(N)} \rightsquigarrow f_k$ in $\overline{\TrP}_m^j$ for each $k$, and that
\begin{equation} \label{eq:limithypothesis}
\lim_{k \to \infty} \limsup_{N \to \infty} \norm*{f_k^{(N)} - f^{(N)}}_{u,R} = 0 \text{ for all } R.
\end{equation}
Then $f_k$ converges in $\overline{\TrP}_m^j$ to some $f$, and we have $f^{(N)} \rightsquigarrow f$.
\end{lemma}

\begin{proof}
Note that
\[
\norm*{f_k - f_j}_{u,R} = \lim_{N \to \infty} \norm*{f_k^{(N)} - f_j^{(N)}}_{u,R}^{(N)}
\]
Then because of our assumption \eqref{eq:limithypothesis}, we see that $\{f_k\}_{k \in \N}$ is Cauchy with respect to $\norm*{\cdot}_{u,R}$ for each $R$.  Thus, $f_k$ converges to some $f$.  Then to show that $f^{(N)} \rightsquigarrow f$ is a routine argument.
\end{proof}

\subsection{Functional Calculus and Operator Norm Bounds}

Now we will show that every element of $L^2(\mathrm{W}^*(x_1,\dots,x_m))$ can be expressed as $f(x_1,\dots,x_m)$ for some $f \in \overline{\TrP}_m^1$.  In fact, we can arrange that $f$ can be approximated uniformly by Lipschitz functions.  It will be convenient to define the uniform norm
\[
\norm*{f}_u = \sup_{R > 0} \norm*{f}_{u,R} = \sup_{x \in (\mathcal{R}^\omega)_{sa}^m} \norm*{f(x)}_2,
\]
and we make the same definition for $\norm*{f}_u^{(N)}$ where the supremum is instead taken over $x \in M_N(\C)_{sa}^m$.

\begin{proposition} \label{prop:realizationofoperators}
Let $x_1$, \dots, $x_m$ be self-adjoint variables which generate a tracial $\mathrm{W}^*$-algebra $(\mathcal{M},\tau)$ that is embeddable into $\mathcal{R}^\omega$.  Let $z \in L^2(\mathcal{M},\tau)$.
\begin{enumerate}
	\item There exists a $\norm*{\cdot}_2$-uniformly continuous $f \in \overline{\TrP}_m^1$ such that $Z = f(x_1, \dots, x_m)$.
	\item The $f$ in (1) can be chosen so that there are $\norm*{\cdot}_2$-Lipschitz functions $f_k \in \overline{\TrP}_m^1$ such that $\norm*{f_k - f}_u \to 0$.
	\item If $z \in \C\ip{x_1,\dots,x_m}$, then $f$ can be chosen to be $\norm*{\cdot}_2$-Lipschitz.
\end{enumerate}
\end{proposition}

We use the following auxiliary observation.  Here $\Sigma_{m,R}$ will denote the space of non-commutative laws for an $m$-tuple of operators with operator norm $\leq R$.  We equip $\Sigma_{m,R}$ with the topology of convergence in moments.  Recall that $\Sigma_{m,R}$ is compact, separable, and metrizable.  In \cite[Lemma 8.2]{Jekel2018}, we noted the relationship between $\overline{\TrP}_m^0$ and continuous functions on $\Sigma_{m,R}$ for each $R$.  This same idea motivates the proof of the next lemma.

\begin{lemma}
Let $\mu \in \Sigma_{m,R}$ and let $\mathcal{U}$ be a neighborhood of $\mu$, and let $\epsilon > 0$.  Then there exists a trace polynomial $f$ such that
\[
\mu(f) = 1 \qquad
0 \leq \nu(f) \leq \mathbf{1}_{\nu \in \mathcal{U}} + \epsilon \text{ for } \nu \in \Sigma_{m,R}.
\]
\end{lemma}

\begin{proof}
By Urysohn's lemma, there exists a continuous function $F: \Sigma_{m,R} \to [0,1]$ such that $F(\mu) = 1$ and $F(\nu) = 0$ for $\nu \not \in \mathcal{U}$.  The functions $\Sigma_{m,R} \to \C$ of the form $\mu \mapsto \mu(f)$ for $f \in \TrP_m^0$ form a self-adjoint algebra in $C(\Sigma_{m,R})$, and they separate points because by definition two laws are the same if they agree on every non-commutative polynomial.  So by the Stone-Weierstrass theorem, this algebra is dense in $C(\Sigma_{m,R})$.  In particular, there exists a trace polynomial $g$ such that $|\nu(g) - F(\nu)| < \epsilon / 2$ for all $\nu \in \Sigma_{m,R}$.  Then let $f = (g + \epsilon/2) / (g(\mu) + \epsilon / 2)$.
\end{proof}

We will also use the following smooth cut-off trick.

\begin{lemma}
Let $0 < R' \leq R$.  Let $\phi \in C_c^\infty(\R;\R)$ such that $\phi(t) = t$ for $t \leq R'$ and $|\phi(t)| \leq R$.  For $y \in (\mathcal{R}^\omega)_{sa}$, define $\Phi(y) = \phi(y)$ where $\phi$ is applied through functional calculus.  Then
\begin{enumerate}
	\item $\Phi(y) = y$ if $\norm*{y}_\infty \leq R'$.
	\item $\norm*{\Phi(y)}_\infty \leq R$ for all $y$.
	\item $\Phi \in \overline{\TrP}_m^1$.
	\item $\Phi$ is globally $\norm*{\cdot}_2$-Lipschitz.
\end{enumerate}
\end{lemma}

\begin{proof}
(1) and (2) follow from the properties of functional calculus.  To prove (3), note by the Weierstrass approximation theorem that for every $r > 0$, there is a polynomial $p$ such that $|p(t) - \phi(t)| < \epsilon$ for $|t| \leq r$.  This implies as with (1) that $|p(y) - \phi(y)| < \epsilon$ for all $y$ with $\norm*{y}_\infty \leq r$.  Claim (4) follows from the results of \cite{Peller2006}; the argument is explained in \cite[(8.9) and Proposition 8.8]{Jekel2018}.
\end{proof}

\begin{proof}[Proof of Proposition \ref{prop:realizationofoperators}]
Let $\mu$ be the law of $x = (x_1,\dots,x_m)$, and let $R > \norm*{X}_\infty$.  Since $z \in L^2(\mathcal{M},\tau)$, there exist non-commutative polynomials $\{p_k\}_{k=1}^\infty$ such that $\norm*{p_k(x) - z}_2 < 1 / 2^{k+1}$ and hence for $k \geq 1$,
\[
\norm*{p_{k+1}(x) - p_k(x)}_2 = \mu[(p_{k+1} - p_k)^2]^{1/2} \leq \frac{1}{2^{k+1}} + \frac{1}{2^{k+2}} < \frac{1}{2^k}.
\]
By scaling, we may assume without loss of generality that $\norm*{z}_2 < 1$ and set $p_0 = 0$, and then the above statement also holds for $k = 0$.  Now let
\[
\mathcal{U}_k = \{\nu \in \Sigma_{m,R}: \nu((p_{k+1} - p_k)^2)^{1/2} < 1/2^k\},
\]
which is a neighborhood of $\mu$ in $\Sigma_{m,R}$.  By the previous lemma, there exists a scalar-valued trace polynomial $u_k$ such that $\mu(u_k) = 1$ and
\[
0 \leq \nu(u_k) \leq \mathbf{1}_{\nu \in \mathcal{U}_k} + \frac{1}{2^k \norm*{p_{k+1} - p_k}_{u,R}}.
\]
(We can assume without loss of generality that $\norm*{p_{k+1} - p_k}_{u,R} \neq 0$.)  Now the function $u_k(p_{k+1} - p_k)$ will evaluate at the point $X$ to $p_{k+1}(x) - p_k(x)$.  If $y \in (\mathcal{R}^\omega)_{sa}^m$ with $\norm*{y}_\infty \leq R$ and if the law of $y$ is in $\mathcal{U}_k$, then we will have
\[
\norm*{u_k(y)(p_{k+1}(Y) - p_k(y))}_2 \leq \norm*{p_{k+1}(y) - p_k(y)}_2 + \frac{1}{2^k \norm*{p_{k+1} - p_k}_{u,R}} \norm*{p_{k+1}(y) - p_k(y)}_2 \leq \frac{1}{2^k} + \frac{1}{2^k}.
\]
On the other hand, if the law of $y$ is not in $\mathcal{U}_k$, then $\norm*{u_k(Y)(p_{k+1}(y) - p_k(y))}_2 \leq 1/2^k$.  Overall, we have
\[
\norm*{u_k \cdot (p_{k+1} - p_k)}_{u,R} \leq \frac{2}{2^k}.
\]
This implies that $\sum_{k=0}^\infty u_k \cdot (p_{k+1} - p_k)$ converges with respect to $\norm*{\cdot}_{u,R}$ for our given choice of $R$, and of course evaluating this function on $X$ it produces the desired operator $Z$ since $u_k(x) = 1$.

To extend the function to be be globally defined on $(\mathcal{R}^\omega)_{sa}^m$, we use the smooth cut-off trick.  Let $\phi \in C_c^\infty(\R;\R)$ such that $\phi(t) = t$ for $|t| \leq \norm*{X}_\infty$ and $|\phi| \leq R$.  For $y = (y_1,\dots,y_m) \in (\mathcal{R}^\omega)_{sa}^m$, let $\Phi(y) = (\phi(y_1),\dots,\phi(y_m))$.  Then $[u_k \cdot (p_{k+1} - p_k)] \circ \Phi \in \overline{\TrP}_m^1$ because it is the composition of a trace polynomial with a function $\Phi \in (\overline{\TrP}_m^1)_{sa}^m$ that is uniformly bounded in operator norm.

Also, since $\Phi$ is globally $\norm*{\cdot}_2$-Lipschitz and since $u_k \cdot (p_{k+1} - p_k)$ is $\norm*{\cdot}_2$-Lipschitz on the operator norm ball of radius $R$, we see that $[u_k \cdot (p_{k+1} - p_k)] \circ \Phi$ is globally Lipschitz in $\norm*{\cdot}_2$.  For all $y \in (\mathcal{R}^\omega)_{sa}^m$,
\[
\norm*{u_k(\Phi(y))(p_{k+1}(\Phi(y)) - p_k(\Phi(y)))}_2 \leq \frac{2}{2^k}.
\]
Therefore,
\[
f(y) := \sum_{k=0}^\infty u_k(\Phi(y))(p_{k+1}(\Phi(y)) - p_k(\Phi(y)))
\]
converges, and clearly $f \in \overline{\TrP}_m^1$ since each of the individual terms is.  Furthermore, $\norm*{\cdot}_2$-uniform continuity of each term and the uniform convergence of the series implies uniform continuity of $f$.   Since $\norm*{x}_\infty \leq R$, we have $\Phi(x) = x$ and $u_k(x) = 1$, so that
\[
f(x) = \sum_{k=0}^\infty [p_{k+1}(x) - p_k(x)] = \lim_{k \to \infty} p_{k+1}(x) = z.
\]
This concludes the proof of (1).

To verify (2), we take $f_n$ to be the $n$th partial sum of the series defining $f$; we have shown that the individual terms are $\norm*{\cdot}_2$-Lipschitz, hence so are the partial sums.  Finally, to prove (3), note that if $z = p(x_1,\dots,x_m)$, then $z$ also equals $f(x_1,\dots,x_m)$ where $f = p \circ \Phi$, and by the same reasoning as above $p \circ \Phi$ is globally $\norm*{\cdot}_2$-Lipschitz.
\end{proof}

We have shown that every element of $L^2(\mathrm{W}^*(x_1,\dots,x_m))$ has the form $f(x_1,\dots,x_m)$ for some $f \in \overline{\TrP}_m^1$.  On the other hand, we will prove that if $f$ is Lipschitz, then $f(x)$ is actually bounded in operator norm.  We state our estimate in terms of unitarily invariant random matrix models which satisfy concentration \eqref{eq:normalizedHerbst}, but as explained in Remark \ref{rem:concentrationmodels} such models exist whenever $L^2(\mathrm{W}^*(x_1,\dots,x_m))$ is embeddable into $\mathcal{R}^\omega$.

\begin{proposition} \label{prop:operatornormestimate}
Let $x = (x_1, \dots, x_m)$ be a tuple of self-adjoint variables in a $\mathrm{W}^*$-algebra $(\mathcal{M},\tau)$ whose non-commutative law is $\lambda$.  Suppose there is a sequence $\{\mu^{(N)}\}$ of probability measures on $M_N(\C)_{sa}^m$, invariant under unitary conjugation, that satisfies the concentration estimate \eqref{eq:normalizedHerbst} for some constant $c$, and such that the corresponding random variables $X^{(N)} = (X_1^{(N)}, \dots, X_m^{(N)})$ satisfy $\lambda_{X^{(N)}} \to \lambda$ in probability.  Then $\mathrm{W}^*(x)$ is embeddable into $\mathcal{R}^\omega$.  Moreover, if $f \in \overline{\TrP}_m^1$ is $\norm{\cdot}_2$-Lipschitz, then $f(x)$ is a bounded operator and
\[
\norm*{f(x) - \tau(f(x))}_\infty \leq \Theta c^{-1/2} \norm*{f}_{\Lip},
\]
where $\Theta$ is a universal constant.
\end{proposition}

\begin{proof}
In light of Lemma \ref{lem:epsilonnet},
\[
P\left(\norm*{X^{(N)} - E(X^{(N)})}_\infty \leq c^{-1/2} (\Theta + N^{-1/3}) \right) \to 1
\]
and
\[
P\left(\norm*{f(X^{(N)}) - E(f(X^{(N)}))}_\infty \leq c^{-1/2} \norm{f}_{\Lip} (\Theta + N^{-1/3}) \right) \to 1.
\]
Also, the non-commutative law of $X^{(N)}$ converges in probability to that of $x$ and finally $\tau_N(f(X^{(N)})) - E[\tau_N(f(X^{(N)}))] \to 0$ in probability as a consequence of concentration.  Therefore, we may choose a sequence of elements $y^{(N)} \in M_N(\C)_{sa}^m$ such that
\begin{align*}
\limsup_{N \to \infty} \norm*{y_j^{(N)} - E(X_j^{(N)})}_\infty &\leq c^{-1/2} \Theta, \\
\limsup_{N \to \infty} \norm*{f(y^{(N)}) - E(f(X^{(N)}))}_\infty &\leq c^{-1/2} \norm{f}_{\Lip} \Theta, \\
\left| \tau_N(f(y^{(N)})) - E[\tau_N(f(X^{(N)}))] \right| &\to 0, \\
\lambda_{y^{(N)}} &\to \lambda_x.
\end{align*}
Because $E(X_j^{(N)}) = E(\tau_N(X_j^{(N)}))$ by unitary invariance and because of concentration, $E(\tau_N(X_j^{(N)}))$ must converge to $\tau(x_j)$ since $\tau_N(X_j^{(N)})$ converges to the $\tau(x_j)$ in probability.  So overall $E(X_j^{(N)}) - \tau_N(x_j) \to 0$ in operator norm.  In particular,
\[
\limsup_{N \to \infty} \norm*{y_j^{(N)} - \tau(x_j)}_\infty \leq c^{-1/2} \Theta,
\]
and hence $\norm{y^{(N)}}_\infty$ is bounded as $N \to \infty$.  Moreover, our choice of $y^{(N)}$ also satisfies
\[
\limsup_{N \to \infty} \norm*{f(y^{(N)}) - \tau_N(f(y^{(N)}))}_\infty \leq c^{-1/2} \norm{f}_{\Lip} \Theta,
\]
since $E[f(X^{(N)})] = E[\tau_N(f(X^{(N)})]$ again by unitary invariance.

Fix a free ultrafilter $\omega$ and let $(\mathcal{M},\tau) = \prod_{N \to \omega} (M_N(\C), \tau_N)$ be the tracial $\mathrm{W}^*$-ultraproduct of the sequence of matrix algebras.  Since $\{y^{(N)}\}$ is bounded in operator norm, $y = \{y^{(N)}\}_{N \in \N}$ defines an element of $(\mathcal{M},\tau)$.  By definition of ultraproducts, $\tau(p(y)) = \lim_{N \to \omega} \tau_N(p(y^{(N)}))$ for every non-commutative polynomial $p$ and therefore the non-commutative law of $y$ is $\lambda$ (which is the same as that of $x$).  In particular, $\mathrm{W}^*(x) \cong \mathrm{W}^*(y)$ embeds into $(\mathcal{M},\tau)$ and hence also into $\mathcal{R}^\omega$.  (Compare \cite[Theorem 4.4]{GS2009}.)

Since $\mathrm{W}^*(x)$ is $\mathcal{R}^\omega$-embeddable, $f(x)$ is well-defined, and clearly $\norm{f(x) - \tau(f(x))}_\infty = \norm{f(y) - \tau(f(y))}_\infty$.  Now we claim that $f(y)$ is given by the sequence $\{f(y^{(N)})\}_{N \in \N}$ as an element of $(\mathcal{M},\tau)$ (that is, application of $f$ commutes with ultralimits).  It is easy to check that $g(y) = \{g(y^{(N)})\}_{N \in \N}$ when $g \in \TrP_m^1$.  But for any $\epsilon > 0$, there exists $g \in \TrP_m^1$ with $\norm{f - g}_{c^{-1/2} \Theta + 1} < \epsilon$.  Thus, $\norm{f(y) - g(y)}_2 < \epsilon$ and also $\norm{f(y^{(N)}) - g(y^{(N)})}_2 < \epsilon$ for sufficiently large $N$.  This implies that $\norm{f(y) - \{f(y^{(N)})\}_{N \in \N}}_2 < 2 \epsilon$.  Thus, $f(y) = \{f(y^{(N)})\}_{N \in \N}$ as claimed.  The same holds with $f$ replaced by $f - \tau(f)$.  This implies
\[
\norm{f(y) - \tau(f(y))}_\infty \leq \limsup_{N \to \infty} \norm{f(y^{(N)}) - \tau(f(y^{(N)}))}_\infty \leq c^{-1/2} \norm{f}_{\Lip} \Theta. \qedhere
\]
\end{proof}

\begin{remark} \label{rem:concentrationmodels}
Suppose that $\mathrm{W}^*(x_1,\dots,x_m)$ is embeddable into $\mathcal{R}^\omega$.  Then there exist tuples $x^{(N)} = (x_1^{(N)},\dots,x_m^{(N)})$ in $M_N(\C)_{sa}^m$ such that $\norm{x^{(N)}}_\infty \leq \norm{x}_\infty$ and $\lambda_{x^{(N)}} \to \lambda_x$.  Let $U^{(N)}$ be an $N \times N$ random Haar unitary matrix and let $X^{(N)} = U^{(N)} x^{(N)} (U^{(N)})^*$.  Clearly, the probability distribution of $X^{(N)}$ is unitarily invariant and also $\lambda_{X^{(N)}} \to \lambda_x$ in probability.

To check concentration, observe that $u \mapsto u x^{(N)}u^*$ is a $2 m^{1/2} \norm{x}_\infty$-Lipschitz function from the unitary group to $M_N(\C)_{sa}^m$ with respect to $\norm{\cdot}_2$.  Therefore, if $f: M_N(\C)_{sa}^m \to \R$ is Lipschitz, then $u \mapsto f(ux^{(N)}u^*)$ is also Lipschitz, with the Lipschitz constant $2 m^{1/2} \norm{x}_\infty \norm{f}_{\Lip}$.  It was proved in \cite[Theorem 15]{Meckes2013}, \cite[Theorem 5.16]{Meckes2019} that the Haar measure on the unitary group satisfies the (non-normalized) log-Sobolev inequality with constant $6/N$ and the corresponding concentration of measure for Lipschitz functions with respect to the Hilbert-Schmidt metric $N^{1/2} \norm{\cdot}_2$.  After renormalization this implies that the Haar measure on the unitary group satisfies \eqref{eq:normalizedHerbst} with $c = 1/6$.  Hence, $X^{(N)}$ satisfies \eqref{eq:normalizedHerbst} with $c = 1 / 12 m \norm{x}_\infty^2$.
\end{remark}


\section{Tools for Differential Equations in $\overline{\TrP}_m^j$} \label{sec:diffeqtools}

This section describes two analytic operations | solution of ODE and convolution with the Gaussian law | that can be performed on tuples in $\overline{\TrP}_m^1$ and on asymptotically approximable sequences of functions on $N \times N$ matrices.  These operations were applied in \cite{Jekel2018}, and will be applied in the remainder of this paper, to analyze the large $N$ limit of certain PDE associated to random matrix models, and hence to understand the behavior of convex matrix models in the large $N$ limit.

\subsection{Flows Along Vector Fields} \label{subsec:vectorfields}

Several times in our study of partial differential equations, we will use flows along vector fields given by functions in $\overline{\TrP}_m^1$ and by asymptotically approximable sequences of functions on matrices.  For instance, this idea was used in \cite[Lemma 4.10]{Jekel2018}, and in this paper, it will be used in the proof of Lemma \ref{lem:diffusionAATP} and Theorem \ref{thm:transport1}.

The setup is roughly speaking as follows.  Consider a time interval $[0,T] \subseteq \R$.  Let $H: (\mathcal{R}^\omega)_{sa}^m \times [0,T] \to L^2(\mathcal{R}^\omega)_{sa}^m$ be a function such that $H(\cdot,t)$ is a tuple of functions in $\overline{\TrP}_m^1$ for each $t$ (satisfying certain uniform continuity assumptions).  Also, let $F_0: (\mathcal{R}^\omega)_{sa}^m \to L^2(\mathcal{R}^\omega)_{sa}^m$.  Then we would like to construct $F: (\mathcal{R}^\omega)_{sa}^m \times [0,T] \to (\mathcal{R}^\omega)_{sa}^m$ such that
\begin{align*}
F(x,0) &= F_0(x) \\
\partial_t F(x,t) &= H(F(x,t),t).
\end{align*}
Moreover, we would like to show that if $H^{(N)}$ is a function on $M_N(\C)_{sa}^m \times [0,T]$ that is asymptotic to $H$ and $F_0^{(N)} \rightsquigarrow F_0$, then the solutions $F^{(N)}$ are asymptotic to the solution $F$.

Such a proof was essentially carried out in \cite[Lemma 4.10]{Jekel2018}, but now we introduce the added complexity that $H$ will depend on $x$, $t$, and an auxiliary parameter $y \in (\mathcal{R}^\omega)_{sa}^m$, and we must solve the initial value problem
\begin{align} \label{eq:IVP}
F(x,y,0) &= F_0(x,y) \\
\partial_t F(x,y,t) &= H(F(x,y,t), y, t). \nonumber
\end{align}
The added parameter $y$ arises naturally in our analysis of \emph{conditional} expectation, entropy, and transport since it represents the variables we are conditioning upon (see for instance \S \ref{subsec:conditionalexpectationstrategy}).

For the sake of future reference, let us state the set of assumptions we make about the vector field $H(x,y,t)$.  These assumptions are framed for a convenient and applicable level of generality rather than maximum generality.

\begin{assumption} \label{ass:vectorfield}
We are given $T > 0$ and a function $H: (\mathcal{R}^\omega)_{sa}^m \times (\mathcal{R}^\omega)_{sa}^n \times [0,T] \to L^2(\mathcal{R}^\omega)_{sa}^m$ satisfying:
\begin{enumerate}
	\item For each $t$, we have $H(\cdot,\cdot,t) \in (\overline{\TrP}_{m+n}^1)_{sa}^m$.
	\item $H$ is $\norm*{\cdot}_2$-Lipschitz in $(x,y)$, that is, for some constant $K$ independent of $t$, we have
	\[
	\norm*{H(x,y,t) - H(x',y',t)}_2 \leq K \norm*{(x,y) - (x',y')}_2.
	\]
	\item The map $t \mapsto H(\cdot,\cdot,t)$ is a continuous function $[0,T] \to (\overline{\TrP}_{m+n}^1)_{sa}^m$ with respect to the Fr{\'e}chet topology on $\overline{\TrP}_{m+n}^1$.  This implies that for every $R > 0$ and for every $\epsilon > 0$, there exists $\delta > 0$, such that
	\[
	|t - t'| < \delta \implies \norm*{H(\cdot,\cdot,t) - H(\cdot,\cdot,t')}_{u,R} < \epsilon \text{ for all } t, t' \in [0,T].
	\]
	(where we have upgraded from continuity to uniform continuity because of compactness of $[0,T]$).
\end{enumerate}
\end{assumption}

\begin{observation}
Under this assumption, as in Observation \ref{obs:continuousextension}, we see that $H(\cdot,\cdot,t)$ has a unique continuous extension to $L^2(\mathcal{R}^\omega)_{sa}^{m+n}$.  Furthermore, for each $(x,y) \in L^2(\mathcal{R}^\omega)_{sa}^{m+n}$, the function $t \mapsto H(x,y,t)$ is continuous (though the modulus of continuity cannot be chosen independent of $(x,y)$).  Continuity follows because there exists a sequence $(x_n,y_n) \in (\mathcal{R}^\omega)_{sa}^{m+n}$ such that $(x_n,y_n) \to (x,y)$ in $\norm*{\cdot}_2$.  Now $H(x_n,y_n,\cdot)$ is continuous by assumption (3), but assumption (2) implies that $H(x_n,y_n,\cdot) \to H(x,y,\cdot)$ uniformly on $[0,T]$.
\end{observation}

Under these assumptions, \eqref{eq:IVP} can be solved by the standard method of Picard iteration.  We first verify that Assumption \ref{ass:vectorfield} is preserved under the composition and integration operations used to define Picard iterates.

\begin{lemma} \label{lem:vectorfieldintegration}
Suppose that $H(x,y,t)$ satisfies Assumption \ref{ass:vectorfield} and suppose that $G_0 \in (\overline{\TrP}_{m+n}^1)_{sa}^m$ is globally $\norm*{\cdot}_2$-Lipschitz.  Then the function
\[
G(x,y,t) = G_0(x,y) + \int_0^t H(x,y,s)\,ds
\]
is well-defined by Riemann integration and it also satisfies Assumption \ref{ass:vectorfield}.
\end{lemma}

\begin{proof}
The Riemann integral is defined because $t \mapsto H(x,y,t)$ is continuous with respect to $\norm*{\cdot}_2$ for each $(x,y) \in (\mathcal{R}^\omega)_{sa}^{m+n}$ (and in fact, each $(x,y) \in L^2(\mathcal{R}^\omega)_{sa}^{m+n}$).  Now let us check that $G$ satisfies Assumption \ref{ass:vectorfield}.

(1) Fix $R > 0$ and $\epsilon > 0$.  By assumption (2) for $H$, there exists $\delta > 0$ such that
\[
|t - t'| < \delta \implies \norm*{H(\cdot,\cdot,t) - H(\cdot,\cdot,t')}_{u,R} < \frac{\epsilon}{2T}.
\]
Fix $t$, then choose a partition $0 = t_0$, \dots, $t_n = t$ of $[0,t]$ such that $|t_j - t_{j-1}| < \delta$.  Then let $h_j \in (\TrP_{m+n}^1)_{sa}^m$ such that
\[
\norm*{h_j - H(\cdot,\cdot,t_j)}_{u,R} < \frac{\epsilon}{2T}.
\]
Then
\[
\norm*{h_j - H(\cdot,\cdot,s)}_{u,R} < \frac{\epsilon}{T} \text{ for all } s \in [t_{j-1},t_j].
\]
Therefore,
\[
\norm*{ \int_0^t H(\cdot,\cdot,s)\,ds - \sum_{j=1}^n (t_j - t_{j-1}) h_j }_{u,R} < \sum_{j=1}^n (t_j - t_{j-1}) \frac{\epsilon}{T} = \frac{\epsilon t}{T} \leq \epsilon.
\]
This shows that $\int_0^t H(\cdot,\cdot,s)\,ds$ is in $(\overline{\TrP}_{m+n}^1)_{sa}^m$.  Because $G_0$ is in this space as well, this implies that $G(\cdot,\cdot,t)$ is in $(\overline{\TrP}_{m+n}^1)_{sa}^m$ as desired.

(2) If $H(\cdot,\cdot,t)$ is $K$-Lipschitz for all $t$, then $\norm*{G(\cdot,t)}_{\Lip} \leq \norm*{G_0}_{\Lip} + tK$.

(3) Since $t \mapsto H(\cdot,\cdot,t)$ is continuous with respect to $\norm*{\cdot}_{u,R}$, we must have $\norm*{H(\cdot,\cdot,t)}_{u,R} \leq M$ for some constant $M$.  Then $\norm*{G(\cdot,\cdot,t) - G(\cdot,\cdot,t')}_{u,R} \leq M |t - t'|$.
\end{proof}

\begin{lemma} \label{lem:vectorfieldcomposition}
Suppose that $H(x,y,t)$ and $G(x,y,t)$ satisfy Assumption \ref{ass:vectorfield}.  Then $H(G(x,y,t),y,t)$ also satisfies Assumption \ref{ass:vectorfield}.
\end{lemma}

\begin{proof}
The composition makes sense because $H(x,y,t)$ extends to be defined for $(x,y) \in L^2(\mathcal{R}^\omega)_{sa}^{m+n}$.  It follows from Lemma \ref{lem:composition} that $H(G(x,y,t),y,t)$ satisfies (1).  The Lipschitz estimate (2) is straightforward and left to the reader.  To prove (3), let $K$ be a Lipschitz constant for $H$ as a function of $(x,y)$ that works for all $t$.  Fix $\epsilon > 0$.  Proceeding as in the proof of Lemma \ref{lem:vectorfieldintegration}, we can choose a partition $\{t_0,\dots,t_n\}$ of $[0,T]$ and $g_j \in (\overline{\TrP}_{m+n}^1)_{sa}^m$ such that
\[
\norm*{g_j - G(\cdot,\cdot,t)}_{u,R} < \frac{\epsilon}{4K} \text{ for } t \in [t_{j-1},t_j].
\]
Then there exists some $R'$ such that $\norm*{(x,y)}_\infty \leq R$ implies $\norm*{(g_j(x,y),y)}_\infty \leq R'$ for all $j$.  Then by applying assumption (3) to $H$, there exists $\delta$ such that
\[
|t - t'| < \delta \implies \norm*{H(\cdot,\cdot,t) - H(\cdot,\cdot,t')}_{u,R'} < \frac{\epsilon}{4}.
\]
We also choose $\delta'$ such that
\[
|t - t'| < \delta' \implies \norm*{G(\cdot,\cdot,t) - G(\cdot,\cdot,t')}_{u,R} < \frac{\epsilon}{4K}.
\]
Supposing that $\norm*{(x,y)}_\infty \leq R$ and $|t - t'| < \min(\delta,\delta')$, we have
\begin{align*}
&\norm*{H(G(x,y,t),y,t) - H(G(x,y,t'),y,t')}_2 \\
&\leq \norm*{H(G(x,y,t),y,t) - H(G(x,y,t'),y,t)}_2 + \norm*{H(G(x,y,t'),y,t) - H(G(x,y,t'),y,t')}_2 \\
&\leq K \norm*{G(x,y,t) - G(x,y,t')}_2 + \norm*{H(G(x,y,t'),y,t) - H(G(x,y,t'),y,t')}_2 \\
&\leq \frac{\epsilon}{4} + \norm*{H(G(x,y,t'),y,t) - H(G(x,y,t'),y,t')}_2.
\end{align*}
Meanwhile, after we pick $j$ such that $t' \in [t_{j-1},t_j]$, then
\begin{align*}
\norm*{H(G(x,y,t'),y,t) - H(G(x,y,t'),y,t')}_2 &\leq \norm*{H(G(x,y,t'),y,t) - H(g_j(x,y),y,t)}_2 \\
& \quad + \norm*{H(g_j(x,y),y,t) - H(g_j(x,y),y,t')}_2 \\
& \quad + \norm*{H(g_j(x,y),y,t) - H(G(x,y,t'),y,t')}_2.
\end{align*}
The middle term can be estimated by $\epsilon / 4$ because $\norm*{g_j(x,y),y)}_\infty \leq R'$.  Meanwhile, the first and third terms can each be estimated by $K (\epsilon / 4K) = \epsilon / 4$ using the Lipschitz property of $H$ and our choice of $g_j$.  Altogether, $|t - t'| < \min(\delta,\delta')$ implies that $\norm*{H(G(x,y,t),y,t) - H(G(x,y,t'),y,t')}_2 < \epsilon$ whenever $\norm*{(x,y)}_\infty \leq R$.
\end{proof}

\begin{proposition} \label{prop:ODE}
Let $H(x,y,t)$ satisfy Assumption \ref{ass:vectorfield} and let $G_0 \in (\overline{\TrP}_{m+n}^1)_{sa}^m$.  Then there exists a unique continuous $F: L^2(\mathcal{R}^\omega)_{sa}^{m+n} \times [0,T] \to L^2(\mathcal{R}^\omega)_{sa}^m$ satisfying
\[
F(x,y,t) = G_0(x,y) + \int_0^t H(F(x,y,s),y,s)\,ds.
\]
Moreover, $F(x,y,t)$ also satisfies Assumption \ref{ass:vectorfield}.
\end{proposition}

\begin{proof}
We define the Picard iterates $F_\ell$ inductively by
\begin{align*}
F_0(x,y,t) &= G_0(x,y) \\
F_{\ell+1}(x,y,t) &= G_0(x,y) + \int_0^t H(F_\ell(x,y,s),y,s)\,ds.
\end{align*}
The previous two lemmas imply that $F_k$ is well-defined and satisfies Assumption \ref{ass:vectorfield}.  Convergence of the Picard iterates follows from the standard proof of Picard-Lindel{\"o}f.  Briefly, given that $H$ is $K$-Lipschitz in $(x,y)$ with respect to $\norm*{\cdot}_2$, we have
\[
\norm*{F_{\ell+1}(x,y,t) - F_\ell(x,y,t)}_2 \leq K \int_0^t \norm*{F_\ell(x,y,s) - F_{\ell-1}(x,y,s)}_2\,ds.
\]
Also, we have
\[
\norm*{F_1(x,y,t) - F_0(x,y,t)}_2 \leq t M(x,y).
\]
where $M(x,y) = \sup_{s \in [0,T]} \norm*{H(G_0(x,y),y,s)}_2$, which is finite because of continuity of $H(G_0(x,y),y,t)$ in $t$.  From here a straightforward induction on $\ell$ shows that for $\ell \geq 1$,
\[
\norm*{F_{\ell}(x,y,t) - F_{\ell-1}(x,y,t)}_2 \leq \frac{K^{\ell-1} t^\ell}{\ell !}
\]
because $K \int_0^t K^{\ell-1} s^\ell / \ell!\,ds = K^\ell s^{\ell+1} / (\ell + 1)!$.  Now because $\sum_{\ell=1}^\infty K^{\ell - 1} s^\ell / \ell!$ converges, we know that
\[
F(x,y,t) := \lim_{\ell \to \infty} F_\ell(x,y,t) \text{ exists,}
\]
and
\[
\norm*{F_\ell(x,y,t) - F(x,y,t)}_2 \leq M(x,y) \sum_{j=\ell+1}^\infty \frac{K^{j-1} t^j}{j!}.
\]
The fact that $F(x,y,t)$ satisfies the integral equation is straightforward, and the proof of the uniqueness of this $F$ is also standard.

It remains to show that $F$ satisfies Assumption \ref{ass:vectorfield}.  First, recall that $H(G_0(x,y),y,t)$ is Lipschitz in $(x,y)$ uniformly for all $t$.  If $K'$ is a Lipschitz constant for this function, then
\[
M(x,y) \leq M(0,0) + K' \norm*{(x,y)}_2.
\]
In particular,
\[
\norm*{F_\ell(x,y,t) - F(x,y,t)}_2 \leq (M(0,0) + K' \norm*{(x,y)}_2) \sum_{j=\ell+1}^\infty \frac{K^{j-1} t^j}{j!}.
\]
This implies that the convergence of $F_\ell$ to $F$ occurs uniformly for $(x,y)$ with $\norm*{(x,y)}_\infty \leq R$ and all $t \in [0,T]$.  Then because $F_\ell(\cdot,\cdot,t)$ can be approximated in $\norm*{\cdot}_{u,R}$ by trace polynomials, the same must be true for $F(\cdot,\cdot,t)$ for each $t$, which shows that $F$ satisfies (1).  Similarly, because of the uniform convergence of $F_\ell$ to $F$ for $\norm*{(x,y)}_\infty \leq R$ and $t \in [0,T]$, the uniform continuity property (3) for $F$ follows from property (3) for $F_\ell$.

Finally, we must show (2) that $F$ is Lipschitz in $(x,y)$.  More precisely, we claim that
\[
\norm*{F(x,y,t) - F(x',y',t)}_2 \leq e^{Kt} \norm*{G_0(x,y) - G_0(x',y')}_2 + (e^{Kt} - 1) \norm*{y - y'}_2.
\]
Now it suffices to check that each Picard iterate $F_\ell$ satisfies this estimate.  This can be verified by induction on $\ell$.  The base case $F_0(x,y,t) = G_0(x,y)$ is immediate.  For the induction step, we observe that
\begin{align*}
& \norm*{F_{\ell+1}(x,y,t) - F_\ell(x',y',t)}_2 \\
&\leq \norm*{G_0(x,y) - G_0(x',y')}_2 + \int \norm*{H(F_\ell(x,y,s),y,s) - H(F_\ell(x',y',s),y',s)}_2\,ds \\
&\leq \norm*{G_0(x,y) - G_0(x',y')}_2 + \int K \left( \norm*{F_\ell(x,y,s) - F_\ell(x',y',s)}_2 + \norm*{y - y'}_2 \right) \,ds,
\end{align*}
using the fact that $H$ is $K$-Lipschitz.  Then we plug in our induction hypothesis that $\norm*{F_\ell(x,y,s) - F_\ell(x',y',s)}_2$ is bounded by $e^{Kt} \norm*{G_0(x,y) - G_0(x',y')}_2 + (e^{Kt} - 1) \norm*{y - y'}_2$, and then directly evaluate the integral to close the induction.
\end{proof}

We have now shown that it makes sense to solve ODE for tuples in $(\overline{\TrP}_m^1)_{sa}$.  There is a parallel list of results which instead deal with functions on $N \times N$ matrices that are asymptotically approximable as $N \to \infty$.  We use the following assumptions.

\begin{assumption} \label{ass:vectorfield2}
We are given $T > 0$ and for each $N \in \N$ a function $H^{(N)}: M_N(\C)_{sa}^m \times M_N(\C)_{sa}^n \times [0,T] \to M_N(\C)_{sa}^m$ such that
\begin{enumerate}
	\item For each $t$, there exists $H(\cdot,\cdot,t) \in (\overline{\TrP}_{m+n}^1)_{sa}^m$ such that $H^{(N)}(\cdot,\cdot,t) \rightsquigarrow H(\cdot,\cdot,t)$.
	\item $H^{(N)}$ is $\norm*{\cdot}_2$-Lipschitz in $(x,y)$ with some Lipschitz constant $K$ independent of $t$ and $N$.
	\item For every $R > 0$ and for every $\epsilon > 0$, there exists $\delta > 0$, such that
	\[
	|t - t'| < \delta \implies \norm*{H^{(N)}(\cdot,\cdot,t) - H^{(N)}(\cdot,\cdot,t')}_{u,R}^{(N)} < \epsilon \text{ for all } t, t' \in [0,T] \text{ for all } N.
	\]
\end{enumerate}
\end{assumption}

\begin{proposition} \label{prop:ODE2}
Let $\{H^{(N)}\}$ satisfy Assumption \ref{ass:vectorfield2}, and let $G_0^{(N)}: M_N(\C)_{sa}^{m+n} \to M_N(\C)_{sa}^m$ be asymptotically approximable such that $G_0^{(N)} \rightsquigarrow G_0$ and $G_0^{(N)}$ is $\norm*{\cdot}_2$ Lipschitz uniformly in $N$.  Then for each $N$ there is a unique $F^{(N)}: M_N(\C)_{sa}^{m+n} \times [0,T] \to M_N(\C)_{sa}^m$ satisfying
\[
F^{(N)}(x,y,t) = G_0^{(N)}(x,y) + \int_0^t H^{(N)}(F^{(N)}(x,y,s),y,s)\,ds.
\]
Moreover, $\{F^{(N)}\}$ also satisfies Assumption \ref{ass:vectorfield}.  Furthermore, the vector field $H$ such that $H^{(N)}(\cdot, \cdot, t) \rightsquigarrow H(\cdot,\cdot,t)$ satisfies Assumption \ref{ass:vectorfield}, and we have $F^{(N)} \rightsquigarrow F$ where $F$ is the solution given by Proposition \ref{prop:ODE}.
\end{proposition}

\begin{proof}
The proof of existence and uniqueness of the solution is almost identical to that of Proposition \ref{prop:ODE}.  First, one shows that Assumption \ref{ass:vectorfield2} is preserved under integration and composition (analogous to Lemma \ref{lem:vectorfieldintegration} and \ref{lem:vectorfieldcomposition}).  Then exactly as in the proof of Proposition \ref{prop:ODE}, one defines Picard iterates, proves they converge, establishes Lipschitz bounds, and checks they satisfy Assumption \ref{ass:vectorfield2}.  The one additional feature in these proofs is to make all the estimates uniform in $N$.  For instance, the quantity $M(x,y)$ in the proof of Proposition \ref{prop:ODE} is replaced by
\[
M^{(N)}(x,y) = \sup_{s \in [0,T]} \norm*{H^{(N)}(G_0^{(N)}(x,y),y,s)}_2.
\]
Then $H^{(N)}(G_0^{(N)}(x,y),y,t)$ has some Lipschitz constant $K'$ independent of $N$, and
\[
M^{(N)}(x,y) \leq M^{(N)}(0,0) + K' \norm*{(x,y)}_2.
\]
But then we can show that $\sup_N M^{(N)}(0,0)$ is finite.  This is because if $\Phi^{(N)}(x,y,t) = H^{(N)}(G_0^{(N)}(x,y),y,t)$, then $\sup_N \sup_t \norm*{\Phi^{(N)}(\cdot,\cdot,t)}_{u,R}^{(N)}$ is finite because of Assumption \ref{ass:vectorfield} (3) and the fact that $\Phi^{(N)}(x,y,0)$ is asymptotically approximable and hence bounded in $\norm*{\cdot}_{u,R}^{(N)}$ as $N \to \infty$.

Now the fact that $H$ satisfies Assumption \ref{ass:vectorfield} is a straightforward limiting argument.  The key ingredient is that if $f^{(N)} \rightsquigarrow f$, then $\norm*{f}_{u,R} = \lim_{N \to \infty} \norm*{f^{(N)}}_{u,R}^{(N)}$.

Finally, to show that $F^{(N)} \rightsquigarrow F$, it suffices to show that for each of the Picard iterates $F_\ell^{(N)} \rightsquigarrow F_\ell$ because of the uniform convergence of $F_\ell^{(N)} \to F^{(N)}$ as $\ell \to \infty$ for $\norm*{(x,y)}_\infty \leq R$, where the rate of convergence is also independent of $N$.  Furthermore, since the Picard iterates are defined inductively by composition and integration, it suffices to show that the asymptotic approximation relation $\rightsquigarrow$ is preserved by these operations.  Preservation under integration follows because the integrals can be approximated by Riemann sums and this approximation is uniformly good for $\norm*{(x,y)}_\infty \leq R$ and for all $N$ because of the uniform continuity Assumption \ref{ass:vectorfield2} (3).  Preservation under composition follows from Lemma \ref{lem:composition}.
\end{proof}

\subsection{The Heat Semigroup} \label{subsec:heatsemigroup}

Recall that the solution to the classical heat equation is given by convolution the heat kernel (which is given by a Gaussian probability density).  In particular, let $\sigma_{m,t}^{(N)}$ be the probability distribution of an $m$-tuple of independent GUE matrices $(S_1^{(N)}, \dots, S_m^{(N)})$ such that $E[\tau_N[(S_j^{(N)})^2]] = t$, which is given by density $(1/Z^{(N)}) e^{-\norm*{x}_2^2 / 2t}\,dx$.  If $u_0: M_N(\C)_{sa}^m \to \C$, then $u_t := u_0 * \sigma_t^{(N)}$ solves the normalized heat equation
\[
\partial_t u_t = \frac{1}{2N} \Delta u_t.
\]
Here $u_0 * \sigma_{m,t}^{(N)}$ is meant in the sense of convolving a function with a measure, and this is the same as convolving of $u_0$ with the density function for $\sigma_{m,t}^{(N)}$.  The meaning of $\Delta$ is to be interpreted using coordinates with respect to some orthonormal basis of $M_N(\C)_{sa}$ in the inner product $\ip{x,y} = \Tr(xy)$; this is \emph{not} the same as differentiating entrywise since some of the entries are real and some are complex.

Our goal is to describe the large $N$ behavior of $u^{(N)} * \sigma_{m,t}^{(N)}$ when $\{u^{(N)}\}$ is asymptotically approximable by trace polynomials, and to define ``$u \boxplus \sigma_{m,t}$'' when $u \in \overline{\TrP}_m^j$.

In \cite[\S 3.2 and 3.3]{Jekel2018}, using similar methods to \cite{Cebron2013}, we explained the computation of $(1/N) \Delta f$ as a function on $M_N(\C)_{sa}^m$ when $f \in \TrP_m^0$ or $\TrP_m^1$.  More precisely, let $\Delta_j f(x_1,\dots,x_m)$ denote the Laplacian with respect to the coordinates of the matrix $x_j$.  We found that for $j = 1, \dots, m$ there are linear maps $L_j^{(N)}, L_j: \TrP_m^0 \to \TrP_m^0$ defined purely algebraically, such that $(1/N) \Delta_j f = L_j^{(N)} f$ when $f$ is viewed as a function on $M_N(\C)_{sa}^m$, $L_j^{(N)}$ and $L_j$ do not increase the degree of a trace polynomial, and $\lim_{N \to \infty} L_j^{(N)} f = L_j f$ coefficient-wise.

A similar analysis holds for the Laplacian of $f \in \TrP_m^1$ viewed as a function $M_N(\C)_{sa}^m \to M_N(\C)$.  Here we follow the standard convention of using the same symbol $\Delta$ for the Laplacians of vector-valued functions as for the Laplacians of scalar-valued functions; thus, the reader must be careful to distinguish scalar-valued and vector-valued functions based on context.  We saw that there were linear transformations $L_j^{(N)}, L_j: \TrP_m^1 \to \TrP_m^1$ such that $(1/N) \Delta_j f = L_j^{(N)}f$ as a function on matrices, $L_j^{(N)}$ and $L_j$ do not increase degree, and $L_j^{(N)} f \to L_j f$ coefficient-wise.

We deduced as a consequence that $e^{L^{(N)}t/2} f = f * \sigma_{m,t}^{(N)}$ has a well-defined large $N$ limit if $f$ is a trace polynomial \cite[Lemma 3.21]{Jekel2018}, and that if $\{u^{(N)}\}$ is asymptotically approximable by trace polynomials, then so is $\{u^{(N)} * \sigma_{m,t}^{(N)}\}$ \cite[Lemma 3.28]{Jekel2018}.

In order to establish ``conditional versions'' of our earlier results, we must consider trace polynomials $f(x_1,\dots,x_m, y_1, \dots, y_n)$ in $m + n$ variables and take the Laplacian with respect to $x = (x_1,\dots,x_m)$ while treating $y = (y_1,\dots,y_n)$ as an auxiliary parameter.  We denote by $\Delta_x = \sum_{j=1}^m \Delta_{x_j}$, $L_x^{(N)} = \sum_{j=1}^m L_{x_j}^{(N)}$, and $L_x = \sum_{j=1}^m L_{x_j}$ the various Laplacian operators with respect to $x$.

Because $L_x^{(N)}$ and $L_x$ map the finite-dimensional vector space trace polynomials of degree $\leq d$ into itself, there are well-defined linear operators $e^{tL_x^{(N)} / 2}$ and $e^{tL_x / 2}$ on the space of trace polynomials in $\TrP_{m+n}^j$ of degree $\leq d$ for each $j = 0, 1$ each $d \in \N$, and each real $t \geq 0$.  Since trace polynomials are the union of the subspaces of trace polynomials with degree $\leq d$, there are linear operators $e^{tL_x^{(N)}/2}, e^{tL_x / 2}: \TrP_{m+n}^j \to \TrP_{m+n}^j$.  Moreover, these operators form a semigroup, and they satisfy the following property, which is an extension of \cite[Theorem 2.4]{Cebron2013} to the spaces $\overline{\TrP}_m^j$.

\begin{lemma}
Let $(X,Y)$ be a random variable in $M_N(\C)_{sa}^{m+n}$ with finite moments, and let $S \sim \sigma_{m,t}^{(N)}$ be an independent GUE random variable in $M_N(\C)_{sa}^m$.  Then we have
\begin{equation} \label{eq:Gaussianconditionalexpectation}
E[f(X+S,Y) | (X,Y)] = [e^{tL_x^{(N)}/2} f](X,Y) \text{ for } f \in \TrP_{m+n}^0 \text{ or } f \in \TrP_{m+n}^1.
\end{equation}
Similarly, suppose that $(X,Y)$ is a tuple of self-adjoint non-commutative random variables, and let $S$ be a freely independent tuple with non-commutative law $\sigma_{m,t}$.  Then
\begin{equation} \label{eq:Gaussianconditionalexpectation2}
f(X+S,Y) = [e^{tL_x/2} f](X,Y) \text{ for } f \in \TrP_m^0,
\end{equation}
and
\begin{equation} \label{eq:Gaussianconditionalexpectation3}
E_{\mathrm{W}^*(X,Y)}[f(X+S,Y)] = [e^{tL_x/2} f](X,Y) \text{ for } f \in \TrP_m^1,
\end{equation}
where $E_{\mathrm{W}^*(X,Y)}: \mathrm{W}^*(X,Y,S) \to \mathrm{W}^*(X,Y)$ is the unique trace-preserving conditional expectation.
\end{lemma}

\begin{proof}
Since $S$ is independent and distributed according to $\sigma_{m,t}^{(N)}$, we have
\[
E[f(X+S,Y) | (X,Y)] = \int f(X + z,Y)\,d\sigma_{m,t}^{(N)}(z).
\]
On the other hand, for $(x,y) \in M_N(\C)_{sa}^{m+n}$,
\[
\int f(x + z,y)\,d\sigma_{m,t}^{(N)}(z) = e^{tL_x^{(N)}/2} f(x,y),
\]
because both sides are the solution to the heat equation on the space of coordinate-wise polynomials on $M_N(\C)_{sa}^{m+n}$ of degree $\leq d$.  This shows \eqref{eq:Gaussianconditionalexpectation}.

To prove the free versions, we assume familiarity with the results of free probability (see e.g.\ \cite{VDN1992}, \cite{NS2006}, \cite[Chapter 5]{AGZ2009}). Suppose that $(X,Y)$ are non-commutative random variables and $S_t$ is a freely independent free semicircular $m$-tuple with law $\sigma_{m,t}^{(N)}$.  We may assume that $(S_t)_{t \geq 0}$ is a free Brownian motion, so that $S_t - S_s \sim S_{t - s}$ for $0 \leq s \leq t$ and $S_t \sim t^{1/2} S_1$.  Note that $e^{-tL_x/2}$ is a well-defined operator on trace polynomials.  To prove \eqref{eq:Gaussianconditionalexpectation2}, it suffices to show that $[e^{tL_x/2}f](X+S_t,Y) = f(X,Y)$ for $f \in \TrP_m^0$.  This will follow if we check that
\[
\frac{d}{dt} \left( [e^{-tL_x/2} f](X+S_t,Y) \right) = 0.
\]
From a free probabilistic computation sketched in \cite[Lemma 3.23]{Jekel2018}, we have
\[
\frac{d}{dt} f(X+S_t,Y) = \frac{1}{2} [L_x f](X+S_t,Y),
\]
and hence
\begin{align*}
\frac{d}{dt} \left( [e^{-tL_x/2} f](X+S_t,Y) \right) &= \frac{d}{dt_1} [e^{-t_1 L_x/2} f](X+S_{t_2},Y)|_{t_1=t_2=t} + \frac{d}{dt_2} [e^{-t_1 L_x/2} f](X+S_{t_2},Y)|_{t_1=t_2=t} \\
&= \left[ \frac{-L_x}{2} e^{-tL_x/2} f \right](X+S_t,Y) + \left[ \frac{L_x}{2} e^{-tL_x/2} f \right](X+S_t,Y) \\
&= 0.
\end{align*}

Next, to prove \eqref{eq:Gaussianconditionalexpectation3}, it suffices to show that for $g \in \TrP_{m+n}^1$, we have
\[
\tau(f(X+S_t,Y) g(X,Y)) = \tau([e^{tL_x/2}f](X,Y) g(X,Y)),
\]
since functions of the form $g(X,Y)$ for $g \in \TrP_m^1$ are dense in $L^2(\mathrm{W}^*(X,Y))$.  Consider the function $F \in \TrP_{m+n+m}$ given by $F(x,y,x') = \tau(f(x,y) g(x',y))$.  Notice that
\[
L_x F(x,y,x') = \tau(L_x[f(x,y) g(x',y)]) = \tau([L_x f(x,y)] g(x',y)).
\]
Here the first equality is checked directly from the definition of the Laplacian \cite[see Def.\ 3.13 and 3.16, proof of Lemma 3.18]{Jekel2018}.  The equality $L_x[f(x,y) g(x',y)] = L_x[f(x,y)] g(x',y)$ again is checked from the definition of the Laplacian; this equality is intuitive since $g(x',y)$ is independent of $x$.  Since the same reasoning may be applied to compute the Laplacian $L_x$ of $\tau([e^{tL_x/2}f](x,y) g(x,y))$, we have
\[
e^{tL_x/2} F(x,y,x') = \tau([e^{tL_x/2}f](x,y) g(x',y)).
\]
We can view $F(x,y,x')$ as a function of the $m$-tuple $x$ and the $(n+m)$-tuple $(y,x')$, that is, an element of $\TrP_{m+(n+m)}^1$.  We apply \eqref{eq:Gaussianconditionalexpectation2} to $f$ and the pair $(X, (Y,X))$ and obtain
\[
F(X+S_t,Y,X) = e^{tL_x/2}F(X,Y,X)
\]
which means precisely that
\[
\tau(f(X+S_t,Y) g(X,Y)) = \tau([e^{tL_x/2}f](X,Y) g(X,Y)),
\]
which completes the proof of \eqref{eq:Gaussianconditionalexpectation3}.
\end{proof}

\begin{remark}
The free conditional expectation formulas \eqref{eq:Gaussianconditionalexpectation2} and \eqref{eq:Gaussianconditionalexpectation3} could also be proved using random matrices provided that $\mathrm{W}^*(X,Y)$ is $\mathcal{R}^\omega$-embeddable.  Indeed, let $(X^{(N)},Y^{(N)})$ be (deterministic) tuples of matrices with non-commutative laws converging to the law of $(X,Y)$ and let $S^{(N)} \sim \sigma_{m,t}^{(N)}$.  Then to prove \eqref{eq:Gaussianconditionalexpectation2} for instance, we could use the fact that $E[f(X^{(N)}+S^{(N)},Y^{(N)}) = [e^{tL_x^{(N)}/2}f](X^{(N)},Y^{(N)})$ and take the limit as $N \to \infty$ using Voiculescu's theorem on asymptotic freeness \cite[Theorem 2.2]{Voiculescu1998}.  A similar proof could be done for \eqref{eq:Gaussianconditionalexpectation3}.
\end{remark}

\begin{lemma}
If $f \in \TrP_{m+n}^j$ for $j = 0, 1$, then we have $\norm*{e^{tL_x/2} f}_{u,R} \leq \norm*{f}_{u,R+2t^{1/2}}$ for $t \geq 0$.  In particular, $f \mapsto e^{tL_x/2} f$ extends to a unique continuous linear operator $\overline{\TrP}_{m+n}^j \to \overline{\TrP}_{m+n}^j$.
\end{lemma}

\begin{proof}
Let $(X,Y) \in (\mathcal{R}^\omega)_{sa}^{m+n}$ with $\norm*{(X,Y)}_\infty \leq R$.  Let $S \sim \sigma_{m,t}$ be a freely independent semicircular tuple.  If $f \in \TrP_{m+n}^0$, then
\[
[e^{tL_x/2} f](X,Y) = f(X+S,Y).
\]
Since $\norm*{S}_\infty = 2t^{1/2}$, we have $\norm*{(X+S,Y)}_\infty \leq R + 2t^{1/2}$.  Therefore, $\norm*{e^{-tL_x/2} f}_{u,R} \leq \norm*{f}_{u,R+2t^{1/2}}$ as desired.  Similarly, if $f \in \TrP_{m+n}^1$, then we check $\norm*{e^{-tL_x/2} f}_{u,R} \leq \norm*{f}_{u,R+2t^{1/2}}$ using the conditional expectation formula \eqref{eq:Gaussianconditionalexpectation3}.  Now the continuous extension of $e^{tL_x/2}$ to $\overline{\TrP}_m^j$ is immediate.
\end{proof}

The semigroup $e^{tL_x/2}$ acting on $\overline{\TrP}_{m+n}^1$ describes the large $N$ limit of the Gaussian convolution semigroup on $M_N(\C)_{sa}^N$ defined as follows.

\begin{definition}
For $f: M_N(\C)_{sa}^{m+n} \to \C$ or $M_N(\C)$, we denote
\[
P_t^{(N)} f(x,y) = \int f(x+z,y)\,d\sigma_{m,t}^{(N)}(z).
\]
Moreover, we denote by $P_t^{\TrP}: \overline{\TrP}_m^j \to \overline{\TrP}_m^j$ the continuous extension of $e^{tL_x/2}$.
\end{definition}

\begin{lemma} \label{lem:Gaussianconvolution}
Suppose that $f^{(N)}: M_N(\C)_{sa}^{m+n} \to \C$ is asymptotically approximable by trace polynomials and $f^{(N)} \rightsquigarrow f \in \overline{\TrP}_m^0$.  Furthermore, assume that for some $A, B > 0$ and $k \in \N$, we have
\begin{equation} \label{eq:growthbound}
\norm*{f}_{u,R} \leq A + B R^k \\
\sup_{\substack{x \in M_N(\C)_{sa}^m \\ \norm*{x}_\infty \leq R}} |f^{(N)}(x)| \leq A + BR^k.
\end{equation}
Then $P_t^{(N)} f^{(N)} \rightsquigarrow P_t^{\TrP} f$.  The same holds for $f^{(N)}: M_N(\C)_{sa}^{m+n} \to M_N(\C)$ and $f \in \overline{\TrP}_m^1$ with $|f^{(N)}(x)|$ replaced by $\norm*{f^{(N)}(x)}_2$.
\end{lemma}

The proof of this lemma is the same as in \cite[Lemma 3.28]{Jekel2018}.

\begin{remark} \label{rem:UCgrowthbound}
In both the scalar-valued and matrix-valued cases, the assumption \eqref{eq:growthbound} holds automatically with $k = 1$ provided that $f^{(N)}$ and $f$ are $\norm*{\cdot}_2$-uniformly continuous (with modulus of continuity independent of $N$).  Let us focus on the matrix-valued case of $\TrP_m^1$, there exists $\delta > 0$ such that
\[
\norm*{x-y}_2 < \delta \implies \norm*{f(x) - f(y)}_2 \leq 1.
\]
In particular, given $x \in (\mathcal{R}^\omega)_{sa}^m$, we can choose an integer $j$ such that $j \delta < \norm*{x}_2 \leq 2 j \delta$.  Then we have
\[
\norm*{f(x) - f(0)}_2 \leq \sum_{i=1}^{2j} \norm*{f(ix/2k) - f((i-1)x/2k)}_2 \leq 2j \leq 2 \norm*{x}_2 / \delta.
\]
Thus,
\[
\norm*{f(x)}_2 \leq \norm*{f(0)}_2 + \frac{2}{\delta} \norm*{x}_2 \leq \norm*{f(0)}_2 + \frac{2m^{1/2}}{\delta} \norm*{x}_\infty,
\]
which implies the first estimate of \eqref{eq:growthbound}.  The case for $f^{(N)}$ is handled similarly, and we note that $\norm*{f^{(N)}(0)}_2$ is bounded as $N \to \infty$ because of our assumption that $f^{(N)} \rightsquigarrow f$.  The same argument works in the case of scalar-valued functions and $f \in \overline{\TrP}_m^0$.
\end{remark}

\section{Conditional Expectation for Free Gibbs States} \label{sec:conditionalexpectation}

\subsection{Free Gibbs States from Convex Potentials} \label{subsec:CEmotivation}

In \cite{Jekel2018} and in the present work, we focus on the following situation:

\begin{assumption} \label{ass:convexRMM}
We are given $0 < c \leq C$ and $V^{(N)}: M_N(\C)_{sa}^m \to \R$ such that
\begin{enumerate}
	\item $HV^{(N)} \geq c$, that is, $V^{(N)}(x) - \frac{1}{2} c \norm*{x}_2^2$ is convex.
	\item $HV^{(N)} \leq C$, that is, $V^{(N)}(x) - \frac{1}{2} C \norm*{x}_2^2$ is concave.
	\item $\{DV^{(N)}\}_{N \in \N}$ is asymptotically approximable by trace polynomials.
\end{enumerate}
We denote by $\mu^{(N)}$ the probability measure on $M_N(\C)_{sa}^m$ given by
\[
d\mu^{(N)}(x) = e^{-N^2 V^{(N)}(x)}\,dx.
\]
Furthermore, we assume that the mean $\int x_j \,d\mu^{(N)}(x)$ is a scalar multiple of the identity matrix.
\end{assumption}

The following was proved in \cite[Theorem 4.1]{Jekel2018}.

\begin{theorem} \label{thm:freeGibbslaw}
Let $V^{(N)}$ and $\mu^{(N)}$ be as in Assumption \ref{ass:convexRMM}.  Then there exists a non-commutative law $\lambda$ such that for every non-commutative polynomial $p$, we have
\[
\lambda(p) = \lim_{N \to \infty} \int \tau_N(p(x))\,d\mu^{(N)}(x).
\]
Moreover, we have for every $R > 0$ and $\epsilon > 0$ that
\[
\lim_{N \to \infty} \frac{1}{N^2} \log \mu_N\left(\left\{x \in M_N(\C)_{sa}^m: \norm*{x}_\infty \leq R, |\tau(p(x)) - \lambda(p)| > \epsilon \right\}\right) < 0
\]
\end{theorem}

\begin{corollary} \label{cor:convergenceofexpectation}
Let $\mu^{(N)}$ and $\lambda$ be as in Theorem \ref{thm:freeGibbslaw}. Let $X^{(N)}$ be a random $m$-tuple of matrices distributed according to $\mu^{(N)}$ and let $X$ be a non-commutative random $m$-tuple distributed according to $\lambda$.  Let $f^{(N)}, g^{(N)}: M_N(\C)_{sa}^m \to M_N(\C)$.  Suppose there are constants $A$ and $B > 0$ and $k \in \N$ such that
\[
\max(\norm*{f^{(N)}(x)}_2, \norm*{g^{(N)}(x)}_2) \leq A + B \norm*{x}_\infty^k
\]
Suppose that $f^{(N)} \rightsquigarrow f$ and $g^{(N)} \rightsquigarrow g$ where $f, g \in \overline{\TrP}_m^1$.  Then
\[
\lim_{N \to \infty} E[\tau_N(f^{(N)}(X^{(N)}) g^{(N)}(X^{(N)}))] = \tau(f(X) g(X)).
\]
\end{corollary}

\begin{proof}
Let $a_j^{(N)} = E[X_j^{(N)}]$ which we assumed to be a scalar multiple of the identity, and which we know has a limit as $N \to \infty$.  By Lemma \ref{lem:epsilonnet}, we have
\[
P(\norm*{X_j^{(N)} - a_j^{(N)}}_\infty \geq c^{-1/2} \Theta + \delta) \leq e^{-cN \delta^2 / 2}.
\]
In particular, letting $R > \sup_{N,j} |a_j^{(N)}| + c^{-1/2} \Theta$, we have
\[
P(\norm*{X^{(N)}}_\infty \geq R) \to 0
\]
and
\[
E[\mathbf{1}_{\norm*{X^{(N)}}_\infty \geq R} \tau_N(f^{(N)}(X^{(N)}) g^{(N)}(X^{(N)}))] \to 0.
\]
Therefore, in order to prove convergence of the expectation, it suffices to check that $\tau_N(f^{(N)}(X^{(N)}) g^{(N)}(X^{(N)}))$ converges in probability to $\tau(f(X) g(X))$.

We already know that $\tau_N(p(X^{(N)}))$ converges to $\tau(p(X))$ in probability for every non-commutative polynomial $p$.  It follows that if $u$ is a scalar-valued trace polynomial, then $u(X^{(N)}) \to u(X)$ in probability.  This also holds for $u \in \overline{\TrP}_m^1$; indeed, we know that $\norm*{X^{(N)}}_\infty \leq R$ with probability tending to $1$ and $\norm*{X}_\infty \leq R$, whereas $u$ can be approximated in $\norm*{\cdot}_{u,R}$ by trace polynomials.  Finally, if $u^{(N)}$ is a sequence of scalar-valued function such that $u^{(N)} \rightsquigarrow u \in \overline{\TrP}_m^0$, then $u^{(N)}(X^{(N)}) - u(X^{(N)})$ converges to $0$ in probability, and hence $u^{(N)}(X^{(N)})$ converges in probability to $u(X)$.  By Lemma \ref{lem:algebra}, we can apply this statement to $u^{(N)} = \tau_N(f^{(N)} g^{(N)})$ and $u = \tau(fg)$, which completes the argument.
\end{proof}

\begin{definition}
Let $V \in \overline{\TrP}_m^0$ and suppose $V$ extends to a function $L^2(\mathcal{R}^\omega)_{sa}^m \to \R$ such that $V(x) - (c/2) \norm*{x}_2^2$ is convex and $V(x) - (C/2) \norm*{x}_2^2$ is concave. In this case, $V$ is differentiable as a function on the real Hilbert space $L^2(\mathcal{R}^\omega)_{sa}^m$, as a consequence of the existence of supporting hyperplanes for convex functions on a Hilbert space.  If we assume also that $DV \in \overline{\TrP}_m^1$, then we say that $V \in \mathcal{E}_m^{\TrP}(c,C)$.
\end{definition}

\begin{remark}
We did not prove or assume that the trace polynomials which approximate $DV$ are the gradients of the \emph{same} trace polynomials that approximate $V$.  Thus, this definition is technically different from that of \cite[\S 8.2]{Jekel2018}.
\end{remark}

\begin{definition}
If $V \in \mathcal{E}_m^{\TrP}(c,C)$, then we may define $V^{(N)} = V|_{M_N(\C)_{sa}^m}$, and in this case $DV^{(N)} = DV|_{M_N(\C)_{sa}^m}$.  Clearly, $DV^{(N)}$ is asymptotically approximable by trace polynomials, and so by Theorem \ref{thm:freeGibbslaw}, there exists a non-commutative law $\lambda_V$ that arises as the large $N$ limit of the associated random matrix models.  Furthermore, the limiting free Gibbs law $\lambda_V$ only depends on $V$, that is, every approximating sequence of functions $V^{(N)} \in \mathcal{E}_m^{(N)}(c,C)$ will produce the same free Gibbs law (see \cite[\S 8.2]{Jekel2018}).  We call $\lambda_V$ the \emph{free Gibbs state given by potential $V$}.
\end{definition}

\begin{remark}
One can check that if $V^{(N)}$ is as in Assumption \ref{ass:convexRMM}, then there exists a $V \in \mathcal{E}_m^{\TrP}(c,C)$ such that $V^{(N)} \rightsquigarrow V$ and $DV^{(N)} \rightsquigarrow DV$.  Thus, the non-commutative laws that arise from these random matrix models are precisely $\lambda_V$ for $V \in \mathcal{E}_m^{\TrP}(c,C)$.
\end{remark}

\begin{remark} \label{rem:unitaryinvariance}
Since $\lambda_V$ is independent of the choice of approximating sequence $V^{(N)}$, we can in particular take $V^{(N)} = V|_{M_N(\C)_{sa}^m}$, which produces a canonical unitarily invariant sequence of random matrices models.
\end{remark}

\subsection{Main Result on Conditional Expectation}

Our main result in this section is in some sense a generalization of \cite[Theorem 4.1]{Jekel2018}, which deals with conditional expectations rather than expectations.  The proof of the earlier theorem was reduced to the following statement:  Suppose $V^{(N)}$ satisfies Assumption \ref{ass:convexRMM} and that $u^{(N)}: M_N(\C)_{sa}^m \to \C$ is $\norm*{\cdot}_2$-Lipschitz (uniformly in $N$) and asymptotically approximable by trace polynomials.  Then
\[
\lim_{N \to \infty} \int u^{(N)}\,d\mu^{(N)} \text{ exists.}
\]
Now, our goal is to prove the following.

\begin{theorem} \label{thm:conditionalexpectation}
Consider functions $V^{(N)}: M_N(\C)_{sa}^{m+n} \to \R$, denoted as $V^{(N)}(x,y)$, which satisfy Assumption \ref{ass:convexRMM} as functions of $(x,y)$.  Let $\mu^{(N)}$ be the associated probability measure on $M_N(\C)_{sa}^{m+n}$.  Let $(X^{(N)},Y^{(N)})$ be an $(m+n)$-tuple of random matrices distributed according to $\mu^{(N)}$, and let $(X,Y)$ be a $(m+n)$-tuple of non-commutative random variables distributed according to the limiting free Gibbs law $\lambda$ given by Theorem \ref{thm:freeGibbslaw}

Let $f^{(N)}: M_N(\C)_{sa}^{m+n} \to M_N(\C)$ be $\norm*{\cdot}_2$-Lipschitz (uniformly in $N$) and suppose $f^{(N)} \rightsquigarrow f \in \overline{\TrP}_{m+n}^1$.  Let $g^{(N)}$ be the function given by
\[
g^{(N)}(Y^{(N)}) = E[f^{(N)}(X^{(N)}, Y^{(N)}) | Y^{(N)}],
\]
which is well-defined function $M_N(\C)_{sa}^m$ because $\mu^{(N)}$ has positive density everywhere.  Then $g^{(N)}$ is Lipschitz with
\[
\norm*{g^{(N)}}_{\Lip} \leq (1 + C/c) \norm*{f^{(N)}}_{\Lip}.
\]
Moreover, there exists $g \in \overline{\TrP}_m^1$ such that $g^{(N)} \rightsquigarrow g$ and hence
\[
g(Y) = E_{\mathrm{W}^*(Y)}[f(X,Y)].
\]
\end{theorem}

The gist of the theorem is that the conditional expectation $E[ \cdot | Y^{(N)}]$ behaves in the large $N$ limit like the $\mathrm{W}^*$-algebraic expectation $\mathrm{W}^*(X,Y) \to \mathrm{W}^*(Y)$.   For instance, if $f \in \overline{\TrP}_{m+n}^1$ is globally Lipschitz in $\norm*{\cdot}_2$, then the $\mathrm{W}^*$-algebraic conditional expectation of $f(X,Y)$ can be approximated by the classical conditional expectation $E[f(X^{(N)}, Y^{(N)}) | Y^{(N)}]$.

In fact, we can approximate $E_{\mathrm{W}^*}(Z)$ for every $Z \in L^2(\mathrm{W}^*(X,Y))$ using classical conditional expectations in the same sense.  Indeed, we showed in Proposition \ref{prop:realizationofoperators} that every $Z$ can be expressed as $f(X,Y)$ where $f \in \overline{\TrP}_m^1$ is $\norm*{\cdot}_2$-uniformly continuous, and there exist $\norm*{\cdot}_2$-Lipschitz functions $f_k \in \overline{\TrP}_m^1$ such that $f_k \to f$ with respect to the uniform norm $\norm*{\cdot}_u$.  Let $g_k^{(N)}$ and $g^{(N)}$ be given by
\[
g_k^{(N)}(Y^{(N)}) = E[f_k(X^{(N)}, Y^{(N)}) | Y^{(N)}],
\]
and the analogous relation for $g^{(N)}$ and $f$.  Because conditional expectation is a contraction in $L^\infty(\mu^{(N)})$ (for functions taking values in $M_N(\C)$ with $\norm*{\cdot}_2$), we have
\[
\norm*{g_k^{(N)} - g^{(N)}}_u^{(N)} = \norm*{g_k^{(N)} - g^{(N)}}_{L^\infty(\mu^{(N)})} \leq \norm*{f_k - f}_u
\]
By the theorem, there exists $g_k \in \overline{\TrP}_m^1$ such that $g_k^{(N)} \rightsquigarrow g_k$.  Given that $\norm*{g_k^{(N)} - g^{(N)}}_u^{(N)} \leq \norm*{f_k - f}_u \to 0$, a routine argument (``exchange of limits and uniform limits'') shows that there exists $g \in \overline{\TrP}_m^1$ such that $g^{(N)} \rightsquigarrow g$.  In other words, the conclusion of Theorem \ref{thm:conditionalexpectation} holds also for $f$ and thus $E_{\mathrm{W}^*(Y)}[Z] = E_{\mathrm{W}^*(Y)}[f(X,Y)]$ can be viewed as the large $N$ limit of $E[f(X^{(N)},Y^{(N)}) | Y^{(N)}]$.

\subsection{Strategy} \label{subsec:conditionalexpectationstrategy}

Our proof will follow the same strategy as the special case in \cite[\S 4]{Jekel2018}.  In that paper, we showed that if $V^{(N)}$ and $\mu^{(N)}$ on $M_N(\C)_{sa}^m$ are as in Assumption \ref{ass:convexRMM} and if $u^{(N)}: M_N(\C)_{sa}^m \to \C$ is uniformly Lipschitz and asymptotically approximable by trace polynomials, then $\lim_{N \to \infty} \int u^{(N)}\,d\mu^{(N)}$ exists.

We considered the diffusion semigroup $T_t^{(N)} = T_t^{V^{(N)}}$ that solves the equation
\[
\partial_t (T_t^{(N)} u^{(N)}) = \frac{1}{2N} \Delta (T_t^{(N)} u^{(N)}) - \frac{1}{2} \ip{DV^{(N)}, D(T_t^{(N)} u^{(N)})}_2.
\]
As mentioned in \cite[\S 4]{Jekel2018}, this diffusion semigroup has an equivalent SDE formulation, and is a standard tool in proving the log-Sobolev inequality and concentration estimates (see for instance, \cite{Ledoux1992}, \cite[\S 4.4.2]{AGZ2009}, \cite{DGS2016}).  

Now $\int T_t^{(N)} u^{(N)}\,d\mu^{(N)} = \int u^{(N)}\,d\mu^{(N)}$ and $\norm*{T_t^{(N)} u^{(N)}}_{\Lip} \leq e^{-ct/2} \norm*{u^{(N)}}_{\Lip}$.  As $t \to \infty$, the function $T_t^{(N)} u^{(N)}$ converges to the constant function $\int u^{(N)}\,d\mu^{(N)}$ at a rate independent of $N$.  On the other hand, we showed in \cite[Lemma 4.10]{Jekel2018} that if $\{u^{(N)}\}_{N \in \N}$ and $\{DV^{(N)}\}_{N \in \N}$ are asymptotically approximable by trace polynomials, then so is $\{T_t^{(N)} u^{(N)}\}_{N \in \N}$.  Hence, we concluded that the sequence of constant functions $\{ \int u^{(N)}\,d\mu^{(N)}\}$ is asymptotically approximable by trace polynomials, which means that the limit as $N \to \infty$ exists.

Now we apply the same method in the conditional setting to prove Theorem \ref{thm:conditionalexpectation}.  Let $V^{(N)}(x,y)$ be a function satisfying Assumption \ref{ass:convexRMM}.  If we fix $y$, then $V^{(N)}(\cdot,y)$ is uniformly convex and semi-concave function of $x$, so it defines a log-concave probability measure on $M_N(\C)_{sa}^m$.  This produces a well-behaved conditional distribution of $X^{(N)}$ given $Y^{(N)}$, where $(X^{(N)},Y^{(N)}) \sim \mu^{(N)}$.   Explicitly, for $f \in L^1(\mu^{(N)}, M_N(\C))$, we have
\[
E[f(X^{(N)},Y^{(N)}) | Y^{(N)}] = \frac{\int f(x,Y^{(N)}) e^{-N^2 V^{(N)}(x,Y^{(N)})}\,dx}{\int e^{-N^2 V^{(N)}(x,Y^{(N)})}\,dx}.
\]
We will evaluate this conditional expectation as the limit as $t \to \infty$ of $T_t^{(N)} f$, where $T_t^{(N)} = T_t^{V^{(N)}}$ is the semigroup, acting on Lipschitz functions of $(x,y)$, that solves
\[
\partial_t (T_t^{(N)} f) = \frac{1}{2N} \Delta_x (T_t^{(N)} f) - \frac{1}{2} J_x(T_t^{(N)} f)^* D_x V^{(N)},
\]
where $J_x(T_t^{(N)} f)$ denotes the differential (Jacobian) of $T_t^{(N)} f$ as a function $x$ from $M_N(\C)_{sa}^m$ to $M_N(\C)$ and $*$ denotes the adjoint.  In \S \ref{subsec:relativediffusion}, we will analyze how $T_t^{(N)}$ affects the Lipschitz norms with respect to $x$ and $y$ separately and hence show that the conditional expectation is given by a Lipschitz function of $y$.  In \S \ref{subsec:relativediffusion2}, we will show that $T_t^{(N)}$ preserves asymptotic approximability by trace polynomials of $(x,y)$ and conclude our argument.  The new aspect compared to \cite{Jekel2018} is that the functions are matrix-valued and depend on an extra parameter $y$.

\subsection{Conditional Diffusion Semigroup} \label{subsec:relativediffusion}

To simplify notation, let us fix $N$ and fix $V: M_N(\C)_{sa}^m \times M_N(\C)_{sa}^n \to \R$ for the remainder of \S \ref{subsec:relativediffusion}.  We will denote
\[
d\mu(x | y) = \frac{1}{\int e^{-N^2 V(x,y)}\,dx} e^{-N^2 V(x,y)}\,dx,
\]
which is a measure on $M_N(\C)_{sa}^m$ depending on the parameter $y$.  The associated semigroup $T_t$ will be approximated by alternating two other operators $P_t$ and $S_t$ on short time intervals.  Let $P_t$ denote the semigroup of convolution with Gaussian with respect to $x$, that is,
\[
P_t f(x,y) = \int f(x+z,y)\,d\sigma_{t,m}^{(N)}(z).
\]
The semigroup $S_t$ is given by
\[
S_t f(x,y) = f(W_t(x,y),y),
\]
where $W_t: M_N(\C)_{sa}^m \times M_N(\C)_{sa}^n \to M_N(\C)_{sa}^n$ is the solution to the initial value problem
\begin{align*}
W_0(x,y) &= x \\
\partial_t W_t(x,y) &= -\frac{1}{2} D_x V(W_t(x,y),y).
\end{align*}
This solution is defined for all $t \geq 0$ by the Picard-Lindel\"of theorem because $D_x V(x,y)$ is globally Lipschitz in $x$ (compare \S \ref{subsec:vectorfields}).

\begin{proposition} \label{prop:diffusionsemigroup}
There exists a semigroup $T_t$ acting on Lipschitz functions $M_N(\C)_{sa}^m \times M_N(\C)_{sa}^n$ such that the following hold:
\begin{enumerate}
	\item If $t = n / 2^\ell$ is a dyadic rational, let $T_{t,\ell} f = (P_{2^{-\ell}} S_{2^{-\ell}})^n$.  Then $T_{t,\ell} f \to T_t f$ as $\ell \to \infty$ and more precisely
	\[
	\norm*{T_{t,\ell} f(\cdot,y) - T_t f(\cdot,y)}_{L^\infty} \leq \frac{Cm^{1/2}}{c(2 - 2^{1/2})} 2^{-\ell/2} \norm*{f(\cdot,y)}_{\Lip}.
	\]
	\item If $0 \leq s \leq t$, we have
	\[
	\norm*{T_t f(x,y) - T_s f(x,y)}_2 \leq e^{-cs/2} \left( \frac{C}{c} (6 + 5 \sqrt{2})(t - s)^{1/2} + \norm*{D_x V(x,y)}_2 \right) \norm*{f(\cdot,y)}_{\Lip}.
	\]
	\item $\norm*{T_t f(\cdot,y)}_{\Lip} \leq e^{-ct/2} \norm*{f(\cdot,y)}_{\Lip}$.
	\item $\int T_t f(x,y)\,d\mu(x|y) = \int f(x,y)\,d\mu(x|y)$.
	\item We have $T_t f(x,y) \to \int f(x',y)\,d\mu(x'|y)$ as $t \to \infty$ and specifically
	\[
	\norm*{T_t f(x,y) - \int f(x',y)\,d\mu(x'|y)}_2 \leq e^{-ct/2} \left( 4\frac{C}{c^2} (6 + 5 \sqrt{2}) t^{-1/2} + \frac{2}{c} \norm*{D_x V(x,y)}_2 \right) \norm*{f(\cdot,y)}_{\Lip}.
	\]
\end{enumerate}
\end{proposition}

\begin{proof}
These results follow by freezing the variable $y$ and applying the results from our previous paper, specifically,
\begin{enumerate}
	\item see \cite[Lemma 4.5]{Jekel2018},
	\item see \cite[Lemma 4.6]{Jekel2018},
	\item see \cite[Lemma 4.6]{Jekel2018},
	\item see \cite[Lemma 4.8]{Jekel2018},
	\item see \cite[Lemma 4.9]{Jekel2018}.
\end{enumerate}
The results of \cite[\S 4]{Jekel2018} were stated only for scalar-valued functions.  However, the arguments hold for functions from $M_N(\C)_{sa}^m$ to any finite-dimensional normed vector space.  The result (4) that $T_t$ is expectation-preserving follows immediately by applying the scalar-valued result to each coordinate of the vector-valued function in some basis.  To verify the estimates, one simply replaces the ``$|\cdot|$'' in the arguments by the appropriate norm, which in our case would be $\norm*{\cdot}_2$ on $M_N(\C)$.
\end{proof}

We will next show that $W_t(x,y)$ and $T_t f(x,y)$ depend in a Lipschitz manner upon $y$.  Let us denote
\begin{align*}
\norm*{f}_{\Lip(dx)} &= \sup_y \norm*{f(\cdot,y)}_{\Lip} \\
\norm*{f}_{\Lip(dy)} &= \sup_x \norm*{f(x,\cdot)}_{\Lip}.
\end{align*}

\begin{lemma}
With the setup above, we have for Lipschitz $f: M_N(\C)_{sa}^m \times M_N(\C)_{sa}^n \to M_N(\C)$
\begin{enumerate}
	\item $\norm*{W_t}_{\Lip,dx} \leq e^{-ct/2}$ and $\norm*{W_t}_{\Lip,dy} \leq (C/c)(1 - e^{-ct/2})$.
	\item $\norm*{S_t f}_{\Lip,dx} \leq e^{-ct/2} \norm*{f}_{\Lip,dx}$.
	\item $\norm*{S_t f}_{\Lip,dy} \leq \norm*{f}_{\Lip,dy} + (C/c)(1 - e^{-ct/2}) \norm*{f}_{\Lip,dx}$.
	\item $\norm*{P_t f}_{\Lip,dy} \leq \norm*{f}_{\Lip,dy}$ and $\norm*{P_t f}_{\Lip,dx} \leq \norm*{f}_{\Lip,dx}$.
	\item $\norm*{T_t f}_{\Lip,dx} \leq e^{-ct/2} \norm*{f}_{\Lip,dx}$.
	\item $\norm*{T_t f}_{\Lip,dy} \leq \norm*{f}_{\Lip,dy} + (C/c)(1 - e^{-ct/2}) \norm*{f}_{\Lip,dx}$.
\end{enumerate}
\end{lemma}

\begin{proof}
(1) Fix $x, x' \in M_N(\C)_{sa}^m$ and $y, y' \in M_N(\C)_{sa}^n$.  Define
\[
\phi(t) = \norm*{W_t(x,y) - W_t(x',y')}_2.
\]
Note that $\phi$ is locally Lipschitz in $t$ and hence absolutely continuous.  Moreover, $\phi(t)^2$ is $C^1$ with
\begin{align*}
\partial_t [\phi(t)^2] &= 2 \ip{\partial_t W_t(x,y) - \partial_t W_t(x',y'), W_t(x,y) - W_t(x',y')}_2 \\
&= -\ip{D_x V(W_t(x,y),y) - D_x V(W_t(x',y'),y'), W_t(x,y) - W_t(x',y')}_2 \\
&= -\ip{D_x V(W_t(x,y),y) - D_x V(W_t(x',y'),y), W_t(x,y) - W_t(x',y')}_2 \\
& \quad - \ip{D_x V(W_t(x',y'),y) - D_x V(W_t(x',y'),y'), W_t(x,y) - W_t(x',y')}_2 \\
&\leq -c \norm*{W_t(x,y) - W_t(x',y')}_2^2 \\
& \quad + \norm*{D_x V(W_t(x',y'),y) - D_x V(W_t(x',y'),y)}_2 \norm*{W_t(x,y) - W_t(x',y')}_2 \\
&\leq -c \norm*{W_t(x,y) - W_t(x',y')}_2^2 + C \norm*{y - y'}_2 \norm*{W_t(x,y) - W_t(x',y')}_2.
\end{align*}
Here we have employed the inequality $\ip{D_xV(z,w) - D_x V(z',w), z - z'}_2 \geq c \norm{z - z'}_2^2$ coming from the uniform convexity of $V$ as well as the Cauchy-Schwarz inequality. This implies that
\[
2 \phi'(t) \phi(t) = \partial_t [\phi(t)^2] \leq -c \phi(t)^2 + C \norm*{y - y'} \phi(t).
\]
Thus, $\phi'(t) \leq -(c/2) \phi(t) + (C/2) \norm*{y - y'}$, so that $\partial_t[e^{ct/2} \phi(t)] \leq (C/2) e^{ct/2} \norm*{y - y'}_2$.  This implies that
\[
e^{ct/2} \phi(t) - \phi(0) \leq \frac{C}{c} (e^{ct/2} - 1) \norm*{y - y'}_2.
\]
But $\phi(t) = \norm*{W_t(x,y) - W_t(x',y')}_2$ and $\phi(0) = \norm*{x - x'}_2$.  Hence,
\[
\norm*{W_t(x,y) - W_t(x',y')}_2 \leq e^{-ct/2} \norm*{x - x'}_2 + \frac{C}{c}(1 - e^{-ct/2}) \norm*{y - y'}_2.
\]
This proves both estimates of (1).

(2) This is immediate since $S_t f(x,y) = f(W_t(x,y),y)$, as in \cite[Lemma 4.4 (5)]{Jekel2018}.

(3) Note that
\begin{align*}
\norm*{S_t f(x,y) - S_t f(x,y')}_2 &= \norm*{f(W_t(x,y),y) - f(W_t(x,y'),y')}_2 \\
&\leq \norm*{f(W_t(x,y),y) - f(W_t(x,y),y')}_2 + \norm*{f(W_t(x,y),y') - f(W_t(x,y'),y')}_2 \\
&\leq \norm*{f}_{\Lip,dy} \norm*{y - y'}_2 + \norm*{f}_{\Lip,dx} \norm*{W_t(x,y) - W_t(x,y')}_2 \\
&\leq \norm*{f}_{\Lip,dy} \norm*{y - y'}_2 + \frac{C}{c}(1 - e^{-ct/2}) \norm*{f}_{\Lip,dx} \norm*{y - y'}_2.
\end{align*}

(4) This follows from basic properties of convolution of a function with a probability measure.

(5) By iterating the estimates (2) and (4), we obtain $\norm*{T_{t,\ell} f}_{\Lip,dx} \leq e^{-ct/2} \norm*{f}_{\Lip,dx}$.  Then by Proposition \ref{prop:diffusionsemigroup} (2) and (3) we may take $\ell \to \infty$ and then extend to all real values of $t \geq 0$.

(6) First, consider $T_{t,\ell}$ for a dyadic rational $t = n / 2^\ell$.  Denote $\delta = 2^{-\ell}$.  For $j = 0$, \dots, $n-1$, we have
\begin{align*}
\norm*{T_{(j+1)\delta,\ell}f}_{\Lip,dy} &= \norm*{P_\delta S_\delta T_{j \delta,\ell} f}_{\Lip,dy} \\
&\leq \norm*{S_\delta T_{j \delta,\ell} f}_{\Lip,dy} \\
&\leq \norm*{T_{j \delta,\ell} f}_{\Lip,dy} + \frac{C}{c}(1 - e^{-c\delta/2}) \norm*{T_{j \delta, \ell} f}_{\Lip,dx},
\end{align*}
where the last inequality follows from (3).  Therefore, by induction
\begin{align*}
\norm*{T_{t,\ell} f}_{\Lip,dy} &= \norm*{T_{n\delta,\ell} f}_{\Lip,dy} \\
&\leq \norm*{f}_{\Lip,dy} + \frac{C}{c}(1 - e^{-c\delta/2}) \sum_{j=0}^{n-1} \norm*{T_{j\delta,\ell} f}_{\Lip,dx} \\
&\leq \norm*{f}_{\Lip,dy} + \frac{C}{c}(1 - e^{-c \delta/2}) \sum_{j=0}^{n-1} e^{-c\delta j / 2} \norm*{f}_{\Lip,dx} \\
&= \norm*{f}_{\Lip,dy} + \frac{C}{c}(1 - e^{-ct/2}) \norm*{f}_{\Lip,dx}.
\end{align*}
In light of Proposition \ref{prop:diffusionsemigroup} (1), we can take $\ell \to \infty$ and conclude that $\norm*{T_t f}_{\Lip,dy} \leq \norm*{f}_{\Lip,dy} + (C/c)(1 - e^{-ct/2}) \norm*{f}_{\Lip,dx}$ for dyadic rational $t$.  This inequality can then be extended to all real $t \geq 0$ by Proposition \ref{prop:diffusionsemigroup} (2).
\end{proof}

\begin{corollary} \label{cor:conditionalexpectationLipschitz}
Let $f: M_N(\C)_{sa}^m \times M_N(\C)_{sa}^n \to M_N(\C)$ be Lipschitz with respect to $\norm*{\cdot}_2$.  Let $g(y) = \int f(x,y)\,d\mu(x|y)$.  Then $g$ is Lipschitz with
\[
\norm*{g}_{\Lip} \leq (1 + C/c) \norm*{f}_{\Lip}.
\]
\end{corollary}

\begin{proof}
By the previous lemma,
\begin{align*}
\norm*{T_t f}_{\Lip,dy} &\leq \norm*{f}_{\Lip,dy} + \frac{C}{c}(1 - e^{-ct/2}) \norm*{f}_{\Lip,dx} \\
&\leq \norm*{f}_{\Lip} + \frac{C}{c} \norm*{f}_{\Lip}.
\end{align*}
As $t \to \infty$, we have $T_t f(x,y) \to g(y)$ by Proposition \ref{prop:diffusionsemigroup} (5).  Hence, $\norm*{g}_{\Lip} \leq (1 + C/c) \norm*{f}_{\Lip}$.
\end{proof}

\subsection{Asymptotic Approximation and Convergence} \label{subsec:relativediffusion2}

Let $V^{(N)}$ and $\mu^{(N)}$ be as in Theorem \ref{thm:conditionalexpectation}, let $(X^{(N)},Y^{(N)})$ be a random variable with distribution $\mu^{(N)}$.  Let $\mu^{(N)}(x|y)$ denote the conditional distribution of $X^{(N)}$ given $Y^{(N)}$.

Let $P_t^{(N)}$, $S_t^{(N)}$, and $T_t^{(N)}$ be the semigroups acting on Lipschitz functions defined as in \S \ref{subsec:relativediffusion} with respect to the potential $V^{(N)}$.

\begin{lemma} \label{lem:diffusionAATP}
With the notation above, suppose that $f^{(N)}: M_N(\C)_{sa}^m \to M_N(\C)$, that $f^{(N)}$ is $K$-Lipschitz for every $N$, and that $f^{(N)}$ is asymptotically approximable by trace polynomials.  Then
\begin{enumerate}
	\item $\{P_t^{(N)} f^{(N)}\}$ is asymptotically approximable by trace polynomials.
	\item $\{S_t^{(N)} f^{(N)}\}$ is asymptotically approximable by trace polynomials.
	\item $\{T_t^{(N)} f^{(N)}\}$ is asymptotically approximable by trace polynomials.
\end{enumerate}
\end{lemma}

\begin{proof}
(1) We proved in Lemma \ref{lem:Gaussianconvolution} that $P_t^{(N)}$ preserves asymptotic approximability by trace polynomials.

(2) Recall that $S_t^{(N)} f^{(N)}(x,y) = f^{(N)}(W_t^{(N)}(x,y),y)$, where
\begin{align*}
W_0^{(N)}(x,y) &= x \\
\partial_t W_t^{(N)}(x,y) &= -\frac{1}{2} D_x V^{(N)}(W_t(x,y),y).
\end{align*}
Now $D_xV^{(N)}(x,y)$ is $C$-Lipschitz in $(x,y)$, asymptotically approximable by trace polynomials, and independent of $t$, and thus it satisfies Assumption \ref{ass:vectorfield2}, so by Proposition \ref{prop:ODE2}, $W_t^{(N)}(x,y)$ is asymptotically approximable by trace polynomials (here we rely on Lemma \ref{lem:AATP} that asymptotic approximability is equivalent to being asymptotic to some element of $\overline{\TrP}_{m+n}^1$).  Then because $f^{(N)}$ is $K$-Lipschitz in $(x,y)$, Lemma \ref{lem:composition} implies asymptotic approximability of $f^{(N)}(W_t^{(N)}(x,y),y)$.

(3) Let $T_{t,\ell}^{(N)} = (P_{2^{-\ell}}^{(N)} S_{2^{-\ell}}^{(N)})^n$ whenever $t = n 2^{-\ell}$.  From (1) and (2), it follows that $T_{t,\ell}^{(N)} f^{(N)}$ is asymptotically approximable by trace polynomials.  Now for each dyadic $t$, Proposition \ref{prop:diffusionsemigroup} (1) shows that $T_{t,\ell}^{(N)} f^{(N)} \to T_t^{(N)} f^{(N)}$ uniformly on $\norm*{\cdot}_2$-balls (and hence on $\norm*{\cdot}_\infty$).  Therefore, by Lemma \ref{lem:limits}, $T_t^{(N)} f^{(N)}$ is asymptotically approximable by trace polynomials.  Then we extend this property from dyadic $t$ to all real $t$ using Proposition \ref{prop:diffusionsemigroup} (2) and Lemma \ref{lem:limits}.
\end{proof}

\begin{proof}[Proof of Theorem \ref{thm:conditionalexpectation}]
Let $f^{(N)}: M_N(\C)_{sa}^{m+n} \to M_N(\C)$ be $K$-Lipschitz and asymptotically approximable by trace polynomials.  Let
\[
g^{(N)}(y) = \int f^{(N)}(x,y) \,d\mu^{(N)}(x|y)
\]
We showed in Corollary \ref{cor:conditionalexpectationLipschitz} that $g^{(N)}$ is Lipschitz with $\norm*{g^{(N)}}_{\Lip} \leq (1 + C/c) \norm*{f^{(N)}}_{\Lip}$.  We know that $T_t^{(N)} f^{(N)}$ is asymptotically approximable by trace polynomials in $(x,y)$.  By Proposition \ref{prop:diffusionsemigroup} (5), we have $T_t^{(N)} f^{(N)}(x,y) \to g^{(N)}(x,y)$ as $t \to \infty$, with the error bounded by
\[
e^{-ct/2} \left( 4\frac{C}{c^2} (6 + 5 \sqrt{2}) t^{-1/2} + \frac{2}{c} \norm*{D_x V(x,y)}_2 \right) \norm{f^{(N)}}_{\Lip}.
\]
Given that $\{DV^{(N)}\}$ is asymptotically approximable by trace polynomials, $\norm*{D_x V}_{u,R}^{(N)}$ is bounded as $N \to \infty$.  This implies that the rate of convergence of $T_t^{(N)} f^{(N)}(x,y) \to g^{(N)}(x,y)$ as $t \to \infty$ is uniform on $\norm*{(x,y)}_\infty \leq R$ and independent of $N$.  So by Lemma \ref{lem:limits}, $g^{(N)}$ is asymptotically approximable by trace polynomials of $(x,y)$.  Yet $g^{(N)}$ is independent of $x$, and so we may approximate $g^{(N)}(y)$ by evaluating these trace polynomials at $(0,y)$, which reduces them to trace polynomials of $y$.

Since $g^{(N)}$ is asymptotically approximable by trace polynomials, let $g \in \overline{\TrP}_m^1$ such that $g^{(N)} \rightsquigarrow g$.  Then it remains to show that $g(Y) = E_{\mathrm{W}^*(Y)}[f(X,Y)]$, where $(X,Y)$ are non-commutative random variables for the free Gibbs law $\lambda$ as in the theorem statement.  It suffices to check that
\[
\tau(\phi(Y) g(Y)) = \tau(\phi(Y) f(X,Y))
\]
whenever $\phi$ is a non-commutative polynomial.  But using Corollary \ref{cor:convergenceofexpectation},
\begin{align*}
\tau(\phi(Y) g(Y)) &= \lim_{N \to \infty} E[\tau_N(\phi(Y^{(N)}) g^{(N)}(Y))] \\
&= \lim_{N \to \infty} E[\tau_N(\phi(Y^{(N)}) f^{(N)}(X^{(N)}Y^{(N)}))] \\
&= \tau(\phi(Y) f(X,Y)).  \qedhere
\end{align*}
\end{proof}

\begin{remark}
We showed in \S \ref{subsec:heatsemigroup} that $P_t^{(N)}$ has a large $N$ limit $P_t^{\TrP}$ acting on $\overline{\TrP}_{m+n}^1$.  Similarly, the results of \S \ref{subsec:vectorfields} imply that $S_t^{(N)}$ has a large $N$ limit $S_t^{\TrP}$ acting on $\overline{\TrP}_{m+n}^1$.  This implies that the semigroup $T_t^{(N)}$ also has a large $N$ limit $T_t^{\TrP}$ in light of Proposition \ref{prop:diffusionsemigroup} (1) and (2) and Lemma \ref{lem:limits}.  Future research should investigate in what sense $F(x,t) = T_t^{\TrP}f(x)$ would solve the differential equation
\[
\partial_t F = \frac{1}{2} L_x F - \frac{1}{2} (J_xF)^* (D_x V),
\]
where $V$ is the large $N$ limit of $\{V^{(N)}\}$ and $J_xF$ is the Jacobian matrix of $F$ with respect to the variable $x$.
\end{remark}

\section{Conditional Entropy and Fisher's Information} \label{sec:entropy}

In this section, we show that for random matrix models satisfying Assumption \ref{ass:convexRMM}, the conditional (classical) entropy $h(X^{(N)} | Y^{(N)})$ converges to the conditional non-microstates free entropy $\chi^*(X | Y)$ (also known as $\chi^*(X: \mathrm{W}^*(Y))$).

\subsection{Conditional Entropy and Fisher's Information in the Classical Setting}

We refer to \cite[\S 3]{Voiculescu2002} and \cite[\S 5]{Jekel2018} for background on classical entropy and Fisher's information and motivation for the free case.  The conditional setting is more technical, and we will state several standard results without proof, since the proofs in the non-conditional case were repeated in some detail in \cite{Jekel2018}.

Recall that the classical entropy of a random variable $X$ in $\R^m$ with probability density $\rho$ is $h(X) = -\int \rho \log \rho$.  Similarly, if $(X,Y)$ is a random variable in $\R^m \times \R^n$ with density $\rho_{X,Y}(x,y)$, then the \emph{conditional entropy} $h(X|Y)$ is defined by
\begin{equation} \label{eq:classicalentropy}
h(X | Y) = \int_{\R^n} \int_{\R^m} \rho_{X|Y}(x|y) \log \rho_{X|Y}(x|y)\,dx \rho_Y(y)\,dy = \int_{\R^n \times \R^m} (\log \rho_{X|Y}(x|y)) \rho(x,y)\,dx\,dy,
\end{equation}
where $\rho_Y$ is the marginal density
\[
\rho_Y(y) = \int_{\R^m} \rho_{X,Y}(x,y)\,dx
\]
and $\rho_{X|Y}$ is the conditional density
\[
\rho_{X|Y}(x|y) = \frac{\rho_{X,Y}(x,y)}{\rho_Y(y)} \text{ defined when } \rho_Y(y) > 0.
\]
It is a standard fact that if $X$ has finite variance, then $h(X|Y)$ is well-defined.  The proof for the non-conditional entropy $h(X)$ was reviewed in \cite[Lemma 5.1]{Jekel2018}, and the conditional case can be handled similarly.

The \emph{conditional Fisher information} given by
\begin{equation} \label{eq:classicalFisherinfo}
\mathcal{I}(X|Y) = \int_{\R^m \times \R^n} \left| \frac{\nabla_x \rho_{X|Y}(x|y)}{\rho_{X|Y}(x|y)} \right|^2 \rho_{X,Y}(x,y)\,dx\,dy,
\end{equation}
whenever the right hand side makes sense and $\infty$ otherwise.  It describes the rate of change of $h(X+t^{1/2}S | Y)$, where $S$ is a Gaussian random variable in $\R^m$ with covariance matrix $I$ independent of $(X,Y)$.  Knowing that the density $\rho_{X+t^{1/2}S,Y}$ satisfies the heat equation
\[
\partial_t \rho_{X+t^{1/2}S,Y} = \frac{1}{2} \Delta_x \rho_{X+t^{1/2}S,Y},
\]
one can show that $\mathcal{I}(X + t^{1/2}S | Y)$ is well-defined and finite for $t > 0$ and that
\begin{equation} \label{eq:classicalentropyrateofchange}
\frac{d}{dt} h(X + t^{1/2} S | Y) = \frac{1}{2} \mathcal{I}(X + t^{1/2} S | Y).
\end{equation}

The Fisher information is the $L^2$ norm of the ($\R^m$-valued) random variable $\Xi$ given by evaluating the \emph{score function} $-\nabla_x \rho_{X|Y} / \rho_{X|Y}$ on the random variable $(X,Y)$, provided that this random variable is in $L^2$.  In this case, the random variable $\Xi$ is known as the \emph{score function} for $X$ given $Y$, and it is the unique element of $L^2$ satisfying the integration-by-parts relation
\begin{equation} \label{eq:integrationbyparts}
E[\Xi f(X,Y)] = E[ \nabla_x f(X,Y) ] \text{ for all } f \in C_c^\infty(\R^m \times \R^n).
\end{equation}
More generally, if there exists a random variable $\Xi$ in $L^2$ satisfying this integration-by-parts formula, then we define the conditional Fisher information to be $\mathcal{I}(X|Y) = E |\Xi|^2$ (and this extends our previous definition of $\mathcal{I}(X|Y)$).  Otherwise, $\mathcal{I}(X|Y)$ is defined to be $\infty$.

In light of the integration-by-parts characterization, score functions behave well under conditionally independent sums.   The following lemma is proved in the same way as the non-conditional case (see \cite[Lemma 5.6]{Jekel2018}) and the free case (see \cite[Proposition 3.7]{VoiculescuFE5}).

\begin{lemma} \label{lem:xiconditionalexpectation}
Let $Y$ be a random variable in $\R^n$ and let $X_1$ and $X_2$ be random variables in $\R^m$ that are conditionally independent given $Y$.  Suppose that $\Xi$ is a score function for $X_1$ given $Y$.  Then $E[\Xi| X_1 + X_2, Y]$ is a score function for $X_1 + X_2$ given $Y$.  Hence,
\[
\mathcal{I}(X_1+X_2|Y) \leq \mathcal{I}(X_1 | Y).
\]
In particular, this holds if $X_2$ is independent from $(X_1,Y)$ or $X_1$ is independent of $(X_2,Y)$.  
\end{lemma}

Score functions also scale in the following way.  The proof is straightforward from the integration-by-parts relation.

\begin{lemma} \label{lem:xiscaling}
If $\Xi$ is a score function for $X$ given $Y$ and $t > 0$, then $(1/t) \Xi$ is a score function for $tX$ given $Y$, and hence $\mathcal{I}(tX | Y) = t^{-2} \mathcal{I}(X|Y)$.
\end{lemma}

\subsection{Random Matrix Renormalization} \label{subsec:matrixentropy}

Suppose that $(X,Y)$ is a random variable in $M_N(\C)_{sa}^m \times M_N(\C)_{sa}^n$ with density $\rho_{X,Y}$.  The trace on $M_N(\C)_{sa}$ produces a real inner product.  But to study the large $N$ limit, we use the normalized trace $\tau_N = (1/N) \Tr$.  The corresponding normalized Gaussian is the GUE ensemble $S = (S_1,\dots,S_m)$ where $S_j$ has variance $1$ with respect to $\tau_N$.  We use the following renormalized entropy, which is motivated by computation of the Gaussian case and by \eqref{eq:normalizedentropyrateofchange} below,
\[
h^{(N)}(X|Y) = \frac{1}{N^2} h(X|Y) + \frac{m}{2} \log N.
\]
Due to the normalization of Gaussian, the evolution of the density for $(X + t^{1/2} S, Y)$ is given by the renormalized heat equation
\[
\partial_t \rho_{X+t^{1/2}S, Y} = \frac{1}{2N} \Delta \rho_{X+t^{1/2}S,Y}.
\]
This results in
\begin{equation} \label{eq:normalizedentropyrateofchange}
\partial_t h^{(N)}(X + t^{1/2} S | Y) = \frac{1}{2N^3} \mathcal{I}(X + t^{1/2}S | Y) =: \frac{1}{2} \mathcal{I}^{(N)}(X+t^{1/2}S | Y),
\end{equation}
where $\mathcal{I}^{(N)}(X|Y) := N^{-3} \mathcal{I}(X|Y)$, assuming that $X$ has finite variance and $t > 0$.

Another heuristic for the normalization $\mathcal{I}^{(N)} = N^{-3} \mathcal{I}$ comes from analyzing the case where $(X,Y)$ have density $(1/Z) e^{-N^2 V(x,y)}\,dx\,dy$ where $V$ is uniformly convex and semi-concave.  Indeed, in this case, the classical score function for $X$ given $Y$ is $-N^2 \nabla_x V(X,Y)$.  Recall that $D_x V = N \nabla_x V$ is the gradient of $V$ with respect to the normalized inner product $\ip{\cdot,\cdot}_2$.  Thus,
\[
\frac{1}{N^3} \mathcal{I}(X|Y) = \frac{1}{N} E \norm{N \nabla_x V(X,Y)}_{\Tr}^2 = E \norm{D_x V(X,Y)}_2^2
\]
is a dimension-independent normalization.  Furthermore, the normalized score function $\xi = (1/N) \Xi$ (which would be $D_x V(X,Y)$ in the case where the law is given by a potential $V$) satisfies the integration-by-parts relation
\[
E \ip{\xi_j, f(X,Y)}_2 = E \frac{1}{N^2} \Div_{x_j} f(X,Y),
\]
where $\xi = (\xi_1,\dots,\xi_m)$ and where $\Div$ is the divergence with respect to the classical coordinates (not normalized).  But if $f$ is a non-commutative polynomial, then
\[
\frac{1}{N^2} \Div_{x_j} f(x,y) = \frac{1}{N^2} \Tr \otimes \Tr (\partial_{x_j} f(x,y)) = \tau_N \otimes \tau_N(\partial_{x_j} f(x,y)),
\]
where $\partial_{x_j}$ denotes the non-commutative derivative or free difference quotient with respect to $x_j$.  Thus, applying the integration-by-parts relation to non-commutative polynomials results in the dimension-independent relation
\[
E \ip{\xi_j, f(X,Y)}_2 = \sum_{j=1}^m E[\tau_N \otimes \tau_N(\partial_{x_j} f(x,y))]
\]
that characterizes the normalized score function.

As a consequence of \eqref{eq:normalizedentropyrateofchange}, $h^{(N)}(X|Y)$ can be recovered by integrating $\mathcal{I}^{(N)}(X+t^{1/2}S | Y)$ and modifying the integral to converge at $\infty$.  This results in
\begin{equation} \label{eq:normalizedentropyformula}
h^{(N)}(X|Y) = \frac{1}{2} \int_0^\infty \left( \frac{m}{1 + t} - \mathcal{I}^{(N)}(X + t^{1/2}S | Y) \right)\,dt + \frac{m}{2} \log 2\pi e
\end{equation}
provided that $(X,Y)$ has a density $\rho_{X,Y}$ and that $X$ has finite variance.  The proof is similar to \cite[Lemma 5.7]{Jekel2018}.  Convergence of the integral at $\infty$ can be deduced from the following estimate, and it also shows convergence of the integral at $0$ if $\mathcal{I}(X|Y)$ is finite.  Compare \cite[Corollary 6.14 and Remark 6.15]{VoiculescuFE5} and \cite[Lemma 5.7]{Jekel2018}.

\begin{lemma} \label{lem:Fisherestimates}
Let $(X,Y)$ be a random variable in $M_N(\C)_{sa}^m \times M_N(\C)_{sa}^n$ such that $a = (1/m) \sum_{j=1}^m E[\tau_N(X_j^2)]  < \infty$, and let $S$ be an independent GUE $m$-tuple.  Then
\[
\frac{m}{a + t} \leq \mathcal{I}^{(N)}(X + t^{1/2} S | Y) \leq \min \left( \frac{m}{t}, \mathcal{I}^{(N)}(X|Y) \right).
\]
\end{lemma}

\begin{proof}
We observe that $\xi_t = E[t^{-1/2} S | X + t^{1/2} S, Y]$ is a normalized score function for $X + t^{1/2} S$ given $Y$ by Lemma \ref{lem:xiconditionalexpectation}.  This yields $\mathcal{I}^{(N)}(X + t^{1/2} S | Y) \leq m / t$.  On the other hand, if $\xi$ is a normalized score function for $X$ given $Y$, we also have $\xi_t = E[\xi | X + t^{1/2} S, Y]$, which yields the upper bound $\mathcal{I}^{(N)}(X|Y)$.  The lower bound follows from observing $(E \norm{\xi_t}_2^2)^{1/2} (E \norm{X + t^{1/2}S}_2^2)^{1/2} \geq E \ip{\xi_t, X + t^{1/2} S}_2$ and evaluating the right hand side using integration by parts.
\end{proof}

\subsection{Convergence to Conditional Free Entropy}

Motivated by the normalized entropy and Fisher's information in the previous section, Voiculescu defined the free versions as follows.  Let $(X,Y)$ be an $(m+n)$-tuple of self-adjoint non-commutative random variables in a tracial $\mathrm{W}^*$-algebra $(\mathcal{M},\tau)$.  We say that $\xi = (\xi_1,\dots,\xi_m) \in L^2(\mathcal{M},\tau)_{sa}^m$ is a \emph{free score function for $X$ given $Y$} (also known as a \emph{conjugate variable}) if for every non-commutative polynomial $f(X,Y)$, we have
\[
\tau(\xi_j f(X,Y)) = \tau \otimes \tau(\partial_{x_j} f(X,Y)).
\]
The \emph{free Fisher information} $\Phi^*(X|Y)$ is defined to be $\norm{\xi}_2^2$ if such a $\xi$ exists, and $\infty$ otherwise.  The \emph{non-microstates free entropy} $\chi^*(X|Y)$ is defined to be
\begin{equation} \label{eq:definechi*}
\chi^*(X|Y) = \frac{1}{2} \int_0^\infty \left( \frac{m}{1 + t} - \Phi^*(X + t^{1/2}S | Y) \right)\,dt + \frac{m}{2} \log 2\pi e.
\end{equation}
Convergence of the integral at $\infty$ follows from the free analogue of Lemma \ref{lem:Fisherestimates}, so that $\chi^*(X | Y)$ is well-defined in $[-\infty,\infty)$ whenever $X$ has finite variance.

\begin{remark}
Voiculescu's original notation in \cite[\S 7]{VoiculescuFE5} was $\chi^*(X: \mathrm{W}^*(Y))$ rather than $\chi^*(X | Y)$, since the definition of the free score function can be rephrased so as to depend only on $\mathrm{W}^*(Y)$ rather than $Y$.  
However, we prefer to write $\chi^*(X | Y)$ instead by analogy with the classical case, using the vertical bar to denote ``conditioning.''  This avoids potential confusion with the notation $\chi(X:Y)$ for microstates entropy of $X$ in the presence of $Y$ used in \cite[\S 1]{VoiculescuFE3}.
\end{remark}

The following lemma gives sufficient conditions for classical Fisher information for random matrix models to converge to free Fisher information.  The main hypotheses are that the non-commutative laws converge, the score functions $D_x V^{(N)}$ for the $N \times N$ matrix models are asymptotically approximable by trace polynomials, and some mild growth conditions on score functions and probability measures as $\norm{(x,y)}_\infty \to \infty$.  We omit the proof since it is a direct adaptation of the proof of \cite[Proposition 5.10]{Jekel2018}.

\begin{lemma} \label{lem:convergenceofFisherinfo}
Let $V^{(N)}: M_N(\C)_{sa}^{m+n} \to \R$ be a potential with $\int_{M_N(\C)^{m+n}} e^{-N^2 V^{(N)}(x,y)}\,dx\,dy < +\infty$, let $\mu^{(N)}$ be the associated probability density, and let $(X^{(N)}, Y^{(N)})$ be a random variable distributed according to $\mu^{(N)}$.  Let $(X,Y)$ be an $(m+n)$-tuple of self-adjoint non-commutative random variables in the tracial $\mathrm{W}^*$-algebra $(\mathcal{M},\tau)$.  Assume that:
\begin{enumerate}[(A)]
	\item The non-commutative law of $(X^{(N)}, Y^{(N)})$ with respect to $\tau_N$ converges in probability to the non-commutative law of $(X,Y)$.
	\item $D_x V^{(N)}$ is defined and continuous, and the sequence $\{D_x V^{(N)}\}$ is asymptotically approximable by trace polynomials, and hence $D_x V^{(N)} \rightsquigarrow g \in (\overline{\TrP}_{m+n}^1)_{sa}^m$.
	\item For some $k \geq 0$ and $a, b > 0$, we have
	\[
	\norm{D_x V^{(N)}(x,y)}_2^2 \leq a + b \norm{(x,y)}_\infty^k
	\]
	\item There exists $R_0 > 0$ such that
	\[
	\lim_{N \to \infty} E \left[ \mathbf{1}_{\norm{(X^{(N)},Y^{(N)})}_\infty \geq R_0} \left(1 + \norm{(X^{(N)},Y^{(N)})}_\infty^k \right) \right] = 0.
	\]
\end{enumerate}
Then $\mathcal{I}^{(N)}(X|Y)$ is finite.  Moreover, $g(X,Y)$ is in $L^2(\mathcal{M},\tau)$ and it is the free score function for $X$ given $Y$, and we have
\[
\mathcal{I}^{(N)}(X|Y) = E[\norm{D_x V^{(N)}(X^{(N)},Y^{(N)})}_2^2] \to \norm{g(X,Y)}_2^2 = \Phi^*(X|Y).
\]
\end{lemma}

\begin{theorem} \label{thm:convergenceofentropy}
Let $V^{(N)}: M_N(\C)_{sa}^m \times M_N(\C)_{sa}^n \to \R$ satisfy Assumption \ref{ass:convexRMM} for some $0 < c \leq C$.  Let $\mu^{(N)}$ be the corresponding measure, let $X^{(N)}, Y^{(N)}$ be random variables chosen according to $\mu^{(N)}$, and let $S^{(N)}$ be an independent $m$-tuple of GUE matrices.

Let $X = (X_1,\dots,X_m)$ and $Y = (Y_1,\dots,Y_n)$ be non-commutative random variables with non-commutative law $\mu = \mu_V$, and let $S$ be a freely independent free semicircular $m$-tuple.  Then for every $t \geq 0$, we have
\begin{equation} \label{eq:convergenceofFisherinfo}
\Phi^*(X + t^{1/2} S | Y) = \lim_{N \to \infty} \mathcal{I}^{(N)}(X^{(N)} + t^{1/2} S^{(N)} | Y^{(N)})
\end{equation}
and
\begin{equation} \label{eq:convergenceofentropy}
\chi^*(X + t^{1/2} S | Y) = \lim_{N \to \infty} h^{(N)}(X^{(N)} + t^{1/2} S^{(N)} | Y^{(N)}).
\end{equation}
\end{theorem}

\begin{proof}
We want to show that the law of $(X^{(N)} + t^{1/2} S^{(N)}, Y^{(N)})$ satisfies the assumptions of Lemma \ref{lem:convergenceofFisherinfo} for each $t \geq 0$.  The joint law of $(X^{(N)},Y^{(N)},S^{(N)})$ is given by the convex potential $U^{(N)}(x,y,s) = V(x,y) + (1/2) \norm{s}_2^2$.  Now $U^{(N)}$ satisfies $\min(c,1) \leq HU^{(N)} \leq \max(C,1)$ and $DU^{(N)}$ is asymptotically approximable by trace polynomials.  Thus, the law of $(X^{(N)},Y^{(N)},S^{(N)})$ has a large $N$ limit given by Theorem \ref{thm:freeGibbslaw}.  In fact, the large $N$ limit must be non-commutative law of $(X,Y,S)$ because of Voiculescu's asymptotic freeness theorem \cite{Voiculescu1998} and because the non-commutative law of $S^{(N)}$ converges to the non-commutative law of $S$.  (Alternatively, this could be proved the same way as \cite[Lemma 7.4]{Jekel2018}.)

Since the non-commutative law of $(X^{(N)},Y^{(N)},S^{(N)})$ converges in probability to that of $(X,Y,S)$, the non-commutative law of $(X^{(N)} + t^{1/2} S^{(N)}, Y^{(N)})$ converges in probability to that of $(X + t^{1/2}S,Y)$, and thus (A) of Lemma \ref{lem:convergenceofFisherinfo} holds.  Moreover, Lemma \ref{lem:epsilonnet} shows that
\[
P(\norm{(X^{(N)} - E(X^{(N)}),Y^{(N)} - E(Y^{(N)}),S^{(N)})}_\infty \geq \min(1,c)^{-1/2} (\Theta + \delta)) \leq e^{-N \delta^2/2}.
\]
From this it is not hard to show that $(X^{(N)} + t^{1/2} S^{(N)}, Y^{(N)})$ satisfies (D).

It remains to check (B) and (C).  The potential for $(X^{(N)} + t^{1/2} S^{(N)}, Y^{(N)})$ is given by
\[
W_t^{(N)}(\tilde{x},y,s) = U^{(N)}(\tilde{x} - t^{1/2} s, y, s),
\]
which follows by applying the change of variables formula for the density.  Here we write $\tilde{x}$ to emphasize that this variable corresponds to $X^{(N)} + t^{1/2} S^{(N)}$ rather than $X^{(N)}$.  Note that $W_t^{(N)}$ is uniformly convex and semi-concave since it is the composition of $U^{(N)}$ with an invertible linear transformation.  Also,
\[
DW_t^{(N)}(\tilde{x},y,s) = (DV^{(N)}(\tilde{x}-t^{1/2}s,y), s - t^{1/2} D_xV^{(N)}(\tilde{x}-t^{1/2}s,y))
\]
is asymptotically approximable by trace polynomials.  The potential corresponding to $(X^{(N)} + t^{1/2} S^{(N)}, Y^{(N)})$ is
\[
V_t^{(N)}(\tilde{x},y) = -\frac{1}{N^2} \log \int e^{-N^2 W_t^{(N)}(\tilde{x},y,s)}\,ds.
\]
Since $W_t^{(N)}$ is uniformly convex, the integrand vanishes rapidly at $\infty$, and thus it is straightforward to differentiate under the integral by dominated convergence, and deduce that $V_t^{(N)}$ is continuously differentiable.  Furthermore,
\[
DV_t^{(N)}(\tilde{x},y) e^{-N^2 V_t^{(N)}(\tilde{x},y)} = \int DW_t^{(N)}(\tilde{x},y,s) e^{-N^2W_t^{(N)}(\tilde{x},y,s)} \,ds,
\]
so that
\[
DV_t^{(N)}(\tilde{x},y) = \frac{\int DW_t^{(N)}(\tilde{x},y,s) e^{-N^2W_t^{(N)}(\tilde{x},y,s)} \,ds}{\int e^{-N^2W_t^{(N)}(\tilde{x},y,s)} \,ds},
\]
or in other words $DV_t^{(N)}$ is given by the conditional expectation
\begin{align} \label{eq:xiconditionalexpectation2}
DV_t^{(N)}(X^{(N)} + t^{1/2} S^{(N)},Y^{(N)}) &= E\left[D_{(x,y)} W_t^{(N)}(X^{(N)} + t^{1/2} S^{(N)},Y^{(N)},S^{(N)}) \bigl| X^{(N)} + t^{1/2} S^{(N)}, Y^{(N)} \right] \\
&=  E\left[DV^{(N)}(X^{(N)}, Y^{(N)}) \bigl| X^{(N)} + t^{1/2} S^{(N)}, Y^{(N)} \right]. \nonumber
\end{align}
Now we apply Theorem \ref{thm:conditionalexpectation} using the potential $W_t^{(N)}$ and conditioning on $(X^{(N)} + t^{1/2} S^{(N)}, Y^{(N)})$ to conclude conclude that $DV_t^{(N)}(\tilde{x},y)$ is asymptotically approximable by trace polynomials, which establishes (B).

Furthermore, Theorem \ref{thm:conditionalexpectation} implies that
\[
\norm{DV_t^{(N)}}_{\Lip} \leq (1 + C/c) \norm{D_{(x,y)}W_t^{(N)}}_{\Lip} \leq (1 + C/c)(1 + t^{1/2}) \norm{DV^{(N)}}_{\Lip} \leq (1 + C/c)(1 + t^{1/2}) C.
\]
This implies that (C) of Lemma \ref{lem:convergenceofFisherinfo} holds with $k = 1$, using Remark \ref{rem:UCgrowthbound}.

Therefore, we may apply Lemma \ref{lem:convergenceofFisherinfo} to $(X^{(N)} + t^{1/2} S^{(N)}, Y^{(N)})$ to obtain that \eqref{eq:convergenceofFisherinfo} holds for every $t \geq 0$, that is,
\[
\lim_{N \to \infty} \mathcal{I}^{(N)}(X^{(N)} + t^{1/2} S^{(N)} | Y^{(N)}) \to \Phi^*(X+t^{1/2}S | Y).
\]
For the second claim \eqref{eq:convergenceofentropy} regarding $h^{(N)}$ and $\chi^*$, it remains to show that
\[
\lim_{N \to \infty} \frac{1}{2} \int_0^\infty \left( \frac{m}{1 + t} - \mathcal{I}^{(N)}(X + t^{1/2}S | Y) \right)\,dt = \frac{1}{2} \int_0^\infty \left( \frac{m}{1 + t} - \Phi^*(X + t^{1/2}S | Y) \right)\,dt.
\]
We just showed the integrand converges pointwise.  But we can take the limit inside the integral by the dominated convergence theorem, because by Lemma \ref{lem:Fisherestimates}, we have
\[
\frac{m}{a + t} \leq \mathcal{I}^{(N)}(X^{(N)} + t^{1/2} S^{(N)} | Y^{(N)}) \leq \min \left( \frac{m}{t}, \mathcal{I}^{(N)}(X^{(N)}|Y^{(N)}) \right),
\]
and we also know that $\mathcal{I}^{(N)}(X^{(N)} | Y^{(N)})$ is bounded as $N \to \infty$ because it converges to $\Phi^*(X|Y)$.
\end{proof}

\begin{remark}
Of course, \eqref{eq:xiconditionalexpectation2} leads to the same conclusion as Lemma \ref{lem:xiconditionalexpectation}.  Indeed, $\xi_t = D_x V_t^{(N)}(X^{(N)} + t^{1/2}S^{(N)},Y^{(N)})$ is the score function for $X^{(N)} + t^{1/2} S^{(N)}$, and Lemma \ref{lem:xiconditionalexpectation} says that $\xi_t$ is the conditional expectation of $\xi_0 = D_x V^{(N)}(X^{(N)}, Y^{(N)})$ given $X^{(N)} + t^{1/2} S^{(N)}$ and $Y^{(N)}$.
\end{remark}

\begin{remark} \label{rem:simplifiedentropyproof}
In \cite[\S 7]{Jekel2018}, we did not use the conditional expectation method to prove $DV_t^{(N)}$ is asymptotically approximable by trace polynomials, but rather we analyzed the evolution of $DV_t^{(N)}$ directly using PDE semigroups.  The proof given here for convergence of entropy is thus considerably shorter.  However, our results on the evolution of $DV_t^{(N)}$ will come in handy for our construction of transport in the next section.
\end{remark}

\section{Conditional Transport to Gaussian} \label{sec:transport}

In this section, we prove our main results about transport (Theorems \ref{thm:transport1} and \ref{thm:transport2}).  Suppose that $V^{(N)}(x,y)$ is a potential as in Assumption \ref{ass:convexRMM}, $\mu^{(N)}$ is the corresponding probability distribution and that $(X^{(N)},Y^{(N)})$ is a random variable with this law.  Let $S^{(N)}$ be an independent $m$-tuple of GUE matrices.  Let $\mu_t^{(N)}$ be the law of $(X^{(N)} + t^{1/2} S^{(N)}, Y^{(N)})$.

The evolution of the potential $V_t^{(N)}$ corresponding to $\mu_t^{(N)}$ was studied in \cite{Jekel2018}, and in particular, we established a dimension-independent way to obtain $DV_t^{(N)}$ from $DV^{(N)}$ using operations that preserve asymptotic approximability by trace polynomials.  By solving an ODE in terms of $DV_t^{(N)}$, we will obtain transport maps $F_{s,t}^{(N)}: M_N(\C)_{sa}^{m+n} \to M_N(\C)_{sa}^m$ such that
\[
(F_{s,t}^{(N)}(X^{(N)} + t^{1/2} S^{(N)}, Y^{(N)}), Y^{(N)}) \sim (X^{(N)} + s^{1/2} S^{(N)}, Y^{(N)}).
\]
Upon renormalizing and taking the limit as $s$ or $t$ goes to $\infty$, we obtain transport to the law of $(S^{(N)}, Y^{(N)})$.

To make each part of the proof more computationally tractable, we proceed in stages.  Up until \S \ref{subsec:largeNtransport}, we fix $N$ (and thus suppress it in the notation).  First, in \S \ref{subsec:basictransport}, we describe the basic construction of transport for functions of $x$ alone (imagining that we have frozen the variable $y$).  In \S \ref{subsec:conditionalHJB}, we describe the properties of $V_t^{(N)}(x,y)$.  Next, \S \ref{subsec:Lipschitztransport} proves Lipschitz estimates for the transport maps $F_{s,t}^{(N)}(x,y)$.

In \S \ref{subsec:largeTtransport}, we introduce renormalized transport maps $\tilde{F}_{s,t}^{(N)}$ that transport $\tilde{\mu}_t$ to $\tilde{\mu}_s$, where $\tilde{\mu}_t$ is the law of $(e^{-t/2} X^{(N)} + e^{-t/2}(e^t - 1)^{1/2} S^{(N)}, Y^{(N)})$.  The renormalized transport map $\tilde{F}_{s,t}$ is the same one used by Otto and Villani in their proof of the Talagrand transportation-entropy inequality \cite[\S 4, proof of Lemma 2]{OV2000}, in the special case where the target measure is Gaussian (and generalized to the conditional setting).  We will explain this inequality further in \S \ref{subsec:entropycost}.

The new element in our paper is the analysis of the large $t$ and large $N$ limits of the transport maps.  In \S \ref{subsec:largeTtransport}, we show that the limit as $s$ or $t$ tends to $\infty$ exists.  Then in \S \ref{subsec:largeNtransport}, we use the machinery of asymptotic approximation by trace polynomials to study the large $N$ limit of $\tilde{F}_{s,t}^{(N)}$.  In order to get dimension-independent estimates for convergence as $s$ or $t$ tends to $\infty$, we conduct a finer analysis of convexity properties of $V_t$ and Lipschitz properties of $\tilde{F}_{s,t}$.  It is convenient to carry out the earlier stages of this analysis in \S \ref{subsec:conditionalHJB} and \S \ref{subsec:Lipschitztransport} for $F_{s,t}$ rather than $\tilde{F}_{s,t}$.

\subsection{Basic Construction of Transport} \label{subsec:basictransport}

In this section, we will fix $N$ and fix a function $V: M_N(\C)_{sa}^m \to \R$ in $\mathcal{E}_{c,C}$ for some $0 < c < C$.  Later, we will allow $V$ to depend on $N$ and to depend on another self-adjoint tuple $y$, but we prefer to simplify notation for the sake of carrying out the basic computation.

Let $\mu$ be the probability measure with density $(1/Z) e^{-N^2 V}$ where $Z = \int_{M_N(\C)_{sa}^m} e^{-N^2 V}$.  We showed in \cite{Jekel2018} that the density of $\mu_t := \mu * \sigma_{t,N}$ is $(1/Z) e^{-N^2 V_t}$, where $V_t$ solves the equation
\begin{equation} \label{eq:simpleHJB}
\partial_t V_t = \frac{1}{2N} \Delta V_t - \norm*{DV_t}_2^2.
\end{equation}
Because $(1/Z) e^{-N^2 V_t}$ solves the heat equation, we know that $V_t$ is a smooth function of $(x,t)$ for $t > 0$ and a continuous function of $(x,t)$ for $t \geq 0$.  Moreover, $V_t \in \mathcal{E}(c(1+ct)^{-1},C(1+Ct)^{-1})$ for each $t$ as proved in Theorem 6.1 (1) of \cite{Jekel2018}.

Now we can describe explicit transport functions $F_{s,t}$ such that $(F_{s,t})_* \mu_s = \mu_t$ for all $s, t \in [0,+\infty)$.

\begin{proposition} \label{prop:basictransport1}
Let $V$, $\mu$, $V_t$, and $\mu_t$ be as above.
\begin{enumerate}
	\item There exists a unique family of functions $F_{s,t}: M_N(\C)_{sa}^m \to M_N(\C)_{sa}^m$ for $0 \leq s \leq t < +\infty$ such that
	\begin{align}
	F_{t,t}(x) &= x \text{ for } t \in [0,+\infty) \label{eq:transportODE} \\
	\partial_s F_{s,t}(x) &= \frac{1}{2} DV_s(F_{s,t}(x)) \text{ for } s, t \in [0,+\infty). \nonumber
	\end{align} 
	\item $F_{t_1,t_2} \circ F_{t_2,t_3} = F_{t_1,t_3}$ and in particular $F_{t,s} = F_{s,t}^{-1}$.
	\item $(F_{s,t})_* \mu_t = \mu_s$.
\end{enumerate}
\end{proposition}

\begin{proof}
(1) Because $V_s \in \mathcal{E}(c(1+cs)^{-1},C(t+Cs)^{-1})$, we know that $DV_s(x)$ is $C$-Lipschitz with respect to $\norm*{\cdot}_2$.  Hence, given $t \in [0,+\infty)$, by the Picard-Lindel\"of theorem, the initial value problem \eqref{eq:transportODE} has a solution defined for all $s \in [0,+\infty)$.

(2) Fix $t_1$, $t_2$, and $t_3$ and fix $x \in M_N(\C)_{sa}^m$.  Let $G(t)$ be the function defined by $G(t_3) = x$ and $\partial_t G(t) = DV_t(G(t))$.  By definition of the functions $F_{s,t}$, we have $G(t_1) = F_{t_1,t_3}(x)$ and $G(t_2) = F_{t_2,t_3}(x)$.  So $G$ also satisfies the initial value problem $\partial_t G(t) = DV_t(G(t))$ and $G(t_2) = F_{t_2,t_3}(x)$.  Therefore, $G(t_1) = F_{t_1,t_2}(F_{t_2,t_3}(x))$, so that $F_{t_1,t_2}(F_{t_2,t_3}(x)) = F_{t_1,t_3}(x)$.

(3) We first prove the claim for $s, t > 0$.  Because $V_s$ is smooth, it follows that $F_{s,t}$ is smooth for $s, t > 0$ by standard theory of smooth dependence for ODE.  Let $JF_{s,t}$ denote the Jacobian linear transformation (differential) of $F_{s,t}$.  Let $\rho_t = (1/Z)e^{-V_t}$ is the density of $\mu_t$.  As a consequence of the change-of-variables formula for multivariable integrals, we see that the density of $(F_{s,t})_* \mu_t = (F_{t,s}^{-1})_* \mu_t$ is 
\[
(\rho_t \circ F_{t,s}) |\det JF_{t,s}| = \frac{1}{Z} \exp\left[ -N^2\left(V_t \circ F_{t,s} - \frac{1}{N^2} \log |\det JF_{t,s}| \right) \right].
\]
Fix $s$.  If $t = s$, then clearly this reduces to $\rho_s$.  Therefore, it suffices to show that $(\rho_t \circ F_{t,s}) |\det JF_{t,s}|$ is a constant function of $t$, or equivalently that
\[
\partial_t \left[ V_t \circ F_{t,s} - \frac{1}{N^2} \log |\det JF_{t,s}| \right] = 0.
\]
Recalling smoothness $V_t$ and $F_{t,s}$ for $s, t > 0$ and using the differential equations \eqref{eq:simpleHJB} for $V_t$ and \eqref{eq:transportODE} for $F_{t,s}$, we obtain
\begin{align*}
\partial_t[V_t \circ F_{t,s}] &= (\partial_t V_t) \circ F_{t,s} + \ip{DV_t(F_{t,s}), \partial_t F_{t,s}}_2 \\
&= \frac{1}{2N} \Delta V_t \circ F_{t,s} - \frac{1}{2} \norm*{DV_t(F_{t,s})}_2^2 + \frac{1}{2} \ip{DV_t(F_{t,s}), DV_t(F_{t,s})}_2 \\
&= \frac{1}{2N} \Delta V_t.
\end{align*}
Meanwhile, to compute $\partial_t \log |\det JF_{t,s}|$, note that for small $\epsilon \in \R$,
\[
JF_{t+\epsilon,s} = J(F_{t+\epsilon,t} \circ F_{t,s}) = JF_{t+\epsilon,t}(F_{t,s}) JF_{t,s},
\]
so that
\[
\log |\det JF_{t+\epsilon,s}| = \log |\det JF_{t,s}| + \log |\det JF_{t+\epsilon,t}(F_{t,s})|.
\]
Using smoothness,
\begin{align*}
\frac{d}{d\epsilon} \bigr|_{\epsilon = 0} JF_{t+\epsilon,t} &= J \left( \frac{d}{d\epsilon} \bigr|_{\epsilon = 0}  F_{t+\epsilon,t} \right) \\
&= \frac{1}{2} J(DV_t) = \frac{1}{2} N J(\nabla V_t),
\end{align*}
Since $JF_{t+\epsilon,t}$ becomes the identity when $\epsilon = 0$, we know that for small enough $\epsilon$, the linear transformation $JF_{t+\epsilon,t}$ has positive determinant and $\log JF_{t+\epsilon,t}$ is well-defined by power series, so that
\begin{align*}
\frac{d}{d\epsilon} \Bigr|_{\epsilon = 0} \log |\det JF_{t+\epsilon,t} &= \frac{d}{d\epsilon} \bigr|_{\epsilon = 0} \Tr \log JF_{t+\epsilon,t} \\
&= \Tr \left( \frac{d}{d\epsilon} \Bigr|_{\epsilon = 0} JF_{t+\epsilon,t} \right) \\
&= \Tr( N(J \nabla V_t)) = \frac{1}{2} N \Delta V_t.
\end{align*}
Hence, $\partial_t \log |\det JF_{t,s}| = \frac{N}{2} \Delta V_t \circ F_{t,s}$.  This implies that
\[
\partial_t \left[ V_t \circ F_{t,s} - \frac{1}{N^2} \log |\det JF_{t,s}| \right] = \frac{1}{2N} \Delta V_t - \frac{1}{N^2} \cdot \frac{N}{2} \Delta V_t \circ F_{t,s} = 0,
\]
completing the proof of the claim for $s, t > 0$.  The equality $(F_{s,t})_* \mu_t = \mu_s$ extends to the case where $s$ or $t$ is zero because both sides depend continuously on $s$ and $t$ with respect to the weak topology on measures.
\end{proof}

In particular, the map $F_{0,t}$ transports $\mu_t = \mu * \sigma_{t,N}$ to the original law $\mu$.  In other words, if $X \sim \mu$ and $S \sim \sigma_{1,N}$, then $F_{0,t}(X + t^{1/2}S) \sim X$ and $F_{t,0}(X) \sim X + t^{1/2} S$.  This implies that $(1 + t)^{-1/2} F_{t,0}(X) \sim (1 + t)^{-1/2}(X + t^{1/2} S)$.  This suggests that we can find a transport map from the law of $X$ to the law of $S$ as the large $t$ limit of $(1 + t)^{-1/2} F_{t,0}$.  In the interest of efficiency, we postpone the details of this argument until after we introduce the dependence on the other set of parameters $y$.

\subsection{Conditional Hamilton-Jacobi-Bellman Equation and Semigroups} \label{subsec:conditionalHJB}

Let us now fix $N$ and fix a potential $V: M_N(\C)_{sa}^m \times M_N(\C)_{sa}^n \to \R$ in $\mathcal{E}_{m+n}^{(N)}(c,C)$ for some $0 \leq c \leq C$.  Let $\mu$ be the corresponding law and let $(X,Y)$ be a random variable in $M_N(\C)_{sa}^m \times M_N(\C)_{sa}^n$ distributed according to $\mu$.  Let $\mu_t$ be the law of $(X + t^{1/2}S,Y)$, where $S$ is an independent tuple of independent GUE.

Our goal is to transport the law $\mu_s$ to the law $\mu_t$.  Upon freezing the variable $y$, the methods of the previous section will produce a transport map $F_{s,t}(x,y)$ such that $F_{s,t}(\cdot,y)$ pushes forward the conditional distribution of $X + s^{1/2} S$ given $Y$ to the conditional distribution of $X + t^{1/2} S$ given $Y$.  Specifically, $F_{s,t}: M_N(\C)_{sa}^m \times M_N(\C)_{sa}^n \to M_N(\C)_{sa}^m$ is the solution to the initial value problem
\begin{align*}
F_{t,t}(x,y) &= x, \\
\partial_s F_{s,t}(x,y) &= D_x V_s(F_{s,t}(x,y),y).
\end{align*}
Then $(F_{s,t}(X + t^{1/2} S, Y),Y) \sim (X+s^{1/2} S, Y)$.

We seek to understand the large $t$ and large $N$ behavior of $F_{s,t}(x,y)$ as a function of $(x,y)$ rather than simply as a function of $x$ for a fixed $y$.  To achieve this, we must understand the behavior of $V(x,y)$ and $D_x V(x,y)$ as a functions of $(x,y)$.  We will first import the results of \cite[\S 6]{Jekel2018} regarding $D_x V(x,y)$ as a function of $x$, then we will extend them to handle the dependence on $y$.

The potential $V_t$ satisfies
\begin{equation} \label{eq:conditionalHJB}
\partial_t V_t = \frac{1}{2N} \Delta_x V_t - \frac{1}{2} \norm*{D_x V_t}_2^2.
\end{equation}
We express $V_t = R_t V$, where $R_t$ is a semigroup acting on convex and semi-concave functions defined as follows.  Let
\begin{align*}
P_t u(x,y) &= \int_{M_N(\C)_{sa}^m} u(x+z,y)\,d\sigma_{m,t}^{(N)}(z) \\
Q_t u(x,y) &= \inf_{z} \left[u(z,y) + \frac{1}{2t} \norm*{z - x}_2^2 \right].
\end{align*}
Then as suggested by Trotter's formula, we want to express $R_t u = \lim_{n \to \infty} (P_{t/n} Q_{t/n})^n u$, but for technical convenience we only apply this to dyadic rationals $t$ and values of $n$ that are powers of $2$.  The following is a direct application of \cite[Theorems 6.1 and 6.17]{Jekel2018} to $V(\cdot,y)$.

\begin{theorem} \label{thm:Rtsemigroup}
There exists a semigroup of nonlinear operators $R_t: \bigcup_{C > 0} \mathcal{E}_{m+n}^{(N)}(0,C) \to \bigcup_{C > 0} \mathcal{E}_{m+n}^{(N)}(0,C)$ with the following properties:
\begin{enumerate}[(1)]
	\item {\bf Change in Convexity:} If $u(\cdot,y) \in \mathcal{E}_m^{(N)}(c,C)$, then $R_t u(\cdot,y) \in \mathcal{E}_m^{(N)}(c(1+ct)^{-1}, C(1+Ct)^{-1})$.
	\item {\bf Approximation by Iteration:} For $\ell \in \Z$ and $t \in 2^{-\ell} \N_0$, denote $R_{t,\ell}u = (P_{2^{-\ell}} Q_{2^{-\ell}})^{2^\ell t} u$.  Suppose $t \in \Q_2^+$ and $u \in \mathcal{E}_{m+n}^{(N)}(0,C)$.
		\begin{enumerate}[(a)]
			\item If $2^{-\ell-1} C \leq 1$, then
			\[
			|R_t u - R_{t,\ell} u| \leq \left( \frac{3}{2} \frac{C^2mt}{1+Ct} + \log(1 + Ct) (m + Cm + \norm*{D_x u}_2^2) \right) 2^{-\ell}.
			\]
			\item $\displaystyle \norm*{D_x(R_{t,\ell} u) - D_x(R_t u)}_{L^\infty} \leq [t/2 + C(t/2)^2] C^2 m^{1/2}(2 \cdot 2^{-\ell/2} + 2^{-3\ell/2}C)$.
	\end{enumerate}
	\item {\bf Continuity in Time:}  Suppose $s \leq t \in \R_+$ and $u \in \mathcal{E}_{m+n}^{(N)}(0,C)$.  Then
		\begin{enumerate}[(a)]
			\item $R_t u \leq R_s u + \frac{m}{2} [\log(1 + Ct) - \log(1 + Cs)]$.
			\item $R_t u \geq R_s u - \frac{1}{2} (t - s)(Cm + \norm*{D_x u}_2^2)$.
			\item If $C(t - s) \leq 1$, then $\norm*{D_x(R_t u) - D_x(R_s u)}_2 \leq 5Cm^{1/2} 2^{1/2}(t - s)^{1/2} + C(t - s) \norm*{D_x u}_2$.
		\end{enumerate}
	\item {\bf Differential Equation:} $R_t u(x)$ is continuous as a function of $(x,t)$ on $M_N(\C)_{sa}^m \times [0,+\infty)$ and smooth on $M_N(|C)_{sa}^m \times (0,+\infty)$, and it satisfies \eqref{eq:conditionalHJB}, and we have $P_t[\exp(-N^2u)] = \exp(-N^2 R_t u)$.
\end{enumerate}
\end{theorem}

Result (1) regarding convexity and semi-concavity only applies to $R_t u$ as a function of $x$ for a fixed $y$.  We now extend this result to control the dependence on $y$, using the same techniques as in \cite[Lemma 6.6]{Jekel2018}.  As remarked in that paper, this type of analysis of $Q_t$ is standard in the PDE literature on viscosity solutions.

We use the following notation, as in Definition \ref{def:convexityHnotation}:  Consider a function $u(x,y)$ on $M_N(\C)_{sa}^m \times M_N(\C)_{sa}^n$.  Let us write $Hu \geq c I_m \oplus c' I_n$ to mean that
\[
u(x,y) - \frac{c}{2} \norm*{x}_2^2 - \frac{c'}{2} \norm*{y}_2^2 \text{ is convex}
\]
and similarly let us write $Hu \leq C I_m \oplus C' I_n$ to mean that
\[
u(x,y) - \frac{C}{2} \norm*{x}_2^2 - \frac{C'}{2} \norm*{y}_2^2 \text{ is concave.}
\]

\begin{lemma} \label{lem:Qt}
Suppose that $u: M_N(\C)_{sa}^m \times M_N(\C)_{sa}^n \to \R$ and that
\[
c I_m \oplus c' I_n \leq Hu \leq C I_m \oplus C' I_n.
\]
Then
\begin{enumerate}
	\item $c I_m \oplus c' I_n \leq H(P_t u) \leq C I_m \oplus C' I_n$.
	\item $D(Q_t u)(x,y) = Du(x - t D_x(Q_t u)(x,y), y)$.
	\item $c(1 + ct)^{-1} I_m \oplus c' I_n \leq H(Q_t u) \leq C(1 + Ct)^{-1} I_m \oplus C' I_n$.
	\item $c(1 + ct)^{-1} I_m \oplus c' I_n \leq H(R_t u) \leq C(1 + Ct)^{-1} I_m \oplus C' I_n$.
\end{enumerate}
\end{lemma}

\begin{proof}
(1) This is left as an exercise.

(2) The proof is a modification of that of \cite[Lemma 6.6]{Jekel2018}, which proves an analogous result in the simpler case of functions of $x$ without the extra variable $y$.  Fix $x_0$ and $y_0$.  Because the function $u(z,y_0) + \frac{1}{2t} \norm*{z - x_0}_2^2$ is uniformly convex with respect to $z$, it has a unique minimizer $z_0$.  This minimizer must be a critical point with respect to the first variable, and hence
\[
D_x u(z_0,y_0) + \frac{1}{t}(z_0 - x_0),
\]
that is,
\[
z_0 = x_0 - t D_x u(z_0,y_0).
\]
Let $p = D_x u(z_0,y_0)$ and $q = D_y u(z_0,y_0)$, so that $Du(z_0,y_0) = (p,q)$.  Our assumption $c I_m \oplus c' I_n\leq Hu \leq C I_m \oplus C' I_n$ implies that
\[
\underline{v}(x,y) \leq u(x,y) \leq \overline{v}(x,y),
\]
where
\begin{align*}
\underline{v}(x,y) &= u(z_0,y_0) + \ip{p, x - z_0} + \ip{q, y - y_0} + \frac{c}{2} \norm*{x - z_0}_2^2 + \frac{c'}{2} \norm*{y - y_0}_2^2 \\
\overline{v}(x,y) &= u(z_0,y_0) + \ip{p, x - z_0} + \ip{q, y - y_0} + \frac{C}{2} \norm*{x - z_0}_2^2 + \frac{C'}{2} \norm*{y - y_0}_2^2.
\end{align*}
Note that $\underline{v} \leq u \leq v$ implies $Q_t \underline{v} \leq Q_t u \leq Q_t \overline{v}$ since monotonicity of $Q_t$ is immediate from the definition.  One can compute $Q_t \underline{v}$ and $Q_t \overline{v}$ directly as in Lemma 6.4 (2) and the proof of Lemma 6.6 in \cite{Jekel2018} and obtain
\begin{align*}
Q_t \underline{v}(x,y) &= u(z_0,y_0) - \frac{t}{2} \norm*{p}_2^2 + \ip{p, x - z_0} + \ip{q, y - y_0} + \frac{c}{2(1 + ct)} \norm*{x - tp - z_0}_2^2 + \frac{c'}{2} \norm*{y - y_0}_2^2 \\
&= u(z_0,y_0) + \frac{1}{2t} \norm*{z_0 - x_0}_2^2 + \ip{p, x - x_0} + \ip{q, y - y_0} + \frac{c}{2(1 + ct)} \norm*{x - tp - z_0}_2^2 + \frac{c'}{2} \norm*{y - y_0}_2^2 \\
&= Q_t u(x_0,y_0) + \ip{p, x - x_0} + \ip{q, y - y_0} + \frac{c}{2(1 + ct)} \norm*{x - x_0}_2^2 + \frac{c'}{2} \norm*{y - y_0}_2^2,
\end{align*}
where the last two lines following from substituting $z_0 = x_0 - tp$ and that the infimum defining $Q_t u$ is achieved at $z_0$.  The analogous formula for $Q_t \overline{v}(x,t)$ holds as well.  The functions $Q_t \underline{v}$ and $Q_t \overline{v}$ thus provide second-order Taylor expansions from above and below for the function $Q_t u$ with respect to $(x,y)$ at the point $(x_0,y_0)$.  Looking at the first-order terms in the expansions shows that $Q_t u$ is differentiable at $(x_0,y_0)$ with
\[
D(Q_t u)(x_0,y_0) = (p,q) = Du(z_0,x_0) = Du(x_0 - tp, y_0) = Du(x_0 - t D_x(Q_t u)(x_0,y_0),y_0),
\]
which proves (2).

(3)  We examine the second-order terms of upper and lower Taylor expansions $Q_t \underline{v}$ and $Q_t \overline{v}$ and apply the claim (2) $\implies$ (1) from Lemma \ref{lem:convexgradient}.  This is the same argument as in the proof of \cite[Proposition 2.13 (2)]{Jekel2018}.

(4) Recall that if $A \leq Hu \leq B$, then $A \leq H(P_tu) \leq B$.  Using this fact together with (3) iteratively, we see that if $t$ is a dyadic rational and $t = 2^{-\ell} k$, then
\[
c(1 + ct)^{-1} I_m \oplus c' I_n \leq H(R_{t,\ell} u) \leq C(1 + Ct)^{-1} I_m \oplus C' I_n.
\]
In light of Theorem \ref{thm:Rtsemigroup} (2), this will also hold in the limit as $\ell \to \infty$, since for any two self-adjoint matrices $A$ and $B$, the family of functions with $A \leq Hu \leq B$ is closed under pointwise limits.  Similarly, using Theorem \ref{thm:Rtsemigroup} (3), we extend this to all real $t \geq 0$.
\end{proof}

\begin{remark}
The convexity conditions of Lemma \ref{lem:Qt} (4) can alternatively be deduced from \cite[Theorem 4.3]{BL1976}.  However, it is convenient for us to use Theorem \ref{thm:Rtsemigroup} here because we want the dimension-independent time-continuity estimates Theorem \ref{thm:Rtsemigroup} (3) in the proof of Theorem \ref{thm:transport1} below.
\end{remark}

\subsection{Lipschitz Estimates for Conditional Transport} \label{subsec:Lipschitztransport}

This subsection proves the technical estimate Lemma \ref{lem:transportLipschitzestimate} on the Lipschitz seminorm of $F_{s,t}$.  This depends crucially on the convexity properties of $V_t(x,y)$.

\begin{lemma} \label{lem:twotermgradientestimate}
\begin{align*}
\ip{D_x V_t(x,y) - D_x V_t(x',y'), x - x'}_2 &\leq \frac{C}{1+Ct} \norm*{x - x'}_2^2 + \frac{C-c}{(1+Ct)^{1/2}(1+ct)^{1/2}} \norm*{x - x'}_2 \norm*{y - y'}_2 \\
\ip{D_x V_t(x,y) - D_x V_t(x',y'), x - x'}_2 &\geq \frac{c}{1+ct} \norm*{x - x'}_2^2 - \frac{C-c}{(1+Ct)^{1/2}(1+ct)^{1/2}} \norm*{x - x'}_2 \norm*{y - y'}_2.
\end{align*}
\end{lemma}

\begin{proof}
First, note that
\begin{multline} \label{eq:twoterms}
\ip{D_x V_t(x,y) - D_x V_t(x',y'), x - x'}_2 \\
= \ip{D_x V_t(x,y) - D_x V_t(x',y), x - x'}_2 + \ip{D_x V_t(x',y) - D_x V_t(x',y'), x - x'}_2
\end{multline}
By Lemma \ref{lem:convexgradient}, the first term on the right hand side of \eqref{eq:twoterms} can be estimated by
\[
\frac{c}{1+ct} \norm*{x - x'}_2^2 \leq \ip{D_x V_t(x,y) - D_x V_t(x',y), x - x'}_2 \leq \frac{C}{1+Ct} \norm*{x - x'}_2^2.
\]
To handle the second term on the right hand side of \eqref{eq:twoterms}, define
\begin{align*}
\overline{V}_t(x,y) &= V_t(x,y) - \frac{c}{2(1+ct)} \norm*{x}_2^2 - \frac{c}{2} \norm*{y}_2^2 \\
\underline{V}_t(x,y) &= V_t(x,y) - \frac{C}{2(1+Ct)} \norm*{x}_2^2 - \frac{C}{2} \norm*{y}_2^2
\end{align*}
and recall that $\overline{V}_t$ is convex and $\underline{V}_t$ is concave and in particular
\begin{align*}
0 \leq H\overline{V}_t &\leq \left( \frac{C}{1+Ct} - \frac{c}{1+ct} \right) I_m \oplus (C - c) I_n \\
&= \frac{C-c}{(1+Ct)(1+ct)} I_m \oplus (C - c) I_n.
\end{align*}
Note that
\begin{align*}
D_x V_t(x',y) - D_x V_t(x',y') &= \left(D_x V_t(x',y) - \frac{c}{1+ct} x' \right) - \left( D_x V_t(x',y) - \frac{c}{1+ct} x' \right) \\
&= D_x \overline{V}_t(x',y) - D_x \overline{V}_t(x',y').
\end{align*}
Therefore,
\begin{align*}
\ip{D_x V_t(x',y) - D_x V_t(x',y'), x - x'}_2 &= \ip{D_x \overline{V}_t(x',y) - D_x \overline{V}_t(x',y'), x - x'}_2 \\
&= \ip{D \overline{V}_t(x',y) - D\overline{V}_t(x',y'), (x - x', 0)}_2.
\end{align*}
Now we apply Lemma \ref{lem:convexgradientestimate} to $\overline{V}_t$ with the matrix $A = \frac{C-c}{(1+Ct)(1+ct)} I_m \oplus (C - c) I_n$ and conclude that
\begin{align*}
|\ip{D \overline{V}_t(x',y) - D\overline{V}_t(x',y'), (x - x', 0)}_2| &\leq \left( (C-c) \norm*{y-y'}_2^2 \right)^{1/2} \left( \frac{C - c}{(1+Ct)(1+ct)} \norm*{x - x'}_2^2 \right)^{1/2} \\
&\leq \frac{C-c}{(1+Ct)^{1/2}(1+ct)^{1/2}} \norm*{x - x'} \norm*{y - y'}.
\end{align*}
Combining this estimate for the second term of \eqref{eq:twoterms} with our earlier estimate for the first term completes the proof.
\end{proof}

\begin{lemma} \label{lem:transportLipschitzestimate}
We have
\begin{equation} \label{eq:transportLipdx}
\norm*{F_{s,t}}_{\Lip,dx} \leq \begin{cases}
\frac{(1+Cs)^{1/2}}{(1+Ct)^{1/2}}, & s \geq t \\
\frac{(1+cs)^{1/2}}{(1 + ct)^{1/2}} & s \leq t. \end{cases}
\end{equation}
and
\begin{equation} \label{eq:transportLipdy}
\norm*{F_{s,t}}_{\Lip,dy} \leq \begin{cases}
(C/c - 1)(1 + Cs)^{1/2}\left( \frac{1}{(1+Ct)^{1/2}} - \frac{1}{(1+Cs)^{1/2}} \right), & s \geq t, \\
(C/c - 1)(1 + cs)^{1/2} \left( \frac{1}{(1+Cs)^{1/2}} - \frac{1}{(1+Ct)^{1/2}} \right) & s \leq t.
\end{cases}
\end{equation}
\end{lemma}

\begin{proof}
Fix $t \geq 0$ and $(x,y)$ and $(x',y')$ in $M_N(\C)_{sa}^m \times M_N(\C)_{sa}^n$ and define
\[
\phi(s) = \norm*{F_{s,t}(x,y) - F_{s,t}(x',y')}_2.
\]
Note that $\phi$ is locally Lipschitz, hence absolutely continuous.  Also,
\begin{align*}
2 \phi(s) \phi'(s) &= \partial_s [\phi(s)^2] \\
&= 2 \ip{\partial_s F_{s,t}(x,y) - \partial_s F_{s,t}(x',y'), F_{s,t}(x,y) - F_{s,t}(x',y')}_2 \\
&= \ip{DV_s(F_{s,t}(x,y),y) - DV_s(F_{s,t}(x',y'),y'), F_{s,t}(x,y) - F_{s,t}(x',y')}_2 \\
&\leq \frac{C}{1+Ct} \norm*{F_{s,t}(x,y) - F_{s,t}(x',y')}_2^2 \\
& \quad + \frac{C-c}{(1+Ct)^{1/2}(1+ct)^{1/2}} \norm*{F_{s,t}(x,y) - F_{s,t}(x',y')}_2 \norm*{y - y'}_2 \\
&= \frac{C}{1+Cs} \phi(s)^2 + \frac{C-c}{(1+Cs)^{1/2}(1+cs)^{1/2}} \phi(s) \norm*{y - y'}_2,
\end{align*}
where we have applied Lemma \ref{lem:twotermgradientestimate}.  It follows that whenever $\phi(s) > 0$,
\[
\phi'(s) \leq \frac{C}{2(1+Cs)} \phi(s) + \frac{C-c}{2(1+Cs)^{1/2}(1+cs)^{1/2}} \norm*{y - y'}_2.
\]
On the other hand, since $\phi(s) \geq 0$, any point where $\phi$ is zero and $\phi$ is differentiable must be a critical point, so when $\phi(s) = 0$ the estimate is vacuously true.  This inequality implies
\begin{align*}
\frac{d}{ds} \left[ \frac{1}{(1 + Cs)^{1/2}} \phi(s) \right] &\leq \frac{C-c}{2(1+Cs)(1+cs)^{1/2}} \norm*{y - y'}_2 \\
&\leq \frac{C(C-c)}{2c(1+Cs)^{3/2}} \norm*{y - y'}_2,
\end{align*}
where in the last line we have observed that $(1 + cs)^{1/2} \geq (c/C)^{1/2} (1 + Cs)^{1/2} \geq (c/C)(1 + Cs)^{1/2}$.  Hence for $s \geq t$
\[
\frac{1}{(1+Cs)^{1/2}} \phi(s) - \frac{1}{(1+Ct)^{1/2}} \phi(t) \leq \frac{C-c}{c} \left( \frac{1}{(1+Ct)^{1/2}} - \frac{1}{(1+Cs)^{1/2}} \right) \norm*{y - y'}_2
\]
Now we substitute $\phi(s) = \norm*{F_{s,t}(x,y) - F_{s,t}(x',y')}_2$ and $\phi(t) = \norm*{x - x'}_2$ and rearrange to obtain
\[
\frac{1}{(1+Cs)^{1/2}} \norm*{F_{s,t}(x,y) - F_{s,t}(x',y')}_2 \leq \frac{1}{(1+Ct)^{1/2}} \norm*{x - x'}_2 + \frac{C-c}{c} \left( \frac{1}{(1+Ct)^{1/2}} - \frac{1}{(1+Cs)^{1/2}} \right) \norm*{y - y'}_2.
\]
This proves the asserted estimates in the case where $s \geq t$.  The argument for the case $s \leq t$ is similar.  Here we use the lower bound rather than the upper bound in Lemma \ref{lem:twotermgradientestimate} and get
\[
\phi'(s) \geq \frac{c}{2(1+cs)} \phi(s) - \frac{C-c}{2(1+Cs)^{1/2}(1+cs)^{1/2}} \norm*{y - y'}_2
\]
so that
\begin{align*}
\frac{d}{ds} \left[ \frac{1}{(1+cs)^{1/2}} \phi(s) \right] &\geq -\frac{C-c}{2(1+Cs)^{1/2}(1+cs)} \norm*{y - y'}_2 \\
&\geq -\frac{C(C-c)}{2c(1+Cs)^{3/2}} \norm*{y - y'}_2.
\end{align*}
Now we take $s \leq t$ and obtain
\begin{multline*}
\frac{1}{(1+ct)^{1/2}} \norm*{x - x'}_2 - \frac{1}{(1+cs)^{1/2}} \norm*{F_{s,t}(x,y) - F_{s,t}(x',y')}_2 \\
\geq - \frac{C-c}{c} \left( \frac{1}{(1+Cs)^{1/2}} - \frac{1}{(1+Ct)^{1/2}} \right) \norm*{y - y'}_2,
\end{multline*}
which yields the desired estimates.
\end{proof}

\subsection{Transport in the Large $t$ Limit} \label{subsec:largeTtransport}

We remind the reader here that we are still working in the finite-dimensional setting for a fixed value of $N$ which is suppressed in the notation.  To understand the large $t$ limit of our transport maps, consider the renormalized law
\begin{equation} \label{eq:renormalizedlaw}
\tilde{\mu}_t := \text{law of } (\tilde{X}_t, Y) := \left( e^{-t/2} X + (1 - e^{-t})^{1/2} S, Y \right).
\end{equation}
A brief computation shows that the corresponding potential is
\begin{equation} \label{eq:renormalizedpotential}
\tilde{V}_t(x,y) := V_{e^t - 1}(e^{t/2} x, y),
\end{equation}
(here the potential is only well-defined up to an additive constant because the probability measure $\tilde{\mu}_t$ includes a normalizing constant $1/\tilde{Z}_t$ anyway, so we made a convenient choice of the additive constant).  This potential satisfies the equation
\[
\partial_t \tilde{V}_t = \frac{1}{2N} \Delta_x \tilde{V}_t - \frac{1}{2} \norm*{D_x \tilde{V}_t}_2^2 + \frac{1}{2} \ip{D_x \tilde{V}_t, x}_2.
\]
We remark that if $\tilde{\rho}_t = (1/Z_t) e^{-N^2 \tilde{V}_t}$ is the density at time $t$ and $r(x,y) = \text{const} e^{-\norm{(x,y)}_2^2/2}$ is the Gaussian density, then
\begin{align*}
\partial_t \tilde{\rho}_t &= \frac{1}{2N} \Delta_x \tilde{\rho}_t + \frac{1}{2} \ip{\nabla_x \tilde{\rho}_t, x}_{\Tr} + \frac{Nm}{2} \tilde{\rho}_t \\
&= \frac{1}{2N} \Div_x \left( \tilde{\rho}_t \nabla_x \left( \log \frac{\tilde{\rho}_t}{r} \right) \right).
\end{align*}
In other words, $\tilde{\rho}_t$ evolves according to the diffusion semigroup with respect to Gaussian measure (compare equation (33) of \cite{OV2000}), while the heat equation represents diffusion with respect to Lebesgue measure.

The transport functions are renormalized as follows.  Because $(F_{s,t}(x,y),y)$ pushes forward $\mu_t$ to $\mu_s$, we may compute that $(\tilde{F}_{s,t}(x,y),y)$ pushes forward $\tilde{\mu}_t = \tilde{\mu}_s$, where
\begin{equation} \label{eq:renormalizedtransport}
\tilde{F}_{s,t}(x,y) := e^{-s/2} F_{e^s-1,e^t-1}(e^{t/2}x,y).
\end{equation}
Moreover, from the differential equation \eqref{eq:transportODE}, we deduce that
\begin{equation} \label{eq:renormalizedtransportODE}
\partial_s \tilde{F}_{s,t}(x,y) = \frac{1}{2} \left( D_x \tilde{V}_s(\tilde{F}_{s,t}(x,y),y) - \tilde{F}_{s,t}(x,y) \right).
\end{equation}

As $t \to \infty$, the law $\tilde{\mu}_t$ converges to the law of $(S,Y)$, which we denote $\tilde{\mu}_\infty$.  Thus, if we show that $\tilde{F}_{s,t}$ has a limit as $s \to +\infty$ or $t \to +\infty$, we will be able to transport our given law $\mu = \tilde{\mu}_0$ of $(X,Y)$ to the law of $(S,Y)$.  As the first step, we deduce from Lemma \ref{lem:transportLipschitzestimate} the following Lipschitz estimates on $\tilde{F}_{s,t}$ which are uniform in $s$ and $t$.  Note also that the coefficient of $\norm*{y - y'}_2$ goes to zero as $s, t \to \infty$.

\begin{lemma} \label{lem:transportLipschitzestimate2}
We have
\begin{equation} \label{eq:transportLipdx2}
\norm*{\tilde{F}_{s,t}}_{\Lip,dx} \leq \max(C,1/c)^{1/2}
\end{equation}
and
\begin{equation} \label{eq:transportLipdy2}
\norm*{\tilde{F}_{s,t}}_{\Lip,dy} \leq (C/c - 1) \max(C,1/C)^{3/2} |e^{-s/2} - e^{-t/2}|.
\end{equation}
In particular,
\begin{equation} \label{eq:transportLip2}
\norm*{\tilde{F}_{s,t}}_{\Lip} \leq \max(C,1/c)^{7/2}.
\end{equation}
\end{lemma}

\begin{proof}
For the first estimate, for the case where $s \geq t$, direct substitution of \eqref{eq:renormalizedtransport} into \eqref{eq:transportLipdx} of Lemma \ref{lem:transportLipschitzestimate} shows that
\[
\norm*{\tilde{F}_{s,t}}_{\Lip,dx} \leq e^{-s/2} \frac{(1 + C(e^s - 1))^{1/2}}{(1 + C(e^t - 1))^{1/2}} e^{t/2} = \frac{(C + (1 - C)e^{-s})^{1/2}}{(C + (1 - C)e^{-t})^{1/2}}.
\]
The function $C + (1 - C) e^{-s}$ is clearly monotone on $[0,+\infty)$ and achieves the values $1$ and $C$ at $0$ and $+\infty$ respectively, and hence is between $\min(1,C)$ and $\max(1,C)$.  Hence,
\[
\norm*{\tilde{F}_{s,t}}_{\Lip,dx} \leq \frac{\max(1,C)^{1/2}}{\min(1,C)^{1/2}} = \max(C,1/C)^{1/2} \leq \max(C,1/c)^{1/2}.
\]
The case where $s \leq t$ follows by the same argument, where the bound this time is $\max(c,1/c)^{1/2} \leq \max(C,1/c)^{1/2}$.

For the second estimate, we apply \eqref{eq:transportLipdy}.  Note in \eqref{eq:transportLipdy}, in the case $s \leq t$, we may use $(1 + cs)^{1/2} \leq (1 + Cs)^{1/2}$ and thus in both cases $s \geq t$ or $s \leq t$,
\begin{align*}
\norm{F_{s,t}}_{\Lip,dy} &\leq (C/c - 1) (1 + Cs)^{1/2} \left| \frac{1}{(1 + Cs)^{1/2}} - \frac{1}{(1 + Ct)^{1/2}} \right| \\
&= (C/c - 1)(1 + Cs)^{1/2} \left| \int_s^t  \frac{C}{2(1 + Cu)^{3/2}}\,du \right|
\end{align*}
This implies that
\begin{align*}
\norm*{\tilde{F}_{s,t}}_{\Lip,dy} &\leq e^{-s/2} (C/c - 1)(1 + C(e^s - 1))^{1/2} \left| \int_{e^s-1}^{e^t-1}  \frac{C}{2(1 + Cu)^{3/2}}\,du \right| \\
&= (C/c - 1)e^{-s/2}(1 + C(e^s - 1))^{1/2} \left| \int_s^t  \frac{Ce^w}{2(1 + C(e^w - 1))^{3/2}}\,dw \right| \\
&\leq (C/c - 1) \max(1,C)^{1/2} \left| \int_s^t  \frac{Ce^w}{2 \min(1,C)^{3/2} e^{3w/2}}\,dw \right| \\
&\leq (C/c - 1) \frac{\max(1,C)^{1/2} C}{\min(1,C)^{3/2}} |e^{-t/2} - e^{-s/2}| \\
&\leq (C/c - 1) \max(C,1/C)^{3/2} |e^{-t/2} - e^{-s/2}|.
\end{align*}
where we have again applied $\min(1,C) e^s \leq 1 + C(e^s - 1) \leq \max(1,C) e^s$.

For the last estimate \eqref{eq:transportLip2}, observe that
\begin{align*}
\norm*{\tilde{F}_{s,t}}_{\Lip} &\leq \norm*{\tilde{F}_{s,t}}_{\Lip,dx} + \norm*{\tilde{F}_{s,t}}_{\Lip,dy} \\
&\leq \max(C,1/c)^{1/2} + (C/c - 1) \max(C,1/C)^{3/2} |e^{-s/2} - e^{-t/2}| \\
&\leq \max(C,1/c)^{3/2} + (\max(C,1/c)^2 - 1) \max(C,1/c)^{3/2} \\
&= \max(C,1/c)^{7/2}. \qedhere
\end{align*}
\end{proof}

\begin{lemma} \label{lem:transportLipschitzestimate3}
Let $\pi_1$ denote the function $\pi_1(x,y) = x$.  Then
\begin{equation} \label{eq:transportLipdx3}
\norm*{\tilde{F}_{s,t} - \pi_1}_{\Lip,dx} \leq  \frac{1}{2} (\max(C,1/c) - 1) \max(C,1/C)^{1/2} |e^{-s} - e^{-t}|
\end{equation}
and
\begin{equation} \label{eq:transportLip3}
\norm*{\tilde{F}_{s,t} - \pi_1}_{\Lip} \leq (\max(C,1/c)^3 - 1) \max(C,1/c)^{1/2} |e^{-s/2} - e^{-t/2}|.
\end{equation}
\end{lemma}

\begin{proof}
Let $U_s(x,y) = \tilde{V}_s(x,y) - (1/2) \norm*{x}_2^2$.  Then \eqref{eq:renormalizedtransportODE} says that
\[
\partial_s \tilde{F}_{s,t}(x,y) = \frac{1}{2} D_x U_s(\tilde{F}_{s,t}(x,y),y).
\]
Moreover, we have
\[
\frac{ce^s}{1 + c(e^s - 1)} \leq H_x \tilde{V}_s \leq \frac{Ce^s}{1 + C(e^s - 1)}.
\]
We can bound $H_x U_s$ above and below by subtracting $1$ from both sides, which after some computation reduces to
\[
\frac{c - 1}{1 + c(e^s - 1)} \leq H_x U_s \leq \frac{C - 1}{1 + C(e^s - 1)}.
\]
Therefore, we have $-L \leq H_x U_s \leq L$ where
\[
L = \max \left( -\frac{c - 1}{1 + c(e^s - 1)}, \frac{C - 1}{1 + C(e^s - 1)} \right).
\]
We claim that $L \leq L' := (\max(C,1/c) - 1) e^{-s}$.  If the first term $(1 - c) / (1 + c(e^s - 1))$ is negative, then it is $\leq L'$ automatically, but if it is positive, then $c \leq 1$ and hence
\[
\frac{1 - c}{1 + c(e^s - 1)} \leq \frac{1 - c}{c + c(e^s - 1)} = (1/c - 1)e^{-s} \leq (\max(C,1/c) - 1)e^{-s}.
\]
Similarly, if $(C - 1) / (1 + C(e^s - 1))$ is negative, there is nothing to prove, but otherwise $C \geq 1$, and hence
\[
\frac{C - 1}{1 + C(e^s - 1)} \leq \frac{C - 1}{1 + (e^s - 1)} = (C - 1) e^{-s} \leq (\max(C,1/c) - 1) e^{-s}.
\]
But $-L' \leq H_x U_s \leq L'$ implies that $D_x U_s$ is $L'$-Lipschitz in $x$.  Therefore,
\begin{align*}
\norm*{\partial_s \tilde{F}_{s,t}(x,y) - \partial_s \tilde{F}_{s,t}(x',y)}_2 &= \frac{1}{2} \norm*{DU_s(\tilde{F}_{s,t}(x,y),y) - DU_s(\tilde{F}_{s,t}(x',y))}_2 \\
&\leq \frac{1}{2}(\max(C,1/c) - 1) e^{-s} \norm*{\tilde{F}_{s,t}(x,y) - \tilde{F}_{s,t}(x',y)}_2.
\end{align*}
Applying \eqref{eq:transportLipdx2} in the case where $s \geq t$, we get
\[
\norm*{\partial_s \tilde{F}_{s,t}(x,y) - \partial_s \tilde{F}_{s,t}(x',y)}_2 \leq \frac{1}{2} (\max(C,1/c) - 1) \max(C,1/c)^{1/2} e^{-s} \norm*{x - x'}_2.
\]
Hence,
\begin{align*}
\norm*{\tilde{F}_{s,t}(x,y) - \tilde{F}_{s,t}(x',y) - (x - x')}_2 &\leq \left| \int_t^s \norm*{\partial_u \tilde{F}_{u,t}(x,y) - \partial_u \tilde{F}_{u,t}(x',y)}_2\,du \right| \\
&\leq \frac{1}{2} (\max(C,1/c) - 1) \max(C,1/c)^{1/2} |e^{-s} - e^{-t}| \norm*{x - x'}_2.
\end{align*}
which proves the desired estimate \eqref{eq:transportLipdx3}.

To check the second estimate \eqref{eq:transportLip3}, first observe
\[
\frac{1}{2} |e^{-s} - e^{-t}| = \int_{\min(s,t)}^{\max(s,t)} \frac{1}{2} e^{-u}\,du \leq \int_{\min(s,t)}^{\max(s,t)} \frac{1}{2} e^{-u/2}\,du = |e^{-s/2} - e^{-t/2}|,
\]
Moreover, $\norm{\tilde{F}_{s,t} - \pi_1}_{\Lip,dy} = \norm{\tilde{F}_{s,t}}_{\Lip,dy}$.  Therefore, using \eqref{eq:transportLipdy2} and \eqref{eq:transportLipdx3},
\begin{align*}
\norm{\tilde{F}_{s,t} - \pi_1}_{\Lip} &\leq \norm{\tilde{F}_{s,t} - \pi_1}_{\Lip,dx} + \norm{\tilde{F}_{s,t} - \pi_1}_{\Lip,dy} \\
&\leq (\max(C,1/c) - 1) \max(C,1/C)^{1/2} \frac{1}{2} |e^{-s} - e^{-t}| +  (C/c - 1) \max(C,1/C)^{3/2} |e^{-s/2} - e^{-t/2}| \\
&\leq [(\max(C,1/c) - 1) \max(C,1/c)^{1/2} + (\max(C,1/c)^2 - 1) \max(C,1/c)^{3/2}] |e^{-s/2} - e^{-t/2}| \\
&= (\max(C,1/c)^3 - 1) \max(C,1/c)^{1/2} |e^{-s/2} - e^{-t/2}|. \qedhere.
\end{align*}
\end{proof}

\begin{proposition} \label{prop:largeTtransport}
The limits $\tilde{F}_{s,\infty} := \lim_{t \to \infty} \tilde{F}_{s,t}$ and $\tilde{F}_{\infty,t} = \lim_{s \to \infty} \tilde{F}_{s,t}$ exist for $s, t \geq 0$.  More precisely, let $(X,Y)$ and $(\tilde{X}_t,Y)$ be a pair of random variables with the laws $\mu$ and $\tilde{\mu}_t$ as above.  Then
\begin{multline} \label{eq:transportconvergenceestimate}
\norm*{\tilde{F}_{s,\infty}(x,y) - \tilde{F}_{s,t}(x,y)}_2 \leq \max(C,1/c)^{1/2} e^{-t/2} \norm*{E(X)}_2 \\ + e^{-t/2}(\max(C,1/c)^3 - 1) \max(C,1/c) \left( \norm{(x, y - E(Y))}_2 + (m + \Var(Y))^{1/2} \right)
\end{multline}
and
\begin{multline}
\norm*{\tilde{F}_{s,t}(x,y) - \tilde{F}_{\infty,t}(x,y)}_2 \leq \frac{1}{2} (\max(C,1/c) - 1) \max(C,1/c)^{1/2} e^{-s} \biggl( e^{-s/2} \norm{E(X)}_2 \\ + \max(C,1/c)^{7/2} \left( \norm{(x - e^{-t/2}E(X), y - E(Y))}_2 + (e^{-t} \Var(X) + (1 - e^{-t})m + \Var(Y))^{1/2} \right) \biggr)
\end{multline}
The estimates of Lemmas \ref{lem:transportLipschitzestimate2} and \ref{lem:transportLipschitzestimate3} extend to the cases where $s$ or $t$ is infinite, where we define $\tilde{F}_{\infty,\infty}(x,y) = x$.  Moreover, if $(\tilde{X}_t,Y) \sim \tilde{\mu}_t$, then we have the relation $(\tilde{F}_{s,t}(\tilde{X}_t,Y),Y) \sim (\tilde{X}_s,Y)$ when $s, t \in [0,\infty]$.
\end{proposition}

\begin{remark}
We have written the explicit form of the estimates here in order to emphasize that the bounds are dimension-independent; they only depend on the parameters $m$, $n$, $c$, $C$, $\norm{E(X)}_2$, $\norm{E(Y)}_2$, $\Var(X)$, and $\Var(Y)$.  The estimates also become sharper when $c$ and $C$ are close to $1$, which would include the situation where $V(x,y)$ is a perturbation of the quadratic potential $(1/2)[\norm*{x}_2^2 + \norm*{y}_2^2]$.  This perturbative setting was studied first in the literature, for instance in \cite{GMS2006} and \cite{GS2014}; see \cite[\S 8.3]{Jekel2018} for further discussion.
\end{remark}

\begin{proof}
We first consider the case where $s$ is fixed and $t \to +\infty$.  Note that by \eqref{eq:transportLipdx2},
\begin{align} \label{eq:convergenceproof1}
\norm*{\tilde{F}_{s,t'}(x,y) - \tilde{F}_{s,t}(x,y)}_2 &= \norm*{\tilde{F}_{s,t'}(\tilde{F}_{t',t}(x,y),y) - \tilde{F}_{s,t'}(x,y)}_2 \\
&\leq \norm*{\tilde{F}_{s,t'}}_{\Lip,dx} \norm*{\tilde{F}_{t,t'}(x,y) - x}_2 \nonumber \\
&\leq \max(C,1/c)^{1/2} \norm*{\tilde{F}_{t,t'}(x,y) - x}_2. \nonumber
\end{align}
By Lemma \ref{lem:transportLipschitzestimate3},
\[
\norm*{\tilde{F}_{t,t'} - \pi_1}_{\Lip} \leq L |e^{-t/2} - e^{t'/2}|,
\]
where $L = \max(C,1/c)^3 - 1) \max(C,1/c)^{1/2}$.  Then we apply Lemma \ref{lem:RVboundedness} to $G(x,y) = \tilde{F}_{t,t'}(x,y) - x$ with the random variable $(\tilde{X}_{t'},Y)$.  Note that $(\tilde{X}_{t'},Y)$ has mean $(e^{-t'/2} E(X), E(Y))$ and variance $e^{-t'} \Var(X) + (1 - e^{-t'})m + \Var(Y)$.  Moreover,
\[
E[G(\tilde{X}_{t'},Y)] = E[\tilde{X}_t] - E[\tilde{X}_{t'}] = (e^{-t/2} - e^{-t'/2}) E(X).
\]
Thus, by Lemma \ref{lem:RVboundedness},
\begin{multline} \label{eq:transportCauchyestimate1}
\norm{\tilde{F}_{t,t'}(x,y) - x}_2 \leq |e^{-t/2} - e^{-t'/2}| \norm{E(X)}_2 \\ + L |e^{-t/2} - e^{t'/2}| \left( \norm{(x - e^{-t'/2} E(X), y - E(Y))}_2 + (e^{-t'} \Var(X) + (1 - e^{-t'})m + \Var(Y))^{1/2} \right).
\end{multline}
Plugging this into \eqref{eq:convergenceproof1}, we see that $\tilde{F}_{s,t}$ is Cauchy in $t$ as $t \to +\infty$.  Moreover, we obtain the estimate \eqref{eq:transportconvergenceestimate} by taking $t' \to \infty$ in \eqref{eq:transportCauchyestimate1} and multiplying by $\norm*{\tilde{F}_{s,t}}_{\Lip,dx} \leq \max(c,1/c)^{1/2}$.

Now let us fix $t$ and consider when $s'$ and $s$ approach $\infty$.  The argument for this case is similar but antisymmetrical.  We estimate
\begin{align*}
\norm*{\tilde{F}_{s',t}(x,y) - \tilde{F}_{s,t}(x,y)}_2 &= \norm*{\tilde{F}_{s',s}(\tilde{F}_{s,t}(x,y),y) - \tilde{F}_{s,t}(x,y)}_2 \\
&\leq \norm*{\tilde{F}_{s',s} - \pi_1}_{\Lip,dx} \norm*{\tilde{F}_{s,t}(x,y)}_2 \\
&\leq \frac{1}{2} (\max(C,1/c) - 1) \max(C,1/c)^{1/2} |e^{-s} - e^{-s'}| \norm*{\tilde{F}_{s,t}(x,y)}_2,
\end{align*}
where the last line follows from \eqref{eq:transportLipdx3}.  Then by applying Lemma \ref{lem:RVboundedness} to the function $\tilde{F}_{s,t}(x,y)$ and the random variable $(\tilde{X}_t,Y)$, together with \eqref{eq:transportLip2}, we obtain
\begin{multline*}
\norm*{\tilde{F}_{s,t}(x,y)}_2 \leq e^{-s/2} \norm{E(X)}_2 \\
 + \max(C,1/c)^{7/2} \left( \norm{(x - e^{-t/2}E(X), y - E(Y))}_2 + (e^{-t} \Var(X) + (1 - e^{-t})m + \Var(Y))^{1/2} \right)
\end{multline*}
This produces an estimate on $\norm*{\tilde{F}_{s',t} - \tilde{F}_{s,t}}_2$ which shows that $\tilde{F}_{s,t}$ is Cauchy as $s \to \infty$, so that $\tilde{F}_{\infty,t}$ is well-defined.  The explicit bound on the rate of convergence follows fixing $s$ and $t$, combining the above estimates, and taking $s' \to \infty$.

Finally, since we have established convergence of $\tilde{F}_{s,t}$ as $s$ or $t$ approaches $\infty$, a routine argument with limits will extend the estimates of Lemmas \ref{lem:transportLipschitzestimate2} and \ref{lem:transportLipschitzestimate3}, and the transport relations, to the cases where $s$ or $t$ is $+\infty$.
\end{proof}

\subsection{Transport in the Large $N$ Limit} \label{subsec:largeNtransport}

If $V^{(N)} \in \mathcal{E}_{m+n}^{(N)}(c,C)$ and $\{DV^{(N)}\}$ is asymptotically approximable by trace polynomials, then we must show that the associated sequence of transport maps is asymptotically approximable by trace polynomials, and hence conclude that they define transport for the non-commutative random variables in the large $N$ limit.

\begin{theorem} \label{thm:transport1}
For $N \in \N$, let $V^{(N)}(x,y)$ be a potential on $M_N(\C)_{sa}^{m+n}$ satisfying Assumption \ref{ass:convexRMM} for some $0 < c \leq C$, and let $\mu^{(N)}$ be the corresponding probability measures on $M_N(\C)_{sa}^{m+n}$.  Let $(X^{(N)}, Y^{(N)})$ be a random variable given by $\mu^{(N)}$ and let $S^{(N)}$ be an independent GUE $m$-tuple.  Let
\[
\tilde{\mu}_t^{(N)} \text{ be the law of } (\tilde{X}_t, Y) = (e^{-t/2} X^{(N)} + (1 - e^{-t})^{1/2} S^{(N)}, Y^{(N)}),
\]
and let $\tilde{V}_t^{(N)}(x,y) = R_{e^t-1}^{(N)} V^{(N)}(e^{t/2} x, y)$ be the corresponding potential.  Similarly, let $\mu_\infty^{(N)}$ be the law of $(S^{(N)}, Y^{(N)})$.  For $s, t \in [0,\infty)$, let $\tilde{F}_{s,t}^{(N)}: M_N(\C)_{sa}^{m+n} \to M_N(\C)_{sa}^m$ be the solution of the initial value problem
\begin{align*}
\tilde{F}_{t,t}^{(N)}(x,y) &= x \\
\partial_s \tilde{F}_{s,t}^{(N)}(x,y) &= \frac{1}{2} \left( D_x \tilde{V}_s(\tilde{F}_{s,t}(x,y), y) - \tilde{F}_{s,t}(x,y) \right).
\end{align*}
Then
\begin{enumerate}
	\item The family $\tilde{F}_{s,t}^{(N)}$ extends continuously to $(s,t) \in [0,\infty]^2$.
	\item $\tilde{F}_{s,t}^{(N)} \circ \tilde{F}_{t,u}^{(N)} = \tilde{F}_{s,u}^{(N)}$.
	\item $(\tilde{F}_{s,t}^{(N)}(\tilde{X}_t^{(N)},Y^{(N)}), Y^{(N)}) \sim (\tilde{X}_s^{(N)}, Y^{(N)})$.
	\item For $s, t \in [0,\infty]$, the sequence $\{\tilde{F}_{s,t}^{(N)}\}_{N \in \N}$ is $(C/c) \max(C,1/c)^{1/2}$-Lipschitz for all $s$, $t$, and $N$, and it is asymptotically approximable by trace polynomials as $N \to \infty$.
\end{enumerate}
\end{theorem}

\begin{proof}
Recall in \S \ref{subsec:largeTtransport} we defined $\tilde{F}_{s,t}^{(N)}$ by renormalizing $F_{s,t}^{(N)}$.  However, that definition is equivalent to the definition of $\tilde{F}_{s,t}$ given in this theorem because both definitions produce a solution to the ODE \eqref{eq:renormalizedtransportODE}.  Of course, global uniqueness of the solution holds because the vector field $D_x \tilde{V}_t^{(N)}(x,y) - x$ is uniformly Lipschitz in $(x,y)$ on any compact time interval (as we discuss in more detail below).

So claims (1), (2), and (3) follow immediately from Proposition \ref{prop:largeTtransport}.  The estimate for the Lipschitz norm of $\tilde{F}_{s,t}^{(N)}$ was shown in \eqref{eq:transportLip2}.

We finish by showing asymptotic approximability using the results of \S \ref{subsec:vectorfields}.  Let $V_t^{(N)} = R_t^{(N)} V^{(N)}$.  By Theorem \ref{thm:Rtsemigroup} (3c), $D_x V_t^{(N)}$ is uniformly continuous in $t$ on $[0,\infty)$.  Since $D_x \tilde{V}_t^{(N)}(x,y) = e^{t/2} D_x V_{e^t - 1}^{(N)}(e^{t/2}x,y)$, it follows that $D_x \tilde{V}_t^{(N)}$ is uniformly continuous in $t$ on $[0,T]$ for every $T > 0$, with modulus of continuity independent of $N$, and recall it is also uniformly Lipschitz in $(x,y)$, since $0 \leq H\tilde{V}_t^{(N)} \leq \max(C, Ce^t / (1 + C(e^t - 1))$.

Consequently, $(1/2)(D_x \tilde{V}_t^{(N)}(x,y) - x)$ is uniformly continuous in $t$ on $[0,T]$ and uniformly Lipschitz in $(x,y)$.  Also, we showed that $D_x V_t^{(N)}$ is asymptotically approximable by trace polynomials in the proof of Theorem \ref{thm:convergenceofentropy}, and hence so is $D_x \tilde{V}_t^{(N)}$.  Thus, $(1/2)(D_x \tilde{V}_t^{(N)}(x,y) - x)$ satisfies Assumption \ref{ass:vectorfield2}, so we may apply Proposition \ref{prop:ODE2} to deduce that $\tilde{F}_{s,t}^{(N)}$ is asymptotically approximable by trace polynomials for $s, t \in [0,\infty)$.  This property extends to the case where $s$ or $t$ is infinite using Lemma \ref{lem:limits} and Proposition \ref{prop:largeTtransport}.
\end{proof}

\begin{remark}
Rather than citing the proof of Theorem \ref{thm:convergenceofentropy}, one could also argue that $D_x V_t^{(N)}$ is asymtotically approximable directly from the construction of the semigroup $R_t^{(N)}$ using the same reasoning as \cite[Proposition 6.8]{Jekel2018}.  Moreover, this method would also show that $D(R_t^{(N)} V^{(N)})$ is asymptotically approximable by trace polynomials provided we can prove analogues of Theorem \ref{thm:Rtsemigroup} (2) and (3) for $D(R_t^{(N)} V^{(N)})$ rather than only $D_x(R_t^{(N)} V^{(N)})$.  However, all this is unnecessary work for our present purpose.
\end{remark}

\begin{theorem} \label{thm:transport2}
With all the notation of the previous theorem, let $(X,Y)$ be a non-commutative random variable distributed according to the limiting free Gibbs law $\lambda$, let $S$ be a freely independent free semicircular $m$-tuple, and let $\tilde{X}_t = e^{-t/2} X + (1 - e^{-t})^{1/2} S$.  Define $\tilde{F}_{s,t}$ by $\tilde{F}_{s,t}^{(N)} \rightsquigarrow \tilde{F}_{s,t}$.  For $s,t,u \in [0,+\infty]$, we have
\begin{enumerate}
	\item $\tilde{F}_{s,t}$ is $(C/c) \max(C,1/c)^{1/2}$-Lipschitz with respect to $\norm{\cdot}_2$.
	\item $\tilde{F}_{s,t} \circ \tilde{F}_{t,u} = \tilde{F}_{s,u}$.
	\item $(\tilde{F}_{s,t}(\tilde{X}_t,Y), Y) \sim (\tilde{X}_s,Y)$ in non-commutative law.
	\item We have
	\begin{multline*}
	\norm*{\tilde{F}_{s,t}(\tilde{X}_t,Y) - \tilde{F}_{s',t}(\tilde{X}_t,Y) - (e^{-s/2} - e^{-s'/2})\tau(X)}_\infty \\
	 \leq (\max(C,1/c)^3 - 1) \max(C,1/c) |e^{-s/2} - e^{-s'/2}| \Theta.
	\end{multline*}
	where $\Theta$ is the universal constant from Proposition \ref{prop:operatornormestimate}.
\end{enumerate}
In particular, $\mathrm{W}^*(X,Y)$ is isomorphic to $\mathrm{W}^*(S,Y)$, which is the free product $\mathrm{W}^*(S) * \mathrm{W}^*(Y)$.
\end{theorem}

\begin{proof}
We know that there exists $\tilde{F}_{s,t}$ such that $\tilde{F}_{s,t}^{(N)} \rightsquigarrow \tilde{F}_{s,t}$ because of Lemma \ref{lem:AATP}.  Then (1) and (2) follow from the corresponding properties of $\tilde{F}_{s,t}^{(N)}$ by straightforward limit arguments.

As remarked in the last proof $D\tilde{V}_t^{(N)}$ is asymptotically approximable by trace polynomials.  We also know $D\tilde{V}_t^{(N)}$ is uniformly convex and semi-concave, and thus by Theorem \ref{thm:freeGibbslaw}, the non-commutative law of $(\tilde{X}_t^{(N)}, Y^{(N)})$ converges in probability to some non-commutative law.  Of course, the limiting non-commutative law must be the non-commutative law of $(\tilde{X}_t,Y)$ because the joint non-commutative law of $(X^{(N)},Y^{(N)},S^{(N)})$ converges in probability to that $(X,Y,S)$ (as in the proof of Theorem \ref{thm:convergenceofentropy}).

With this relation between the laws of $(\tilde{X}_t^{(N)},Y^{(N)})$ and $(\tilde{X}_t,Y)$ in hand, we can prove (3) by taking the large $N$ limit using Corollary \ref{cor:convergenceofexpectation}.  Indeed, if $f \in \overline{\TrP}_{m+n}^1$ is $\norm{\cdot}_2$-uniformly continuous, then $f(\tilde{F}_{s,t}^{(N)}(x,y),y)$ is also $\norm{\cdot}_2$-uniformly continuous and asymptotically approximable by trace polynomials by Lemma \ref{lem:composition}.  Thus, applying Corollary \ref{cor:convergenceofexpectation} to this function and the function $1$, we get
\begin{align*}
\tau\left( f(\tilde{F}_{s,t}(\tilde{X}_t,Y), Y) \right) &= \lim_{N \to \infty} E[\tau_N (f(\tilde{F}_{s,t}(\tilde{X}_t^{(N)},Y^{(N)}), Y^{(N)}))] \\
&= \lim_{N \to \infty} E[\tau_N(f(\tilde{X}_s^{(N)},Y^{(N)}))] \\
&= \tau(f(\tilde{X}_s,Y)).
\end{align*}
Hence, $\tau\left( f(\tilde{F}_{s,t}(\tilde{X}_t,Y), Y) \right) = \tau(f(\tilde{X}_s,Y))$ for all $f \in \overline{\TrP}_m^1$ that are uniformly continuous in $\norm{\cdot}_2$.  But by Proposition \ref{prop:realizationofoperators} such functions $f$ can realize every element in the $\mathrm{W}^*$-algebra generated by $(\tilde{X}_s,Y)$, and in particular all the non-commutative polynomials in $(\tilde{X}_s,Y)$.  Hence, $(\tilde{F}_{s,t}(\tilde{X}_t,Y), Y) \sim (\tilde{X}_s,Y)$ in non-commutative law as desired.

(4) Note that
\[
\tilde{F}_{s,t}(\tilde{X}_t,Y) - \tilde{F}_{s',t}(\tilde{X}_t,Y) = (\pi_1 - \tilde{F}_{s,s'}) \circ (\tilde{F}_{s',t}(\tilde{X}_t,Y),Y),
\]
but $(\tilde{F}_{s',t}(\tilde{X}_t,Y),Y) \sim (\tilde{X}_{s'}, Y)$ in non-commutative law.  Hence, it suffices to prove the desired estimate for $\tilde{F}_{s,s'}(\tilde{X}_{s'},Y) - \tilde{X}_{s'}$ rather than $\tilde{F}_{s,t}(\tilde{X}_t,Y) - \tilde{F}_{s',t}(\tilde{X}_t,Y)$.  Now $(\tilde{X}_{s'}, Y)$ arises as the large $N$ limit of the matrix models given by potential $\tilde{V}_{s'}^{(N)}$.  By Lemma \ref{lem:Qt} (4), we have $HV_t^{(N)} \geq c(1+ct)^{-1} I_m \oplus c I_n$, so that
\[
H \tilde{V}_{s'}^{(N)} \geq \frac{ce^{s'}}{1 + c(e^{s'} - 1)} I_m \oplus c I_n \geq \min(1,c) I_{m+n} \geq \frac{1}{\max(C,1/c)} I_{m+n}.
\]
By Remark \ref{rem:unitaryinvariance}, there exists a sequence of random matrix models for $(\tilde{X}_{s'},Y)$ given by uniformly convex potentials which are also unitarily invariant (even if this is not true of our original model), with the same lower bound $1 / \max(C,1/c)$ for the Hessian of the potential.  Therefore, by Proposition \ref{prop:operatornormestimate},
\[
\norm*{\tilde{F}_{s,s'}(\tilde{X}_{s'},Y) - \tilde{X}_{s'} - \tau[\tilde{F}_{s,s'}(\tilde{X}_{s'},Y) - \tilde{X}_{s'}]}_\infty \leq \max(C,1/c)^{1/2} \Theta \norm*{\tilde{F}_{s,s'} - \pi_1}_{\Lip}.
\]
We finish by substituting the estimate
\[
\norm{\tilde{F}_{s,s'} - \pi_1}_{\Lip} \leq (\max(C,1/c)^3 - 1) \max(C,1/c)^{1/2} |e^{-s/2} - e^{-s'/2}|
\]
which follows from \eqref{eq:transportLip3} and Lemma \ref{lem:limituniformlycontinuous} (the latter lemma is needed since the original statement of \eqref{eq:transportLip3} is for the finite-dimensional setting for a fixed $N$).

The last claim regarding $\mathrm{W}^*$-algebras follows from (3) by examining the case with $s = 0$ and $t = \infty$ or vice versa.
\end{proof}

\section{Applications} \label{sec:applications}

We show that Assumption \ref{ass:convexRMM} is preserved under independent joins, marginals, convolution, and linear changes of variables.  We conclude that for the convex free Gibbs laws considered here, $\chi^*$ satisfies additivity under conditioning.  Moreover, by iterating our conditional transport results, we obtain ``lower-triangular'' transport maps from a convex free Gibbs law to the law of a free semicircular family, which also satisfy the entropy-cost inequality relative to the semicircular law, analogous to the triangular transport achieved in the classical case by \cite[Corollary 3.10]{BKM2005}.

\subsection{Operations on Convex Gibbs Laws} \label{subsec:operations}

Recall that Assumption \ref{ass:convexRMM} for a sequence $\{V^{(N)}\}$ of potentials $M_N(\C)_{sa}^m \to \R$ states that $c \leq HV^{(N)} \leq C$ for some constants $c$ and $C$, the sequence $\{DV^{(N)}\}$ is asymptotically approximable by trace polynomials, and $\int x_j \,d\mu^{(N)}(x)$ is a scalar matrix for each $j$, where $\mu^{(N)}$ is the measure associated to $DV^{(N)}$.

\begin{proposition} \label{prop:independentjoins}
Suppose that $V_1^{(N)}: M_N(\C)_{sa}^m \to \R$ and $V_2^{(N)}: M_N(\C)_{sa}^n \to \R$ satisfy Assumption \ref{ass:convexRMM} for some $0 < c \leq C$.  Then $V^{(N)}(x,y) := V_1^{(N)}(x) + V_2^{(N)}(y)$ also satisfies Assumption \ref{ass:convexRMM} for the same $c$ and $C$.

Moreover, let $\mu_1^{(N)}$, $\mu_2^{(N)}$, and $\mu^{(N)}$ be the measures associated to $V_1^{(N)}$, $V_2^{(N)}$, and $V^{(N)}$ respectively, and let $\lambda_1$, $\lambda_2$, and $\lambda$ be the respective limiting free Gibbs laws given by Theorem \ref{thm:freeGibbslaw}.  Then $\mu^{(N)}$ is the independent join of $\mu_1^{(N)}$ and $\mu_2^{(N)}$ and $\lambda$ is the freely independent join of $\lambda_1$ and $\lambda_2$.
\end{proposition}

\begin{proof}
The claim $c \leq HV^{(N)} \leq C$ follows because $HV^{(N)}(x,y) = HV_1^{(N)}(x) \oplus HV_2^{(N)}(y)$.  The claim about asymptotic approximation by trace polynomials follows because $DV^{(N)}(x,y) = (DV_1^{(N)}(x), DV_2^{(N)}(y))$ and each component is asymptotically approximable by trace polynomials.

The probability density for $\mu^{(N)}$ is the tensor product of the probability densities for $\mu_1^{(N)}$ and $\mu_2^{(N)}$ and hence $\mu^{(N)}$ is the independent join of these two marginal laws.  It follows that $\int x_j \,d\mu^{(N)}(x)$ and $\int y_j \,d\mu^{(N)}(y)$ are scalar matrices, hence Assumption \ref{ass:convexRMM} holds for $V^{(N)}$.

Let $(X^{(N)},Y^{(N)}) \sim \mu^{(N)}$ be random variables and let $(X,Y) \sim \lambda$ be non-commutative random variables.  Then by Theorem \ref{thm:convergenceofentropy},
\begin{align*}
\Phi^*(X,Y) &= \lim_{N \to \infty} E\left[ \norm{DV^{(N)}(X^{(N)},Y^{(N)})}_2^2 \right] \\
&= \lim_{N \to \infty} E\left[ \norm{DV_1^{(N)}(X^{(N)})}_2^2 \right] + \lim_{N \to \infty} E\left[ \norm{DV_2^{(N)}(Y^{(N)})}_2^2 \right] \\
&= \Phi^*(X) + \Phi^*(Y).
\end{align*}
It was shown in \cite[Proposition 5.18(c)]{VoiculescuFE6} that $\Phi^*(X,Y) = \Phi^*(X) + \Phi^*(Y)$ implies that $X$ and $Y$ are freely independent.
\end{proof}

\begin{proposition} \label{prop:marginals}
Suppose that $V^{(N)}: M_N(\C)_{sa}^{m+n} \to \R$ satisfies Assumption \ref{ass:convexRMM}.  Let $\mu^{(N)}$ be the corresponding law, let $(X^{(N)},Y^{(N)}) \sim \mu^{(N)}$ and let $\mu_1^{(N)}$ and $\mu_2^{(N)}$ be the laws of $X^{(N)}$ and $Y^{(N)}$.  Then $\mu_1^{(N)}$ and $\mu_2^{(N)}$ are given by a potentials $W_1^{(N)}$ and $W_2^{(N)}$ that also satisfy Assumption \ref{ass:convexRMM} for the same values of $c$ and $C$.
\end{proposition}

\begin{proof}
By symmetry, it suffices to prove the claims for $\mu_2^{(N)}$.  First, it is immediate that the mean of $y_j$ under $\mu_2^{(N)}$ is a scalar, since it is $E[Y_j^{(N)}]$.  Moreover, if we define
\[
V_2^{(N)}(x) = -\frac{1}{N^2} \int e^{-N^2 V^{(N)}(x,y)}\,dx,
\]
then (as in the proof of Theorem \ref{thm:convergenceofentropy}) we may compute $DV_2^{(N)}$ by differentiating under the integral and obtain
\[
DV_2^{(N)}(Y^{(N)}) = E[D_y V(X^{(N)}, Y^{(N)}) | Y^{(N)}].
\]
It follows by Theorem \ref{thm:conditionalexpectation} that $\{DV_2^{(N)}\}$ is asymptotically approximable by trace polynomials.

Finally, the fact that $c \leq HV_2^{(N)} \leq C$ follows from \cite[Theorem 4.3]{BL1976}, or alternatively by the following reasoning.  Let $\mu_t^{(N)}$ be the law of $(e^{-t/2} X^{(N)} + (1 - e^{-t})^{1/2} S^{(N)}, Y^{(N)})$, where $S^{(N)}$ is an independent GUE tuple.  The corresponding potential $\tilde{V}_t^{(N)}$ is given by \eqref{eq:renormalizedpotential} and it satisfies
\[
\frac{ce^t}{1 + c(e^t - 1)} I_m \oplus c I_n \leq H\tilde{V}_t^{(N)} \leq \frac{Ce^t}{1 + C(e^t - 1)} I_m \oplus C I_n
\]
by direct substitution of \eqref{eq:renormalizedpotential} into Lemma \ref{lem:Qt} (4) and hence
\[
\min(c,1) I_m \oplus c I_n \leq H\tilde{V}_t^{(N)} \leq \max(C,1) I_m \oplus C I_n
\]
Now as $t \to \infty$, the law $\tilde{\mu}_t$ converges to the law $\tilde{\mu}_\infty$ of $(S^{(N)},Y^{(N)})$.  By applying Lemma \ref{lem:limitoflogconcave}, $\tilde{\mu}_\infty$ is given by some potential $W_2^{(N)}(x,y)$ satisfying
\[
\min(c,1) I_m \oplus c I_n \leq HW_2^{(N)} \leq \max(C,1) I_m \oplus C I_n.
\]
However, we know that $W_2^{(N)}(x,y) = (1/2) \norm{x}_2^2 + V_2^{(N)}(y) + \text{constant}$ because the potential corresponding to a law is unique up to an additive constant.  This implies that $c \leq HV_2^{(N)} \leq C$ as desired.
\end{proof}

\begin{proposition} \label{prop:lineartransformations}
Let $V^{(N)}: M_N(\C)_{sa}^m \to \R$ satisfy Assumption \ref{ass:convexRMM} for some $0 < c \leq C$, and let $X^{(N)}$ be the corresponding random variable.  Let $A$ be an invertible $m \times m$ matrix with real entries and let $A^{(N)}$ denote the linear transformation $M_N(\C)_{sa}^m \to M_N(\C)_{sa}^m$ given by
\[
(A^{(N)}x)_i = \sum_{j=1}^m A_{i,j}x.
\]
Then $\widehat{V}^{(N)} = V^{(N)}((A^{-1})^{(N)})$ is the potential corresponding to $A^{(N)} X^{(N)}$, and $\widehat{V}^{(N)}$ satisfies Assumption \ref{ass:convexRMM} with constants $c / \norm{A}$ and $C \norm{A^{-1}}$.
\end{proposition}

\begin{proof}
The fact that $\widehat{V}^{(N)}$ is the potential corresponding to $A^{(N)} X^{(N)}$ follows from change of variables.  Now it is immediate that the expectation of $(A^{(N)} X^{(N)})_i$ is a scalar multiple of identity for each $i$.  Next, by the chain rule
\[
D\widehat{V}^{(N)}(x) = ((A^{-1})^{(N)})^T DV^{(N)}((A^{-1})^{(N)}x) = ((A^{-1})^T)^{(N)} DV^{(N)}((A^{-1})^{(N)} x),
\]
and from this it follows that $\{D\widehat{V}^{(N)}\}$ is asymptotically approximable by trace polynomials.  Similarly, by the chain rule,
\[
H\widehat{V}^{(N)}(x) = ((A^{-1})^T)^{(N)} HV^{(N)}((A^{-1})^{(N)}x) (A^{-1})^{(N)}.
\]
The maximum and minimum singular values of $(A^{-1})^{(N)}$ are the same as those of $A^{-1}$, which are $\norm{A^{-1}}$ and $1 / \norm{A}$ respectively.  By a basic linear algebra argument, it follows that $c / \norm{A} \leq H\widehat{V}^{(N)} \leq C \norm{A^{-1}}$.
\end{proof}

\begin{proposition} \label{prop:convolutions}
Let $V_1^{(N)}$ and $V_2^{(N)}$ be two potentials $M_N(\C)_{sa}^m \to \R$ satisfying Assumption \ref{ass:convexRMM} with constants $c$ and $C$.  Let $X^{(N)}$ and $Y^{(N)}$ be the corresponding random tuples of matrices.  Then the law of $X^{(N)} + Y^{(N)}$ is given by another potential $\widehat{V}^{(N)}$ satisfying Assumption \ref{ass:convexRMM} with constants $\sqrt{2} c$ and $\sqrt{2} C$.  Moreover, the free Gibbs state corresponding to $\{\widehat{V}^{(N)}\}$ is the free convolution of those corresponding to $\{V_1^{(N)}\}$ and $\{V_2^{(N)}\}$.
\end{proposition}

\begin{proof}
Let $V^{(N)}(x,y) = V_1^{(N)}(x) + V_2^{(N)}(y)$, which satisfies Assumption \ref{ass:convexRMM} (with the same constants) by Proposition \ref{prop:independentjoins}.  Now let $A$ be the $2m \times 2m$ matrix
\[
A = \begin{pmatrix} I & I \\ -I & I \end{pmatrix}.
\]
Since $A / \sqrt{2}$ is an isometry, we have $\norm{A} = \sqrt{2}$ and $\norm{A^{-1}} = 1/\sqrt{2}$.  Therefore, by Proposition \ref{prop:lineartransformations}, the law of $(X^{(N)} + Y^{(N)}, -X^{(N)} + Y^{(N)})$ is given by a potential satisfying Assumption \ref{ass:convexRMM} with constants $\sqrt{2} c$ and $\sqrt{2} C$.  Then by Proposition \ref{prop:marginals}, the law of $X^{(N)} + Y^{(N)}$ is given by such a potential with the same constants $\sqrt{2} c$ and $\sqrt{2} C$.

We showed in Proposition \ref{prop:independentjoins} that the large $N$ limit of the law of $(X^{(N)},Y^{(N)})$ given a freely independent join of the corresponding marginals.  Hence, the large $N$ limit of the law of $X^{(N)} + Y^{(N)}$ is given by the free convolution.
\end{proof}

As a consequence, we have additivity of entropy under conditioning.

\begin{corollary}
Let $V^{(N)}(x,y)$ be a potential satisfying Assumption \ref{ass:convexRMM} as in the setup of Theorem \ref{thm:convergenceofentropy}.  Let $(X,Y)$ be a tuple of non-commutative random variables distributed according to the limiting free Gibbs law associated to $V^{(N)}$.  Then
\[
\chi^*(X,Y) = \chi^*(X | Y) + \chi^*(Y).
\]
\end{corollary}

\begin{proof}
From standard classical results, we have
\[
h(X^{(N)}, Y^{(N)}) = h(X^{(N)} | Y^{(N)}) + h(Y^{(N)}).
\]
Dividing by $N^2$ and adding $\frac{1}{2} (m+n) \log N$ to both sides, we obtain the normalized version
\[
h^{(N)}(X^{(N)}, Y^{(N)}) = h^{(N)}(X^{(N)} | Y^{(N)}) + h^{(N)}(Y^{(N)}).
\]
By the previous theorem, we obtain the desired relation for $\chi^*$ in the limit as $N \to \infty$.  More precisely, we apply the theorem as stated to $h^{(N)}(X^{(N)} | Y^{(N)})$.  Meanwhile, for $h^{(N)}(X^{(N)},Y^{(N)})$ and $h^{(N)}(Y^{(N)})$ we apply the special case of the theorem where we condition on $0$ variables.
\end{proof}

\subsection{Entropy and Fisher Information Relative to Gaussian}

As background for our discussion of the entropy-cost inequality in \S \ref{subsec:entropycost}, we review the entropy of one probability measure relative to another.  If $\nu$ is a measure on $\R^m$, then the \emph{entropy of $\mu$ relative to $\nu$} is
\[
h(\mu | \nu) := -\int \rho \log \rho\,d\mu, \text{ where } \rho = \frac{d\mu}{d\nu}
\]
whenever the integral is well-defined.  The standard entropy $h(\mu) = -\int \rho \log \rho$ corresponds to the choice of Lebesgue measure for $\nu$.

\begin{remark}
The reader should be careful to distinguish between the relative entropy $h(\mu | \nu)$ and the conditional entropy $h(X | Y)$.  The first changes the ambient measure while the second describes conditioning on $Y$.
\end{remark}

\begin{remark}
If $\mu$ and $\nu$ are both probability measures, then $h(\mu | \nu) \leq 0$.  For this reason, many authors choose to change the sign.  We will keep the sign convention given above to be consistent with our convention for $h(\mu)$ relative to Lebesgue measure, but we will write absolute value signs around relative entropy when it is natural to use the positive version.
\end{remark}

For probability measures on $M_N(\C)_{sa}^m$, we may study entropy relative to the Gaussian measure $\sigma_{m,t}^{(N)}$ on $M_N(\C)_{sa}^m$.  A direct computation shows that if $X \sim \mu$ is a random variable in $M_N(\C)_{sa}^m$, then we have
\[
h(\mu | \sigma_{m,t}^{(N)}) = h(\mu) - \frac{N^2}{2} E \norm{X}_2^2 + \frac{mN^2}{2} \log \frac{N}{2\pi} .
\]
We denote the normalized version by
\[
h_g^{(N)}(\mu) := h_g^{(N)}(X) := \frac{1}{N^2} h(\mu | \sigma_{m,t}^{(N)}) = h^{(N)}(X) - \frac{1}{2} E \norm{X}_2^2 - \frac{m}{2} \log 2\pi.
\]
Similarly, if $\mu$ is a measure on $M_N(\C)_{sa}^{m+n}$ which absolutely continuous with respect to Lebesgue measure and $(X,Y)$ is the corresponding random variable, we define
\[
h_g^{(N)}(X | Y) = h^{(N)}(X | Y) - \frac{1}{2} E \norm{X}_2^2 - \frac{m}{2} \log 2\pi,
\]
which is equivalent to
\[
h_g^{(N)}(X | Y) = \int h_g^{(N)}(\mu_{X|Y=y})\,d\mu_Y(y),
\]
where $\mu_{X|Y=y}$ is the conditional distribution of $X$ given $Y = y$, and $\mu_Y$ is the marginal law of $Y$.  Similarly, if $(X,Y)$ is an $(m+n)$-tuple of non-commutative random variables, we define the free entropy $\chi^*$ relative to Gaussian by
\[
\chi_g^*(X | Y) = \chi^*(X | Y) - \frac{1}{2} \norm{X}_2^2 - \frac{m}{2} \log 2\pi.
\]

We define the \emph{normalized conditional Fisher information relative to Gaussian} by
\begin{equation} \label{eq:definenormalizedFisher}
\mathcal{I}_g^{(N)}(X | Y) = \mathcal{I}^{(N)}(X | Y) - 2m + E \norm{X}_2^2.
\end{equation}
Note that if this Fisher information is finite and if $\xi$ is the normalized score function for $X$ given $Y$ as in \S \ref{subsec:matrixentropy}, then
\[
\mathcal{I}_g^{(N)}(X | Y) = E \norm{\xi - X}_2^2
\]
because
\begin{align*}
E \norm{\xi - X}_2^2 &= E \norm{\xi}_2^2 - 2 E \ip{\xi,X}_2 + E \norm{X}_2^2 \\
&= \mathcal{I}^{(N)}(X | Y) - 2m + E \norm{X}_2^2,
\end{align*}
where we have evaluated the middle term on the right hand side using integration by parts.  Similarly, for an $(m+n)$-tuple $(X,Y)$ of non-commutative random variables, we define
\begin{equation} \label{eq:definenormalizedFisher2}
\Phi_g^*(X | Y) = \Phi^*(X | Y) - 2m + \norm{X}_2^2 = \norm{\xi - X}_2^2,
\end{equation}
where the second equality holds provided that $\Phi^*$ is finite and $\xi$ is the free score function.  We have the following version of \eqref{eq:normalizedentropyformula} and Lemma \ref{lem:Fisherestimates} for entropy and Fisher information relative to Gaussian.

\begin{lemma} \label{lem:GaussianFisherinfo}
Let $(X,Y)$ be a random variable in $M_N(\C)_{sa}^{m+n}$ with a density and with finite variance and let $S$ be an independent GUE $m$-tuple.  Then
\[
|h_g^{(N)}(X | Y)| = \frac{1}{2} \int_0^\infty \mathcal{I}_g^{(N)}(e^{-t/2}X + (1 - e^{-t})^{1/2}S | Y)\,dt.
\]
Similarly, suppose that $(X,Y)$ is an $(m+n)$-tuple of non-commutative random variables and let $S$ be a freely independent free semicircular $m$-tuple.  Then
\[
|\chi_g^*(X | Y)| = \frac{1}{2} \int_0^\infty \Phi_g^*(e^{-t/2} X +  (1 - e^{-t})^{1/2} S | Y)\,dt.
\]
\end{lemma}

\begin{proof}
The first formula follows from \cite[\S 4, Lemma 1]{OV2000} after renormalization.  However, we will give an argument by a change of variables in \eqref{eq:normalizedentropyformula} that will apply to both $h_g^{(N)}$ and $\chi_g^*$.  Note that by \eqref{eq:normalizedentropyformula}
\begin{align*}
h_g^{(N)}(X|Y) &= h^{(N)}(X|Y) - \frac{1}{2} E \norm{X}_2^2 - \frac{m}{2} \log 2\pi \\
&=  \frac{1}{2} \int_0^\infty \left( \frac{m}{1 + t} - \mathcal{I}^{(N)}(X + t^{1/2}S | Y) \right)\,dt + \frac{m}{2} - \frac{1}{2} E \norm{X}_2^2,
\end{align*}
and in particular, we know that the integral is well-defined in $[-\infty,+\infty)$.  Now we do a change of variables in the integral $t = e^u - 1$, $dt = e^u \,du$ and obtain
\begin{align*}
\frac{1}{2} \int_0^\infty \left( \frac{m}{1 + t} - \mathcal{I}^{(N)}(X + t^{1/2}S | Y) \right)\,dt &= \frac{1}{2} \int_0^\infty \left( \frac{m}{e^u} - \mathcal{I}^{(N)}(X + (e^u - 1)^{1/2}S| Y) \right) e^u\,du \\
&= \frac{1}{2} \int_0^\infty \left( m - \mathcal{I}^{(N)}(e^{-u/2} X + (1 - e^{-u})^{1/2}S| Y) \right) \,du,
\end{align*}
where we have applied the scaling relation Lemma \ref{lem:xiscaling} for Fisher information.  On the other hand,
\begin{align*}
\frac{m}{2} - \frac{1}{2} E \norm{X}_2^2 &= \frac{1}{2} \left( E \norm{S}_2^2 - E \norm{X}_2^2 \right) \\
&= \frac{1}{2} \int_0^\infty e^{-u} \left( E \norm{S}_2^2 - E \norm{X}_2^2 \right)\,du \\
&= \frac{1}{2} \int_0^\infty \left( m - E \norm{e^{-u/2} X + (1 - e^{-u})^{1/2} S}_2^2 \right)\,du.
\end{align*}
Therefore, altogether
\begin{align*}
h_g^{(N)}(X|Y) &= \frac{1}{2} \int_0^\infty \left( 2m - \mathcal{I}^{(N)}(e^{-u/2} X + (1 - e^{-u})^{1/2}S| Y) - E \norm{e^{-u/2} X + (1 - e^{-u})^{1/2} S}_2^2 \right) \,du \\
&= -\frac{1}{2} \int_0^\infty \mathcal{I}_g^{(N)}(e^{-u/2}X + (1 - e^{-u})^{1/2}S | Y)\,du,
\end{align*}
which is the desired formula.  The statement for $\chi^*$ can be proved by exactly the same computation, since the definition of $\chi^*$ in \eqref{eq:definechi*} is completely analogous to \eqref{eq:normalizedentropyformula}.
\end{proof}

Furthermore, the log-Sobolev inequality for the Gaussian measure has the following interpretation for entropy and Fisher's information.  This in fact generalizes to entropy and Fisher's information relative to any measure $\nu$ satisfying LSI, see \cite[Definition 1]{OV2000}, but we only use the case where $\nu$ is Gaussian and $\mu$ is sufficiently regular.

\begin{lemma} \label{lem:informationLSI}
Let $X$ be a random variable in $M_N(\C)_{sa}^m$ that has a $C^1$ density with respect to Lebesgue measure.  Then
\[
|h_g^{(N)}(X | Y)| \leq \frac{1}{2} \mathcal{I}_g^{(N)}(X | Y).
\]
\end{lemma}

\begin{proof}
First, it suffices to check the non-conditional version $h_g^{(N)}(X) \leq \frac{1}{2} \mathcal{I}_g^{(N)}(X)$.  Indeed, in the conditional case, the left hand side is $\int h_g^{(N)}(\mu_{X|Y=y})\,d\mu_Y(y)$ and the right hand side is $\int \mathcal{I}_g^{(N)}(\mu_{X|Y=y})\,d\mu_Y(y)$, and solving the non-conditional case would allow us to compare the integrands pointwise.

Now suppose that $X$ has density $\rho$ with respect to Lebesgue measure and let $\tilde{\rho}$ be the density with respect to Gaussian, so that
\[
\rho(x) = \tilde{\rho}(x) \frac{1}{(2\pi / N)^{N^2/2}} e^{-N^2 \norm{x}_2^2 / 2}\,dx.
\]
By Corollary \ref{cor:matrixLSI}, the measure $\sigma_{m,t}^{(N)}$ satisfies the normalized log-Sobolev inequality \eqref{eq:normalizedLSI} with $c = 1$, so that
\[
\frac{1}{N^2} \int f^2 \log \frac{f^2}{\int f^2 \,d\sigma_{m,t}^{(N)}} \,d\sigma_{m,t}^{(N)} \leq \frac{2}{N^4} \int \norm{Df}_2^2\,d\sigma_{m,t}^{(N)}.
\]
Let $f = \tilde{\rho}^{1/2}$.  Then $\int f^2\,d\sigma_{m,t}^{(N)}$ reduces to $1$, so the right hand side is $|h_g^{(N)}(X)|$.  On the other hand, letting $V(x) = -(1/N^2) \log \rho$, we get
\[
f(x) = \tilde{\rho}(x)^{1/2} = (\text{constant}) e^{-N^2V(x)/2 - N^2 \norm{x}_2^2 / 4},
\]
and hence on the support of $f$, we have
\[
Df(x) = -\frac{N^2}{2} (DV(x) - x) \tilde{\rho}(x)^{1/2}.
\]
Thus,
\[
\frac{2}{N^4} \int \norm{Df}_2^2\,d\sigma_{m,t}^{(N)} = \frac{1}{2} \int \norm{DV(x) - x}_2^2\,d\mu(x) = \frac{1}{2} \mathcal{I}_g^{(N)}(X).
\]
Hence, the log-Sobolev inequality implies the desired inequality.
\end{proof}

\subsection{Conditional Transport and the Entropy-Cost Inequality} \label{subsec:entropycost}

Now we will show that the transport maps constructed in \S \ref{subsec:largeNtransport} satisfy the Talagrand entropy-cost inequality.  It was shown in \cite[Theorem 1]{OV2000} that if a measure $\nu$ satisfies the log-Sobolev inequality \eqref{eq:LSI} with some constant $c$ (and some regularity conditions), then it satisfies the Talagrand inequality
\[
W_2(\mu,\nu)^2 \leq \frac{2 h(\mu|\nu)}{\rho} \text{ for all } \mu,
\]
where $W_2$ is the $L^2$-Wasserstein distance, which is equivalent to the infimum of $\norm{X - Y}_{L^2}$ over all coupled random variables $X$ and $Y$ with $X \sim \mu$ and $Y \sim \nu$.

Adapting Otto and Villani's argument, we will show that the transport maps constructed in \S \ref{subsec:largeNtransport} witness the (conditional) entropy-cost inequality relative to the GUE law for the $N \times N$ matrix models and the corresponding free entropy-cost inequality for the non-commutative random variables.  This is claim (5) below, while the other claims in Theorem \ref{thm:transport3} summarize the results of our earlier construction.

We remark that the free Talagrand inequality for self-adjoint tuples was studied in greater generality in \cite{HU2006} and \cite[\S 3.3]{Dabrowski2010}.  Although we restricted ourselves to the case where the target measure is Gaussian/semicircular, our goal in this paper was not merely to estimate the Wasserstein distance using some coupling, but rather to exhibit a coupling that arises from a transport map, and to show Lipschitzness of this transport map.

\begin{theorem} \label{thm:transport3}
As in Theorem \ref{thm:transport1}, let $V^{(N)}(x,y)$ be a potential on $M_N(\C)_{sa}^{m+n}$ satisfying Assumption \ref{ass:convexRMM} for some $0 < c \leq C$, and let $\mu^{(N)}$ and $(X^{(N)}, Y^{(N)})$ be the corresponding probability measures and random variables.  Let $S^{(N)}$ be an independent GUE $m$-tuple.  Let $(X,Y)$ be a tuple of non-commutative random variables given by the limiting free Gibbs law $\lambda$ and let $S$ be a freely independent free semicircular $m$-tuple.  Let $\pi_1(x,y) = x$ and $\pi_2(x,y) = y$.  Then there exist functions $F^{(N)}$, $G^{(N)}: M_N(\C)_{sa}^{m+n} \to M_N(\C)_{sa}^m$ and $F, G \in (\overline{\TrP}_{m+n}^1)^m$ such that
\begin{enumerate}
	\item We have $(F^{(N)}(X^{(N)},Y^{(N)}),Y^{(N)}) \sim (S^{(N)},Y^{(N)})$ and $(G^{(N)}(S^{(N)},Y^{(N)}),Y^{(N)}) \sim (X^{(N)},Y^{(N)})$ in law, and $(F(X,Y),Y) \sim (S,Y)$ and $(G(S,Y),Y) \sim (X,Y)$ in non-commutative law.
	\item $(F^{(N)},\pi_2) \circ (G^{(N)},\pi_2) = \id = (G^{(N)},\pi_2) \circ (F^{(N)},\pi_2)$ and the same holds for $F$ and $G$.
	\item $F^{(N)} \rightsquigarrow F$ and $G^{(N)} \rightsquigarrow G$.
	\item We have $\norm{F^{(N)} - \pi_1}_{\Lip}$ and $\norm{G^{(N)} - \pi_1}_{\Lip} \leq (\max(C,1/c)^3 - 1) \max(C,1/c)^{1/2}$, and the same holds for $F$ and $G$.
	\item We have
	\[
	\norm{F^{(N)}(X^{(N)}, Y^{(N)}) - X^{(N)}}_{L^2}^2 = \norm{G^{(N)}(S^{(N)}, Y^{(N)}) - S^{(N)}}_{L^2}^2 \leq 2 h_g^{(N)}(X^{(N)} | Y^{(N)}).
	\]
	and
	\[
	\norm{F(X,Y) - X}_2^2 = \norm{G(S,Y) - S}_2^2 \leq 2 \chi_g^*(X | Y).
	\]
\end{enumerate}
\end{theorem}

\begin{proof}
Let $\tilde{F}_{s,t}^{(N)}$ and $\tilde{F}_{s,t}$ be as in Theorems \ref{thm:transport1} and \ref{thm:transport2}.  Then let
\begin{align*}
F^{(N)} &= \tilde{F}_{\infty,0}^{(N)} \\
G^{(N)} &= \tilde{F}_{0,\infty}^{(N)} \\
F &= \tilde{F}_{\infty,0} \\
G &= \tilde{F}_{0,\infty}.
\end{align*}
The only property that was not shown in the earlier theorems is (5).  First, note that as a consequence of (1),
\[
\norm{F^{(N)}(X^{(N)}, Y^{(N)}) - X^{(N)}}_{L^2} = \norm{S^{(N)} - G^{(N)}(S^{(N)}, Y^{(N)})}_{L^2}.
\]
The rest of the proof of (5) proceeds as in \cite[\S 4]{OV2000}.  As in \S \ref{subsec:largeNtransport}, let $\tilde{V}_t^{(N)}$ denote the potential corresponding to $(\tilde{X}_t^{(N)},Y^{(N)}) = (e^{-t/2}X^{(N)} + (1 - e^{-t})^{1/2} S^{(N)}, Y^{(N)})$ and recall that
\[
\partial_s \tilde{F}_{s,0}^{(N)}(x,y) = \frac{1}{2} \left( D_x \tilde{V}_s^{(N)}(\tilde{F}_{s,t}^{(N)}(x,y),y) - \tilde{F}_{s,t}^{(N)}(x,y) \right),
\]
and hence
\[
\norm*{\tilde{F}_{\infty,0}^{(N)}(x,y) - x}_2 \leq \frac{1}{2} \int_0^\infty \norm*{D_x \tilde{V}_s(\tilde{F}_{s,t}^{(N)}(x,y),y) - \tilde{F}_{s,t}^{(N)}(x,y)}_2\,ds.
\]
Then we apply Minkowski's inequality with respect to integration $d\mu^{(N)}(x,y)$ to obtain
\[
\left( \int \norm*{\tilde{F}_{\infty,0}^{(N)}(x,y) - x}_2^2 \,d\mu^{(N)}(x,y) \right)^{1/2} \leq \frac{1}{2} \int_0^\infty \left( \int \norm*{D_x \tilde{V}_s(\tilde{F}_{s,t}^{(N)}(x,y),y) - \tilde{F}_{s,t}^{(N)}(x,y)}_2^2 \,d\mu^{(N)}(x,y) \right)^{1/2} \,ds
\]
which can be rewritten as
\begin{align*}
\norm*{F^{(N)}(X^{(N)}, Y^{(N)}) - X^{(N)}}_{L^2} &\leq \frac{1}{2} \int_0^\infty \norm*{D_x \tilde{V}_s(\tilde{F}_{s,t}^{(N)}(X^{(N)},Y^{(N)}),Y^{(N)}) - \tilde{F}_{s,t}^{(N)}(X^{(N)},Y^{(N)})}_{L^2}\,ds \\
&= \frac{1}{2} \int_0^\infty \norm*{D_x \tilde{V}_s^{(N)}(\tilde{X}_s^{(N)},Y^{(N)}) - \tilde{X}_s^{(N)}}_{L^2}\,ds \\
&= \frac{1}{2} \int_0^\infty \mathcal{I}_g^{(N)}(\tilde{X}_s^{(N)} | Y^{(N)})^{1/2}\,ds,
\end{align*}
where we have applied the fact that $(\tilde{F}_{s,0}^{(N)}(X^{(N)},Y^{(N)}), Y^{(N)}) \sim (\tilde{X}_s^{(N)}, Y^{(N)})$.  It follows from Lemma \ref{lem:GaussianFisherinfo} and a change of variables that
\begin{equation} \label{eq:integrateentropy}
|h_g^{(N)}(\tilde{X}_t^{(N)} | Y^{(N)})| = \int_t^\infty \mathcal{I}_g^{(N)}(\tilde{X}_s^{(N)} | Y^{(N)})\,ds.
\end{equation}
It is easy to see that $s \mapsto \mathcal{I}_g^{(N)}(\tilde{X}_s^{(N)} | Y^{(N)})$ is bounded on compact sets because of Lemma \ref{lem:Fisherestimates} and \eqref{eq:definenormalizedFisher}. Therefore, we have for almost every $t$,
\[
\frac{d}{dt} |h_g^{(N)}(\tilde{X}_t^{(N)} | Y^{(N)})| = \frac{1}{2} \mathcal{I}_g^{(N)}(\tilde{X}_t^{(N)} | Y^{(N)}).
\]
Hence, for almost every $t$,
\begin{align*}
\frac{d}{dt} |h_g^{(N)}(\tilde{X}_t^{(N)} | Y^{(N)})|^{1/2} &= \frac{1}{4} \mathcal{I}_g^{(N)}(\tilde{X}_t^{(N)} | Y^{(N)}) |h_g^{(N)}(\tilde{X}_t^{(N)} | Y^{(N)})|^{-1/2} \\
&\geq \frac{1}{2 \sqrt{2}} \mathcal{I}_g^{(N)}(\tilde{X}_t^{(N)} | Y_t^{(N)})^{1/2}, 
\end{align*}
where the last line follows from Lemma \ref{lem:informationLSI}.  Therefore,
\begin{align*}
\norm*{F^{(N)}(X^{(N)}, Y^{(N)}) - X^{(N)}}_{L^2} &\leq \sqrt{2} \int_0^\infty \frac{d}{dt} |h_g^{(N)}(\tilde{X}_t^{(N)} | Y^{(N)})|^{1/2} \,dt \\
&= \sqrt{2} |h_g^{(N)}(X^{(N)} | Y^{(N)})|^{1/2},
\end{align*}
where we have employed the fact that $\lim_{t \to \infty} |h_g^{(N)}(\tilde{X}_t^{(N)} | Y^{(N)})| = 0$ by \eqref{eq:integrateentropy}.  This establishes the first claim of (5).

The second claim of (5) follows by taking the large $N$ limit using Corollary \ref{cor:convergenceofexpectation} and Theorem \ref{thm:convergenceofentropy}.  More precisely, for the left hand side, we take the limit using Corollary \ref{cor:convergenceofexpectation}.  Meanwhile, for the right hand side, note that $h_g^{(N)}(X^{(N)} | Y^{(N)}) \to \chi_g^*(X | Y)$ because $h^{(N)}(X^{(N)} | Y^{(N)}) \to \chi^*(X | Y)$ and $E \norm{X^{(N)}}_2^2 \to \norm{X}_2^2$ by Corollary \ref{cor:convergenceofexpectation}.
\end{proof}

\subsection{Construction of Triangular Transport}

Finally, by iterating Theorem \ref{thm:transport3}, we obtain the following result concerning ``lower-triangular transport.''  This is analogous to the classical result \cite[Corollary 3.10]{BKM2005}.  Of course, the challenge in our situation was to understand the large $N$ behavior of the transport maps in a dimension-independent way.  Unfortunately, the transport constructed here is not optimal among triangular mappings, since indeed Otto and Villani's construction does not produce the optimal transport map.

\begin{theorem} \label{thm:transport4}
Let $V^{(N)}: M_N(\C)_{sa}^m \to \R$ be a potential satisfying Assumption \ref{ass:convexRMM}.  Let $\mu^{(N)}$ and $X^{(N)}$ be the corresponding law and random variable.  Let $\lambda$ be the limiting free Gibbs law, and let $X \sim \mu$ be an $m$-tuple of non-commutative random variables.  Let $S^{(N)}$ be an independent GUE $m$-tuple and let $S$ be a freely independent free semicircular family.  Then there exist functions $\Phi^{(N)}$, $\Psi^{(N)}: M_N(\C)_{sa}^m \to M_N(\C)_{sa}^m$ and $\Phi, \Psi \in (\overline{\TrP}_m^1)^m$ such that
\begin{enumerate}
	\item $\Phi^{(N)}(X^{(N)}) \sim S^{(N)}$ and $\Psi^{(N)}(S^{(N)}) \sim X^{(N)}$ in law, and similarly, $\Phi(X) \sim S$ and $\Psi(S) \sim X$ in non-commutative law.
	\item $\Phi^{(N)}$ and $\Psi^{(N)}$ are inverse functions of each other, and the same holds for $\Phi$ and $\Psi$.
	\item $\Phi^{(N)} \rightsquigarrow \Phi$ and $\Psi^{(N)} \rightsquigarrow \Psi$.
	\item $\Phi^{(N)}$ is upper triangular in the sense that
	\[
	\Phi^{(N)}(x_1,\dots,x_m) = (\Phi_1^{(N)}(x_1), \Phi_2^{(N)}(x_1,x_2), \dots, \Phi_m^{(N)}(x_1,\dots,x_m))
	\]
	and the same holds for $\Psi^{(N)}$, $\Phi$, and $\Psi$.  In particular, the isomorphism $\mathrm{W}^*(X) \to \mathrm{W}^*(S)$ induced by $\Phi$ maps $\mathrm{W}^*(X_1,\dots,X_k)$ onto $\mathrm{W}^*(S_1,\dots,S_k)$ for each $k = 1$, \dots, $m$.
	\item We have $\norm{\Phi^{(N)} - \id}_{\Lip} \leq m^{1/2} (\max(C,1/c)^3 - 1) \max(C,1/c)^{1/2}$ and $\norm{\Psi^{(N)} - \id}_{\Lip}$ is bounded by some constant $L(c,C,m)$ which goes to zero as $c, C \to 1$.
	\item We have
	\[
	\norm{\Phi^{(N)}(X^{(N)}) - X^{(N)}}_{L^2}^2 = \norm{\Psi^{(N)}(S^{(N)}) - S^{(N)}}_{L^2}^2 \leq 2 h_g^{(N)}(X^{(N)}).
	\]
	and
	\[
	\norm{\Phi(X) - X}_2^2 = \norm{\Psi(S) - S}_2^2 \leq 2 \chi_g^*(X).
	\]
	\item We have
	\begin{align*}
	\norm{\Phi_j(X_1,\dots,X_j) - (X_j - \tau(X_j))}_\infty &= \norm{(\Psi_j(S_1,\dots,S_j) - \tau(\Psi_j(S_1,\dots,S_j))) - S_j}_\infty \\
	&\leq (\max(C,1/c)^3 - 1) \max(C,1/c) \Theta,
	\end{align*}
	where $\Theta$ is the universal constant from Proposition \ref{prop:operatornormestimate}.
\end{enumerate}
\end{theorem}

\begin{proof}
First, by Proposition \ref{prop:marginals}, the marginal law of $(X_1^{(N)}, \dots, X_j^{(N)})$ is given by a convex potential satisfying the same assumptions.

For each $j$, we apply Theorem \ref{thm:transport3} with $X_j^{(N)}$ as the first variable and $(X_1^{(N)}, \dots, X_{j-1}^{(N)})$ as the second variable.  We thus obtain maps $\Phi_j^{(N)}: M_N(\C)_{sa}^j \to M_N(\C)_{sa}$ such that
\[
(X_1^{(N)}, \dots, X_{j-1}^{(N)}, \Phi_j^{(N)}(X_1^{(N)}, \dots, X_j^{(N)})) \sim (X_1^{(N)}, \dots, X_{j-1}^{(N)}, S_j^{(N)}).
\]
Let
\[
\Phi^{(N)}(x_1,\dots,x_m) = (\Phi_1^{(N)}(x_1), \Phi_2^{(N)}(x_1,x_2), \dots \Phi_m^{(N)}(x_1,\dots,x_m)).
\]
Let $Y^{(N)} = \Phi^{(N)}(X^{(N)})$.  Then we can check by backwards induction on $j$ that
\[
(X_1^{(N)}, \dots, X_j^{(N)}, Y_{j+1}^{(N)}, \dots, Y_m^{(N)}) \sim (X_1^{(N)}, \dots, X_j^{(N)}, S_{j+1}^{(N)}, \dots, S_m^{(N)}).
\]
Indeed, the base case $j = m$ is trivial.  For the induction step, suppose the claim holds for $j$.  Since $Y_{j+1}^{(N)}$ is a function of $X_1^{(N)}$, \dots, $X_j^{(N)}$, then the induction hypothesis implies that
\begin{align*}
(X_1^{(N)}, \dots, X_{j-1}^{(N)}, Y_j^{(N)}, Y_{j+1}^{(N)}, \dots, Y_m^{(N)}) &\sim (X_1^{(N)}, \dots, X_{j-1}^{(N)}, Y_j^{(N)}, S_{j+1}^{(N)}, \dots, S_m^{(N)}) \\
&\sim(X_1^{(N)}, \dots, X_{j-1}^{(N)}, S_j^{(N)}, S_{j+1}^{(N)}, \dots, S_m^{(N)}),
\end{align*}
where the last line follows because $(X_1^{(N)}, \dots, X_{j-1}^{(N)}, Y_j^{(N)}) \sim (X_1^{(N)}, \dots, X_{j-1}^{(N)}, S_j^{(N)})$ and because $S_{j+1}^{(N)}$, \dots, $S_m^{(N)}$ are independent of the other variables.  By Theorem \ref{thm:transport3}, $\Phi_j^{(N)}$ is asymptotic to some $\Phi_j \in (\overline{\TrP}_j^1)_{sa}$, and the objects $\Phi$, $X$, and $S$ satisfy the analogous transport relations in the non-commutative setting.  Now because each $\Phi_j^{(N)} - \pi_{x_j}$ is $(\max(C,1/c)^3 - 1) \max(C,1/c)^{1/2}$-Lipschitz, we see that $\Phi^{(N)} - \id$ is $m^{1/2} (\max(C,1/c)^3 - 1) \max(C,1/c)^{1/2}$-Lipschitz.

By Theorem \ref{thm:transport3}, there is a map $G_j^{(N)}: M_N(\C)_{sa}^{m-j+1} \to M_N(\C)_{sa}$ such that $(x_1,\dots,x_{j-1},G_j^{(N)}(x_1,\dots,x_j))$ is the inverse of $(x_1,\dots,x_{j-1},\Phi_j(x_1,\dots,x_j))$.  Define $\Psi_j^{(N)}$ by induction by
\[
\Psi_j^{(N)}(x_j,\dots,x_m) = G_j^{(N)}(\Psi_1(x_1),\dots,\Psi_{j-1}(x_1,\dots,x_{j-1}),x_j).
\]
Then $\Psi^{(N)} = (\Psi_1^{(N)},\dots,\Psi_m^{(N)})$ is the inverse of $\Phi^{(N)}$.  Since $G_j^{(N)} - \id$ is $(\max(C,1/c)^3 - 1) \max(C,1/c)^{1/2}$-Lipschitz, we can show by induction that $\norm{\Psi_j^{(N)}}_{\Lip}$ is bounded by a constant depending only on $c$, $C$, and $m - j$, and which goes to zero as $c, C \to 1$.  Moreover, by Lemma \ref{lem:composition}, $\Psi^{(N)}$ is asymptotic to some $\Psi \in (\overline{\TrP}_m^1)_{sa}^m$.

This concludes the verification of (1) - (5).  Now to prove (6), we apply Theorem \ref{thm:transport3} (5) and get
\begin{align*}
\norm{\Phi^{(N)}(X^{(N)}) - X^{(N)}}_{L^2}^2 &= \sum_{j=1}^m \norm*{\Phi_j^{(N)}(X_j^{(N)},\dots,X_m^{(N)}) - X_j^{(N)}}_{L^2}^2 \\
&\leq \sum_{j=1}^m 2 h_g^{(N)}(X_j^{(N)} | X_1^{(N)}, \dots, X_{j-1}^{(N)}) \\
&= 2 \sum_{j=1}^m \left( h^{(N)}(X_j^{(N)} | X_1^{(N)}, \dots, X_{j-1}^{(N)}) - \frac{1}{2} E \norm{X_j^{(N)}}_2^2 - \frac{1}{2} \log 2 \pi \right) \\
&= 2 \left( h^{(N)}(X^{(N)}) - \frac{1}{2} E \norm{X^{(N)}}_2^2 - \frac{m}{2} \log 2 \pi \right) \\
&= 2 h_g^{(N)}(X^{(N)}),
\end{align*}
where we have applied the definition of $h_g^{(N)}$ and the classical fact that $h^{(N)}$ is additive under conditioning.  As before, because $\Phi^{(N)}(X^{(N)}) \sim S^{(N)}$, we see that $\norm{S^{(N)} - \Psi^{(N)}(S^{(N)})}_{L^2} = \norm{\Phi^{(N)}(X^{(N)}) - X^{(N)}}_{L^2}$.  Finally, the second claim of (6) regarding the free case follows by taking the limit as $N \to \infty$.

Finally, to prove (7), recall that the map $\Phi_j$ is a special case of the map $\tilde{F}_{0,\infty}$ in Theorem \ref{thm:transport2}.  Thus, by applying Theorem \ref{thm:transport2} (4) in the case where $s = \infty$ and $s' = t = 0$, we obtain $\norm{\Phi_j(X_1,\dots,X_j) - (X_j - \tau(X_j))}_\infty \leq (\max(C,1/c)^3 - 1) \max(C,1/c) \Theta$.  Moreover, the middle quantity in claim (7) equals the left hand side because $\Phi(X) \sim S$.
\end{proof}

\subsection*{Funding}

This work was supported by the National Science Foundation [grant DMS-1762360] and the UCLA graduate division.

\subsection*{Acknowledgements}

I thank Dima Shlyakhtenko, Ben Hayes, Brent Nelson, Yoann Dabrowski, Yoshimichi Ueda, and Todd Kemp for various useful conversations and comments on drafts of this paper.  The results of this paper were motivated in part by discussions with Ben Hayes regarding free entropy and maximal amenable subalgebras.  Dima Shlyakhtenko suggested the name ``triangular transport.''   The anonymous referees suggested several references and improvements to the exposition, including the connection with model theory.


\end{document}